\documentclass[article]{smfart}
\usepackage[utf8]{inputenc}
\usepackage[T1]{fontenc}

\usepackage{etex}
\usepackage{hyperref}

\usepackage{style}
\usepackage{baskervald}
\usepackage{caption}
\usepackage[all]{xy}
\usepackage{extarrows}

\usepackage{smfthm}
\author{Ahmed MOUSSAOUI}
\address{Fondation Mathématique  Jacques Hadamard et Université de Versailles St-Quentin-en-Yvelines, Laboratoire de Mathématiques de Versailles, 45 avenue des Etats-Unis, 78035 Versailles Cedex, France}
\email{\href{mailto:ahmed.moussaoui@math.cnrs.fr}{ahmed.moussaoui@math.cnrs.fr}}
\email{\href{mailto:ahmedmoussaouimath@gmail.com}{ahmedmoussaouimath@gmail.com}}
\urladdr{\url{http://ahmed.moussaoui.perso.math.cnrs.fr}}
\title{Centre de Bernstein dual pour les groupes classiques}

\usepackage{array,multirow,makecell}
\usepackage{float}

\usepackage{hhline}
\usepackage{enumitem}
\usepackage{frcursive}
\usepackage[switch,pagewise,running]{lineno}
\usepackage{stmaryrd}

\usepackage{enumitem}
\usepackage{multicol}

\usepackage[usenames,dvipsnames]{xcolor} 

\usepackage{tikz,tikz-3dplot, pgfplots}
\usetikzlibrary[positioning,patterns,cd]

\newtheorem*{defintro}{Définition}
\newtheorem*{conjintro}{Conjecture}
\newtheorem*{thmintro}{Théorème}

\newcommand{\lBrace}{\left\lbrace \!\!\!\!\left\lbrace}
\newcommand{\rBrace}{\right\rbrace \!\!\!\!\right\rbrace}

\def\O{{\rm{O}}}

\usepackage[citestyle=alphabetic,bibstyle=alphabetic,maxnames=4,backend=bibtex]{biblatex}
\DeclareFieldFormat{postnote}{#1}
\DeclareFieldFormat{multipostnote}{#1}

\begin{document}

\frontmatter

\begin{abstract} Dans cet article, nous nous intéressons aux liens entre l'induction parabolique et la correspondance de Langlands locale. Nous énonçons une conjecture concernant les paramètres de Langlands (enrichis) des représentations supercuspidales pour les groupes réductifs $p$-adiques déployés. Nous vérifions la validité de cette conjecture grâce aux résultats connus pour le groupe linéaire et les groupes classiques. À la suite de cela, en définissant la notion de support cuspidal d'un paramètre de Langlands enrichi, nous obtenons une décomposition à la Bernstein de l'ensemble des paramètres de Langlands enrichis pour les groupes classiques. Nous vérifions que ces constructions se correspondent par la correspondance de Langlands et en conséquence, nous obtenons la compatibilité de la correspondance de Langlands avec l'induction parabolique. \end{abstract}
\begin{altabstract} 
In this article, we consider the links between parabolic induction and the local Langlands correspondence. We enunciate a conjecture about the (enhanced) Langlands parameters of supercuspidal representation of split reductives $p$-adics groups. We are able to verify it thanks to the known cases of the local Langlands correspondence for linear groups and classical groups. Furthermore, in the case of classical groups, we can construct the "cuspidal support" of an enhanced Langlands parameter and get a decomposition of the set of enhanced Langlands parameters à la Bernstein. We check that these constructions match under the Langlands correspondence and as consequence, we obtain the compatibility of the Langlands correspondence with parabolic induction.
\end{altabstract}
\subjclass{22E50,20C08,20G10}
\keywords{Correspondance de Langlands, centre de Bernstein stable, correspondance de Springer généralisée} \altkeywords{Langlands correspondence, stable Bernstein center, generalized Springer correspondence}
\thanks{Ce travail est issu de ma thèse dirigée par Anne-Marie Aubert. Je la remercie chaleureusement pour son aide, sa patience, sa gentillesse, ses nombreuses relectures et corrections. Je remercie également Colette M\oe glin pour avoir répondu à mes questions et Thomas Haines pour ses remarques.} 
\maketitle
\tableofcontents 
\mainmatter

\section{Introduction}

Soient $F$ un corps $p$-adique, $G$ le groupe des $F$-points rationnels d'un groupe algébrique réductif connexe défini et déployé sur $F$. L'un des principaux résultats de la théorie du centre de Bernstein est la décomposition de la catégorie des représentations lisses de $G$ en sous-catégories pleines $\Rep(G)_{\mathfrak{s}}$, avec $\mathfrak{s}=[M,\sigma]$ la classe d'équivalence pour une certaine relation d'équivalence (dite inertielle) de $(M,\sigma)$, où $M$ est un sous-groupe de Levi de $G$ et $\sigma$ une représentation irréductible supercuspidale de $M$. On note $\Irr(G)_{\mathfrak{s}}$ l'ensemble des (classes d'isomorphismes de) représentations irréductibles de $G$ admettant $\mathfrak{s}$ pour support inertiel. À toute paire inertielle $\mathfrak{s}$ sont associés un tore $T_{\mathfrak{s}}$, un groupe fini $W_{\mathfrak{s}}$, une action de $W_{\mathfrak{s}}$ sur $T_{\mathfrak{s}}$ et on dispose de la notion de support cuspidal $\Sc : \Irr(G)_{\mathfrak{s}} \longrightarrow T_{\mathfrak{s}}/W_{\mathfrak{s}}$.\\

La correspondance de Langlands fournit conjecturalement une description des représentations irréductibles de $G$.  Notons $\widehat{G}$ le dual de Langlands de $G$, $W_F'=W_F \times \SL_2(\C)$ et $WD_F=\C \rtimes W_F$ les groupes de Weil-Deligne. À tout paramètre de Langlands $\phi : W_F' \longrightarrow \widehat{G}$ est associé un « paquet » de représentations de $G$ noté habituellement $\Pi_{\phi}(G)$. Ce paquet de représentations est paramétré par les représentations irréductibles d'un groupe fini $\mathcal{S}_{\phi}^{G}$, qui est un quotient de $A_{\widehat{G}}(\phi)=Z_{\widehat{G}}(\phi)/Z_{\widehat{G}}(\phi)^{\circ}$, le groupe des composantes du centralisateur de l'image $\phi(W_F')$ dans $\widehat{G}$. La donnée d'un couple formé d'un paramètre de Langlands $\phi$ de $G$ et d'une représentation irréductible de $\mathcal{S}_{\phi}^{G}$ sera appelé paramètre de Langlands enrichi de $G$.\\

Dans cet article, nous nous intéressons aux liens entre ces deux paramétrages. Plus précisément, nous sommes amenés à définir en terme « galoisien » les analogues des objets mis en jeu dans la théorie du centre de Bernstein et à étudier les liens entre induction parabolique et correspondance de Langlands. Nous utiliserons intensivement les résultats de Lusztig sur la correspondance de Springer généralisée.\\

En général, dans un $L$-paquet il y a des représentations de support cuspidaux différents, en particulier il y a parfois des représentations supercuspidales et des représentations non-supercuspidales dans un même $L$-paquet. Une première étape est donc d'identifier quels sont les paramètres de Langlands enrichis des représentations supercuspidales. On définit une notion de paramètre de Langlands cuspidal et on énonce ainsi une conjecture sur le paramétrage des représentations supercuspidales des groupes réductifs déployés : 

\begin{defintro}
Soit $\varphi$ un paramètre de Langlands de $G$. On a une surjection $A_{\widehat{G}}(\varphi) \twoheadrightarrow \mathcal{S}_{\varphi}^{G}$. Si $\varepsilon \in \Irr(\mathcal{S}_{\varphi}^{G})$, on note $\widetilde{\varepsilon}$ la représentation de $A_{\widehat{G}}(\varphi)$ obtenue en tirant en arrière $\varepsilon$. On dit que $\varphi$ est un paramètre cuspidal de $G$ si $\varphi$ est discret et s'il existe une représentation irréductible $\varepsilon$ de $\mathcal{S}_{\varphi}^{G}$ telle que toutes les représentations irréductibles de $A_{Z_{\widehat{G}}(\restriction{\varphi}{W_F})^{\circ}}(\restriction{\varphi}{\SL_2(\C)})$ apparaissant dans la restriction $\restriction{\widetilde{\varepsilon}}{A_{Z_{\widehat{G}}(\restriction{\varphi}{W_F})^{\circ}}(\restriction{\varphi}{\SL_2(\C)})}$, sont cuspidales au sens de Lusztig. On note $\Irr(\mathcal{S}_{\varphi}^{G})_{\cusp}$ l'ensemble des représentations irréductibles $\varepsilon$ vérifiant la condition précédente (il peut être donc être vide).
\end{defintro}

\begin{conjintro}
Soit $\varphi$ un paramètre de Langlands de $G$. Le $L$-paquet $\Pi_{\varphi}(G)$ contient des représentations supercuspidales de $G$, si et seulement si, $\varphi $ est un paramètre de Langlands cuspidal. De plus, si $\varphi$ est un paramètre de Langlands cuspidal, les représentations supercuspidales sont paramétrées par $\Irr(\mathcal{S}_{\varphi}^{G})_{\cusp}$.
\end{conjintro}

Nous prouvons la validité de cette conjecture à l'aide des résultats connus pour la correspondance de Langlands pour $\GL_n$ et pour les groupes classiques d'après les travaux d'Arthur et M\oe glin dans la proposition \ref{verifparamcusp}. De plus, cette conjecture est aussi compatible avec une propriété attendue de la correspondance de Langlands, à savoir qu'un paramètre de Langlands discret $\lambda : W_F \longrightarrow \widehat{G}$, définit un $L$-paquet qui n'est constitué que de représentations supercuspidales (voir la proposition \ref{cocaracusp}).\\

Ayant identifié les paramètres des représentations supercuspidales, la seconde étape consiste à déterminer les paramètres qui devraient correspondre aux sous-quotients irréductibles d'une induite parabolique d'une représentation supercuspidale. Ici, nous sommes guidés par la conjecture de compatibilité entre l'induction parabolique et la correspondance de Langlands. \\

Revenons à la décomposition de Bernstein précédemment évoquée. Fixons une paire inertielle $\mathfrak{s}$ de $G$ et rappelons qu'à une telle paire inertielle est associée un tore $T_{\mathfrak{s}}$, un groupe fini $W_{\mathfrak{s}}$ et une action de $W_{\mathfrak{s}}$ sur $T_{\mathfrak{s}}$. \\

Au début des années 90, Vogan a introduit un analogue pour les paramètres de Langlands du centre
de Bernstein, qu’il a qualifié de centre de Bernstein « stable » et qui a été revisité récemment
par Haines dans \cite{Haines:2014}. Pour résumer très grossièrement, dans \cite[\textsection 5]{Haines:2014}, Haines définit une décomposition à la Bernstein de l'ensemble des paramètres de Langlands, mais ceci n'est pas suffisant si nous souhaitons étudier plus précisément les liens entre l'induction parabolique et la correspondance de Langlands. En effet, il est nécessaire de considérer en plus du paramètre de Langlands $\phi$, la représentation irréductible de $\mathcal{S}_{\phi}^G$. En suivant le même esprit, nous sommes amenés à considérer les triplets $(\widehat{L},\varphi,\varepsilon)$ constitués d'un sous-groupe de Levi $\widehat{L}$ de $\widehat{G}$, d'un paramètre de Langlands cuspidal $\varphi$ de $L$ et d'une représentation irréductible $\varepsilon \in \Irr(\mathcal{S}_{\varphi}^{L})_{\cusp}$. On introduit alors deux relations d'équivalences sur ces triplets : la conjugaison et la conjugaison à un caractère non-ramifié près. On notera respectivement $\Omega_{e}^{\st}(G)$ et $\mathcal{B}_{e}^{\st}(G)$ les classes d'équivalence. À toute classe d'équivalence inertielle $\cj=[\widehat{L},\varphi,\varepsilon] \in \mathcal{B}_{e}^{\st}(G)$, on associe un tore complexe $\mathcal{T}_{\cjp}$ et un groupe fini $\mathcal{W}_{\cjp}$ qui agit sur $\mathcal{T}_{\cjp}$. En supposant la correspondance de Langlands pour les représentations supercuspidales des sous-groupes de Levi de $G$, une classe d'équivalence inertielle $\mathfrak{s}$ devrait correspondre à une classe d'équivalence inertielle $\cj$. Et même, plus précisément, \begin{conjintro} La correspondance de Langlands induit des isomorphismes : 
$$\begin{array}{ccc}
T_{\mathfrak{s}} & \longrightarrow & \mathcal{T}_{\cjp} \\
t & \longmapsto & \widehat{t}
\end{array}, \quad \begin{array}{ccc}
W_{\mathfrak{s}} & \longrightarrow & \mathcal{W}_{\cjp} \\
w & \longmapsto & \widehat{w}
\end{array},$$
tels que pour tout $t\in T_{\mathfrak{s}}, \,\, w \in W_{\mathfrak{s}}$ : $$\widehat{w \cdot t}=\widehat{w} \cdot \widehat{t}.$$
\end{conjintro}

Cette conjecture sera vérifiée pour les groupes classiques dans le théorème \ref{verifcompatibilite} et on obtient :

\begin{thmintro}
Soit $G$ l'un des groupes $\SO_{N}(F)$ ou $\Sp_{N}(F)$. Soient $\mathfrak{s}=[L,\sigma] \in \mathcal{B}(G)$ et $\cj=[\widehat{L},\varphi,\varepsilon] \in \mathcal{B}_{e}^{\st}(G)$ le triplet inertiel de $\widehat{G}$ correspondant par la correspondance de Langlands. Notons $T_{\mathfrak{s}}, \mathcal{T}_{\cjp}, W_{\mathfrak{s}}, \mathcal{W}_{\cjp}$ les tores de Bernstein et groupes de Weyl associés. La correspondance de Langlands induit un isomorphisme  $T_{\mathfrak{s}} \simeq \mathcal{T}_{\cjp}$. De plus, les groupes $W_{\mathfrak{s}}$ et $\mathcal{W}_{\cjp}$ sont isomorphes et l' action de chacun sur $\mathcal{T}_{\cjp}$ et $T_{\mathfrak{s}}$ est compatible avec cet isomorphisme.
\end{thmintro}

On notera $\Phi_{e}(G)$ l'ensemble des paramètres de Langlands enrichis de $G$, c'est-à-dire l'ensemble des couples $(\phi,\eta)$ formés d'un paramètre de Langlands $\phi$ de $G$ et d'une représentation irréductible $\eta$ de $\mathcal{S}_{\phi}^{G}$.\\

À la suite de cela, nous construisons le « support cuspidal partiel » d'un paramètre de Langlands enrichi d'un groupe déployé $G$. Plus précisément, on associe à tout paramètre de Langlands enrichi de $G$, la classe d'un triplet formé d'un sous-groupe de Levi $\widehat{L}$ de $\widehat{G}$, d'un paramètre de Langlands  cuspidal de $L$ et d'une représentation irréductible d'un sous-groupe de $\mathcal{S}_{\varphi}^{L}$. Ceci fait intervenir la correspondance de Springer généralisée (pour des groupes non-connexes) et des constructions de Lusztig. Pour les groupes classiques $\GL_N, \, \SO_N, \, \Sp_{N}$ nous définissons entièrement le support cuspidal de tout paramètre de Langlands enrichi. On obtient ainsi le théorème suivant (voir le théorème \ref{supportcuspidalpartiel}) :

\begin{thmintro}
Il existe une application de support cuspidal surjective $$\cSc : \begin{array}[t]{ccc}
\Phi_{e}(G) & \longrightarrow & \Omega_{e}^{\st}(G) \\
(\phi,\eta)  & \longmapsto & (\widehat{L}, \varphi,\varepsilon)
\end{array}.$$ De plus, les fibres de cette application sont paramétrées par les représentations irréductibles de $N_{Z_{\widehat{G}}(\restriction{\varphi}{W_F} \chi_{c}})(A_{\widehat{L}})/Z_{\widehat{L}}(\restriction{\varphi}{W_F}\chi_{c})$, où $c$ parcourt l'ensemble (fini) des cocaractères de correction de $\varphi$ dans $\widehat{G}$ (voir \ref{cocaractercorrection} pour la définition de cette notion).
\end{thmintro}

On peut ainsi partitionner en blocs l'ensemble des classes de paramètres de Langlands enrichis, indexé par $\mathcal{B}_{e}^{\st}(G)$ : $$\Phi_{e}(G) = \bigsqcup_{\cjp \in \mathcal{B}_{e}^{\st}(G)} \Phi_{e}(G)_{\cjp}.$$ Cette partition permet d'énoncer la conjecture suivante :

\begin{conjintro}
Soit $G$ un groupe réductif $p$-adique connexe déployé. On suppose la correspondance de Langlands satisfaite pour les représentations supercuspidales des sous-groupes de Levi de $G$ ainsi que la conjecture sur le paramétrage des représentations supercuspidales. Soient $\mathfrak{s} \in \mathcal{B}(G)$ et $\cj \in \mathcal{B}_{e}^{\st}(G)$ se correspondant. On a alors une bijection : $$\Irr(G)_{\mathfrak{s}} \simeq \Phi_{e}(G)_{\cjp},$$ telle que les applications de support cuspidal font commuter le diagramme :\\
\begin{center}
\begin{tikzcd}
\Irr(G)_{\mathfrak{s}} \arrow[leftrightarrow]{r} \arrow{d}[swap]{\Sc}  & \Phi_{e}(G)_{\cjp} \arrow{d}{\cScp} \\
T_{\mathfrak{s}}/W_{\mathfrak{s}} \arrow[leftrightarrow]{r} & \mathcal{T}_{\cjp}/\mathcal{W}_{\cjp} 
\end{tikzcd}
\end{center}
\end{conjintro}

Cette conjecture sera prouvée pour les groupes classiques au théorème \ref{theorembijection}. Par ailleurs, Vogan et Haines ont énoncé une conjecture de compatibilité de la correspondance de Langlands avec l'induction parabolique (voir \ref{conjindpar}). Cette conjecture sera prouvée pour les groupes classiques dans le théorème \ref{compatibiliteinductionpreuve}.
\begin{thmintro}
Soit $G$ un groupe classique déployé. Soient $\mathfrak{s}=[L,\sigma] \in \mathcal{B}(G)$ et $\cj=[\widehat{L},\varphi,\varepsilon] \in \mathcal{B}_{e}^{\st}(G)$ correspondant. Notons $\Irr(G)_{\mathfrak{s},2}$ (resp. $\Irr(G)_{\mathfrak{s},\temp}$) l'ensemble des représentations essentiellement discrètes (resp. tempérées) dans $\Irr(G)_{\mathfrak{s}}$ et $\Phi_{e}(G)_{\cjp,2}$ (resp. $\Phi_{e}(G)_{\cjp,\bdd}$) l'ensemble des paramètres de Langlands enrichis discrets (resp. tels que la restriction à $W_F$ est bornée) dans $\Phi_{e}(G)_{\cjp}$. On a une bijection $$\Irr(G)_{\mathfrak{s}} \leftrightarrow \Phi_{e}(G)_{\cjp},$$ induisant des bijections $$\Irr(G)_{\mathfrak{s},2} \leftrightarrow\Phi_{e}(G)_{\cjp,2},$$ et $$\Irr(G)_{\mathfrak{s},\temp} \leftrightarrow \Phi_{e}(G)_{\cjp,\bdd}.$$ En conséquence, la correspondance de Langlands est compatible avec l'induction parabolique.
\end{thmintro}

Ayant toutes ces constructions en main, on se propose de comparer nos résultats à ceux obtenus par M\oe glin et M\oe glin-Tadi{{\'c}} pour les représentations de la série discrète. Nous obtenons une description explicite du support cuspidal d'un paramètre de Langlands enrichi discret et que nos résultats coïncident avec ceux de M\oe glin et M\oe glin-Tadi{{\'c}} (voir les propositions \ref{suppcuspidalexplicite}, \ref{comparaisonsupportcuspidalmoeglin} et le théorème \ref{suppcuspicompatible}).
\begin{thmintro}
Soient $G$ un groupe classique, $\phi \in \Phi(G)_2$ un paramètre de Langlands discret d'un groupe classique $G$, $N$ le rang semi-simple de $\widehat{G}$ et $\eta \in \Irr(A_{\widehat{G}}(\phi))$. On note $\Std_G : \widehat{G} \hookrightarrow \GL_N(\C)$ le plongement standard de $\widehat{G}$ et on décompose $\Std_G \circ \phi = \bigoplus_{\pi \in I} \pi \boxtimes S_{\pi}$, où $I$ est l'ensemble des représentations irréductibles de $W_F$ apparaissant dans $\restriction{\phi}{W_F}$ et $S_{\pi}$ une représentation de $\SL_2(\C)$. Pour tout $p \in \N^*$, on note $S_p$ la représentation irréductible de dimension $p$ de $\SL_2(\C)$. Pour tout $\pi \in I$, on peut décomposer $S_{\pi}=\bigoplus_{i=1}^{k_{\pi}} S_{p_{\pi,i}}$ avec $k_{\pi}$ et $p_{\pi,i}$ des entiers naturels. L'image de la représentation $S_{\pi}$ est incluse dans un groupe orthogonal ou un groupe symplectique qu'on notera $\widehat{G}_{\pi}$. On a un isomorphisme $A_{\widehat{G}}(\phi) \simeq \prod_{\pi \in I} A_{\widehat{G}_{\pi}}(S_{\pi})$ et pour tout $\pi \in I, \,\, A_{\widehat{G}_{\pi}}(S_{\pi})=\prod_{i=1}^{k_{\pi}} \langle z_{p_{\pi,i}} \rangle$ avec $z_{p_{\pi,i}}$ d'ordre $2$. Écrivons $\eta \simeq \boxtimes_{\pi \in I} \eta_{\pi}$.\\
Pour tout $\pi \in I$, on pose : $$n_{\pi}=\dim \pi, \,\, m_{\pi}=\dim S_{\pi}, \,\, d_{\pi}'=\left\{ \begin{array}{ll}
 \sum_{i=1}^{k_{\pi}} (-1)^{i+k_{\pi}} \eta_{\pi}(z_{p_{\pi,i}}) +2k_{\pi}+2-2 \left[ \frac{k_\pi+1}{2}\right] & \text{si $\widehat{G}_{\pi}$ est symplectique} \\
 \sum_{i=1}^{k_{\pi}} (-1)^{i+1} \eta_{\pi}(z_{p_{\pi,i}}) & \text{si $\widehat{G}_{\pi}$ est orthogonal} \\
 \end{array} \right. ,$$ $$\ell_{\pi}=\left\{ \begin{array}{ll}
 \frac{m_{\pi}-d_{\pi}'(d_{\pi}'-1)}{2} & \text{si $\widehat{G}_{\pi}$ est symplectique} \\
 \frac{m_{\pi}-d_{\pi}'^2}{2} & \text{si $\widehat{G}_{\pi}$ est orthogonal} \\
 \end{array} \right. \quad \text{et} \quad N^{\sharp}=\sum_{\substack{\pi \in I \\ \widehat{G}_{\pi} \;\; \text{symp.}}} n_{\pi} d_{\pi}'(d_{\pi}'-1) + \sum_{\substack{\pi \in I \\ \widehat{G}_{\pi} \;\; \text{orth.}}} n_{\pi} d_{\pi}'^2.$$  Si $\widehat{G}_{\pi}$ est symplectique, soit $d_{\pi} \in \N$ l'unique entier naturel tel que $(d_{\pi}+1)d_{\pi}=d_{\pi}'(d_{\pi}'-1)$ et si $\widehat{G}_{\pi}$ est orthogonal, soit $d_{\pi}= |d_{\pi}'|$. Considérons les multiensembles définis par $$E_{c,\pi}=\left\{ \begin{array}{ll}
\bigcup_{i=1}^{k_{\pi}} \lBrace \frac{p_{\pi,i}-1}{2}-j,j \in \llbracket 0, p_{\pi,i}-1 \rrbracket \rBrace - \bigcup_{i=1}^{d_{\pi}} \lBrace \frac{2i-1}{2}-j,j \in \llbracket 0, 2i-1 \rrbracket \rBrace, & \text{si $\widehat{G}_{\pi}$ est symplectique} \\
\bigcup_{i=1}^{k_{\pi}} \lBrace \frac{p_{\pi,i}-1}{2}-j,j \in \llbracket 0, p_{\pi,i}-1 \rrbracket \rBrace - \bigcup_{i=1}^{d_{\pi}} \lBrace \frac{2i-2}{2}-j,j \in \llbracket 0, 2i-2 \rrbracket \rBrace, & \text{si $\widehat{G}_{\pi}$ est orthogonal} \\
 \end{array} \right.$$ et notons $E_{c,\pi}'$ le sous-multiensemble de $E_{c,\pi}$ constitué des éléments positifs si $\widehat{G}_{\pi}$ est symplectique et constitué des éléments strictement positifs et de $0$ compté avec multiplicité $\frac{m_{\pi}-d_{\pi}}{2}$ si $\widehat{G}_{\pi}$ est orthogonal, de sorte que $E_{c,\pi}=E_{c,\pi}' \sqcup -E_{c,\pi}'$.\\
 
Le support cuspidal de $(\phi,\eta)$ que nous avons défini au théorème \ref{theoremesupportcuspidal} est $(\widehat{L},\varphi,\varepsilon)$ avec : \begin{itemize}
\item $\displaystyle \widehat{L}=\prod_{\pi \in I} \GL_{n_{\pi}}(\C)^{\ell_{\pi}} \times \widehat{G}'$, avec $\widehat{G^{\sharp}}$ un groupe classique de rang $N^{\sharp}$ de même type que $\widehat{G}$ ;
\item $\displaystyle \varphi=\left(\bigoplus_{\pi \in I} \bigoplus_{e \in E_{c,\pi}'} |\cdot|^{e} \pi \oplus \left( |\cdot|^{e} \pi \right)^{\vee} \right)\oplus \left( \bigoplus_{\substack{\pi \in I \\ \widehat{G}_{\pi} \;\; \text{symp.}}} \bigoplus_{a=1}^{d_{\pi}} \pi \boxtimes S_{2a} \oplus \bigoplus_{\substack{\pi \in I \\ \widehat{G}_{\pi} \;\; \text{orth.}}} \bigoplus_{a=1}^{d_{\pi}} \pi \boxtimes S_{2a-1} \right)$;
\item $\varepsilon \simeq \boxtimes_{\substack{\pi \in I \\ \widehat{G}_{\pi} \;\; \text{symp.}}} \varepsilon_{d_{\pi}(d_{\pi}+1)}^{\S} \boxtimes \boxtimes_{\substack{\pi \in I \\ \widehat{G}_{\pi} \;\; \text{orth.}}} \varepsilon_{d_{\pi}^2}^{\O \pm}$.
\end{itemize}
De plus, le support cuspidal des séries discrètes défini par M\oe glin et M\oe glin-Tadi{{\'c}} coïncide avec le nôtre.
\end{thmintro}

Décrivons l'organisation de l'article.\\

La section $2$ sert de préliminaire. Nous rappelons les définitions et donnons des exemples de la correspondance de Springer généralisée. Après avoir énoncé la conjecture locale de Langlands, nous rappelons le théorème de M\oe glin concernant le paramétrage des représentations supercuspidales des groupes classiques. Dans la suite, nous décrivons brièvement la théorie du centre de Bernstein et celle du centre de Bernstein stable par Haines.\\

Dans la section $3$ nous énonçons notre conjecture sur le paramétrage des représentations supercuspidales puis nous vérifions qu'elle est compatible avec les cas connus et certaines propriétés attendues. La suite est consacrée à la définition du support cuspidal d'un paramètre de Langlands enrichi de façon « intrinsèque ».\\

Dans la section $4$ nous faisons les liens entre nos constructions précédentes du coté des paramètres de Langlands et ce qui devrait correspondre en terme de représentations. En particulier, nous montrons que les blocs de paramètres de Langlands enrichis que nous avons définis sont en bijection avec les blocs de Bernstein. En conséquence, nous obtenons la compatibilité de la correspondance de Langlands avec l'induction parabolique. En fin de section, on s'intéresse à la généricité et à la question du changement de paramétrage dans un $L$-paquet.\\

Enfin, dans la section $5$, nous donnons pour les séries discrètes des formules pour calculer explicitement le support cuspidal qu'on a défini et nous vérifions que notre application de support cuspidal est compatible avec la correspondance de Langlands (voir théorème \ref{suppcuspicompatible}).\\

Dans l'appendice nous étendons la correspondance de Springer généralisée au groupe orthogonal ainsi qu'à un certain sous-groupe d'un produit de groupes orthogonaux.

\section{Préliminaires}

\subsection{Correspondance de Springer généralisée}
Soient $H$ un groupe algébrique linéaire complexe réductif et $x \in H$ un élément de $H$. On note $Z_H(x)=\{ g \in H \mid gxg^{-1}=x\}$ le centralisateur de $x$ dans $H$ et $A_H(x)=Z_H(x)/Z_H(x)^{\circ}$ le groupe des composantes du centralisateur de $x$ de $H$. Supposons désormais que $H$ est connexe. Il y a une bijection naturelle entre les représentations irréductibles de $A_H(x)$ et les systèmes locaux $H$-équivariants irréductibles sur $\mathcal{C}_x^H$, où $\mathcal{C}_x^H=\{gxg^{-1}, \, g \in H \}$ désigne la $H$-classe de conjugaison de $x$ (voir \cite[12.10]{Jantzen:2004}). Dans la suite, on s'intéressera aux orbites unipotentes de $H$ (ou aux orbites nilpotentes de l'algèbre de Lie de $H$), c'est à dire aux $H$-classes de conjugaison d'éléments unipotents de $H$. On note $\mathcal{N}_H^{+}$ le cône unipotent « complet », c'est à dire l'ensemble des $H$-classes de conjugaison des couples formés d'un élément unipotent $u$ de $H$ et d'une représentation irréductible du groupe des composantes du centralisateur dans $H$ de $u$ :
\begin{align*}
\mathcal{N}_H^{+} &= \left\{(u,\eta) \mid u \in H \,\, \mbox{unipotent} ,\,\, \eta \in \Irr(A_H(u)) \right\}_{ / H-\rm{conj}}  \\
 &\simeq \left\{(\mathcal{C}_{u}^{H},\mathcal{F}) \mid u \in H \,\, \mbox{unipotent} ,\,\, \mathcal{F} \,\, \text{système local $H$-équivariant irréductible sur $\mathcal{C}_u^{H}$} \right\}.
\end{align*}

\begin{defi}[{\cite[6.2]{Lusztig:1984}}]\phantomsection\label{defsyslocusp}
Soit $u \in H$ un élément unipotent. Un système local $H$-équivariant irréductible $\mathcal{L}$ sur $\mathcal{C}_{u}^{H}$ est dit \emph{cuspidal} lorsque pour tout sous-groupe parabolique propre $P$ de $H$, de radical unipotent $U$ et pour tout élément unipotent $g \in P$, la cohomologie à support compact de $gU \cap \mathcal{C}_{u}^{H}$ à coefficient dans la restriction de $\mathcal{L}$ est nulle.
\end{defi}
Décrivons une définition équivalente en terme de représentations de $A_H(u)$. Soit $P$ un sous-groupe parabolique de $H$ de décomposition de Levi $P=LU$ et soient $u\in H, v \in L$ des éléments unipotents. On note $$Y_{P,u,v}=\left\{g Z_{L}(v)^{\circ}U \in H/Z_{L}(v)^{\circ}U  \mid g \in H, g^{-1}ug \in vU \right\} \quad \text{et} \quad d_{P,u,v}=\frac{1}{2} \left( \dim Z_{H}(u) -\dim Z_{L}(v) \right).$$ D'après Springer et Lusztig \cite[\textsection 1]{Lusztig:1984}, $\dim Y_{P,u,v} \leqslant d_{P,u,v}$ et le groupe $Z_H (u)$ agit par translation à gauche sur $Y_{P,u,v}$. L'action de $Z_H(u)$ se factorise en une action de $A_H(u)=Z_H(u)/Z_H(u)^{\circ}$ et on note $S_{P,u,v}$ la représentation de $A_H(u)$ par permutation sur les composantes irréductibles de dimension $d_{P,u,v}$ de $Y_{P,u,v}$. 

\begin{defi}\phantomsection\label{suv} 
Soient $u \in H$ unipotent et $\varepsilon\in \Irr(A_{H}(u))$. On dit que $\varepsilon$ est \emph{cuspidale} lorsque pour tout sous-groupe parabolique propre $P=LU$ de $H$, pour tout élément unipotent $v \in L$, $$\Hom_{A_H(u)}(\tau,S_{P,u,v})=0.$$
\end{defi}

Si $\mathcal{L}$ est un système local $H$-équivariant irréductible cuspidal sur $\mathcal{C}_{u}^{H}$, alors la représentation irréductible du groupe fini $A_{H}(u)$ est cuspidale dans le sens précédent et réciproquement. On appelera paire cuspidale de $H$, tout couple $(u,\varepsilon)$ (ou $(\mathcal{C}_{u}^{H},\mathcal{L})$) formé d'un élément unipotent $u \in H$ et d'une représentation irréductible cuspidale de $A_{H}(u)$ (ou d'un système local $H$-équivariant irréductible cuspidal).\\
Notons $\mathcal{S}_H$ l'ensemble des classes (de conjugaison par $H$) de triplets $\mathfrak{t}=[L,\mathcal{C}_v^{L},\mathcal{L}]$ formés de \begin{itemize}
\item un sous-groupe de Levi $L$ de $H$ ;
\item une $L$-orbite $\mathcal{C}_v^L$ d'un élément unipotent $v \in L$ ;
\item un système local $L$-équivariant irréductible cuspidal $\mathcal{L}$ sur $\mathcal{C}_v^L$.
\end{itemize}%
Dans \cite[6.3]{Lusztig:1984}, Lusztig associe à chaque couple $(\mathcal{C}_{u}^{H},\mathcal{F}) \in \mathcal{N}_H^{+}$ un unique triplet $\mathfrak{t} \in \mathcal{S}_H$. Notons $\Psi_{H} : \mathcal{N}_H^{+} \longrightarrow \mathcal{S}_H$ cette application. Elle induit alors une partition de $\mathcal{N}_H^{+}$ : $$\mathcal{N}_H^{+}=\bigsqcup_{\mathfrak{t} \in \mathcal{S}_H} \mathcal{M}_{\mathfrak{t}},$$ où $\mathcal{M}_{\mathfrak{t}}=\Psi_{H}^{-1}(\mathfrak{t})$. Pour tout $\mathfrak{t}=[L,\mathcal{C},\mathcal{L}] \in \mathcal{S}_H$, soient $S=Z_{L}^{\circ} \cdot \mathcal{C}$ et $\mathcal{E}$ le système local $L$-équivariant sur $S$ obtenu en tirant en arrière $\mathcal{L}$ via la projection $Z_{L}^{\circ} \cdot \mathcal{C} \longrightarrow \mathcal{C}$. Notons $$W_{\mathfrak{t}}=\left\{ n \in N_{H}(L) \mid nSn^{-1}=S, \Ad(n^{-1})^{*} \mathcal{E} \simeq \mathcal{E} \right\}/L.$$ D'après \cite[9.2]{Lusztig:1984}, $$W_{\mathfrak{t}} \simeq N_{H}(L)/L,$$ et c'est un groupe de Coxeter fini.

\begin{rema}\phantomsection\label{interpretationW}
Soient $v \in \mathcal{C}$ et $\gamma : \SL_2(\C) \longrightarrow L$ un morphisme de groupes algébriques tel que $\gamma \left( \begin{smallmatrix}
1 & 1 \\  0 & 1 \end{smallmatrix} \right) = v$. Notons $T=Z_{L}^{\circ}$ et $Z=Z_{H}(\gamma)^{\circ}$. D'après \cite[2.6]{Lusztig:1988}, $T$ est un tore maximal de $Z$ et on a l'isomorphisme suivant : $$W_{\mathfrak{t}} \simeq N_{H}(L)/L \simeq N_{Z}(T)/T.$$ 
\end{rema}

Rappelons l'ordre partiel sur l'ensemble des orbites unipotentes de $H$ : si $\mathcal{O}$ et $\mathcal{O}'$ sont deux orbites unipotentes de $H$, $$\mathcal{O}\leqslant \mathcal{O}' \Leftrightarrow \mathcal{O} \subseteq \overline{\mathcal{O}'},$$ où $\overline{\mathcal{O}'}$ désigne l'adhérence de $\mathcal{O}'$. Soit $\mathfrak{t}=[L,\mathcal{C},\mathcal{L}] \in \mathcal{S}_H$. D'après \cite[6.5]{Lusztig:1984} et \cite[6.5]{Lusztig:1995a}, pour tout $(\mathcal{O},\mathcal{F}) \in \mathcal{M}_{\mathfrak{t}}$, on a : $$H \cdot \mathcal{C} \leqslant \mathcal{O} \leqslant \Ind_L^H(\mathcal{C}),$$ où $\Ind_L^H(\mathcal{C})$ est l'orbite obtenue par induction (de Lusztig-Spalteinstein) $\mathcal{C}$ de $L$ à $H$ (voir \cite{Lusztig:1979}). Notons $\mathcal{C}_{\mathfrak{t}}^{\min}=H \cdot \mathcal{C}$ et $\mathcal{C}_{\mathfrak{t}}^{\max}=\Ind_L^H(\mathcal{C})$. Chacune de ces orbites supporte un système local $H$-équivariant irréductible particulier défini en \cite[9.2,9.5]{Lusztig:1984} qu'on note $\mathcal{L}_{\mathfrak{t}}^{\min}$ pour $\mathcal{C}_{\mathfrak{t}}^{\min}$ et $\mathcal{L}_{\mathfrak{t}}^{\max}$ pour $\mathcal{C}_{\mathfrak{t}}^{\max}$. À présent, on peut énoncer la correspondance de Springer généralisée due à Lusztig.
\begin{theo}[Lusztig,{\cite[6.5]{Lusztig:1984}}]
Soit $H$ un groupe réductif connexe complexe. L'application $\Psi_H : \mathcal{N}_H^+ \longrightarrow \mathcal{S}_H$ est surjective et pour tout $\mathfrak{t} \in \mathcal{S}_H$, on a unique une bijection $$\Sigma_{\mathfrak{t}} : \mathcal{M}_{\mathfrak{t}} \longrightarrow \Irr(W_\mathfrak{t}),$$ telle que : $$\Sigma_{\mathfrak{t}}(\mathcal{C}_{\mathfrak{t}}^{\min},\mathcal{L}_{\mathfrak{t}}^{\min})=\mathrm{triv} \quad\mbox{et} \quad \Sigma_{\mathfrak{t}}(\mathcal{C}_{\mathfrak{t}}^{\max},\mathcal{L}_{\mathfrak{t}}^{\max})=\mathrm{sgn}.$$ 
On obtient ainsi une bijection : $$\mathcal{N}_H^{+}  \leftrightarrow \bigsqcup_{\mathfrak{t} \in \mathcal{S}_H} \Irr(W_\mathfrak{t}).$$
\end{theo}

\begin{rema}\phantomsection\label{remarquethm92}
Dans \cite{Lusztig:1984}, la correspondance de Springer généralisée est telle que $\Sigma_{\mathfrak{t}}(\mathcal{C}_{\mathfrak{t}}^{\min},\mathcal{L}_{\mathfrak{t}}^{\min})=\mathrm{sgn}$ et $\Sigma_{\mathfrak{t}}(\mathcal{C}_{\mathfrak{t}}^{\max},\mathcal{L}_{\mathfrak{t}}^{\max})=\mathrm{triv}$ (voir \cite[9.2,9.5]{Lusztig:1984}). La normalisation que l'on a choisie est donc la tensorisation par le caractère signature de la correspondance de Springer définie par Lusztig.
\end{rema}

\subsection{Orbites unipotentes et paires cuspidales pour les groupes classiques} On se propose de rappeler brièvement la classification des orbites unipotentes et de leurs centralisateurs dans certains groupes classiques (voir \cite[\textsection 5.1 \& \textsection 6.1]{Collingwood:1993} et \cite[\textsection 3.8]{Jantzen:2004}).\\

On appelle partition $\mathbf{p}$ d'un entier $N \geqslant 1$ une suite croissante d'entiers $1 \leqslant p_1 \leqslant \ldots \leqslant p_k $ telle que $N=p_1+\ldots+p_k$.  A priori, les entiers $p_i$ ne sont pas distincts deux à deux. Soient $q_1 < \ldots < q_s $ les entiers deux à deux distincts tels que $\{p_i, \, 1 \leqslant i \leqslant k \}=\{ q_j , \, 1 \leqslant j \leqslant s\}$ et $r_{q}$ le nombre de fois que $q$ apparait dans $\mathbf{p}$. On notera $\mathbf{p}=(q_1^{r_{q_1}},\ldots,q_s^{r_{q_s}})$, si bien que, $n=p_1+\ldots+p_k=r_{q_1} q_1 + \ldots + r_{q_s} q_s$. Les $q_j$ (ou $p_i$) s'appellent les parts de la partition $\mathbf{p}$ et $r_{q_j}$ la multiplicité de la part $q_j$.\\

Pour toute partition $\mathbf{p}$, on note $\mathcal{O}_{\mathbf{p}}$ l'orbite unipotente associé à la partition $\mathbf{p}$ par la décomposition de Jordan. Dans un cas, il correspondra à une même partition $\mathbf{p}$ deux orbites unipotentes distinctes que l'on notera $\mathcal{O}_{\mathbf{p}}^{I}$ et $\mathcal{O}_{\mathbf{p}}^{II}$. \begin{itemize}
\item Pour $\GL_{N}(\C)$, les orbites unipotentes sont en bijection avec les partitions de $N$ via la décomposition de Jordan.
\item Pour $\Sp_{N}(\C)$, les orbites unipotentes sont en bijection avec les partitions de $N$ pour lesquelles les parts impaires admettent une multiplicité paire.
\item Pour $\SO_{N}(\C)$, les partitions de $N$ pour lesquelles les parts paires admettent une multiplicité paire et toutes les parts ne sont pas paires correspondent à une orbite unipotente.  Les partitions de $N$ qui n'admettent que des parts paires, de multiplicités paires correspondent à deux orbites unipotentes distinctes. Ce dernier cas ne se produit que si $N$ est pair.
\end{itemize}

La table suivante résume la description des orbites unipotentes et des centralisateurs correspondants:
$$\renewcommand{\arraystretch}{1.3}
\begin{array}{| >{\displaystyle}c | >{\displaystyle}c | >{\displaystyle}c | >{\displaystyle}c |}
\hline
H  & \mathrm{partition} & \mathrm{orbite} & \mathrm{centralisateur} \\ 
\hline
\hline 
\GL_{N}(\C)& &\mathbf{p} \leftrightarrow \mathcal{O}_{\mathbf{p}} & Z_{\GL_N(\C)}(u)_{\red} \simeq \prod_{q \in \mathbf{p}} \GL_{r_q}(\C)  \\ \hline
\Sp_{N}(\C)&\forall q \in \mathbf{p},  q \,\, \mathrm{impair} \Rightarrow r_q \,\, \mathrm{pair} & \mathbf{p} \leftrightarrow \mathcal{O}_{\mathbf{p}}  & Z_{\Sp_{N}(\C)}(u)_{\red} \simeq \prod_{\substack{q \in \mathbf{p} \\ q \,\, \mathrm{pair}}} \O_{r_q}(\C)  \times \prod_{\substack{q \in \mathbf{p} \\ q \,\, \mathrm{impair}}} \Sp_{r_q}(\C)  \\ \hline
\multirow{2}*{$\SO_{N}(\C)$} & \forall q \in \mathbf{p},  q \,\, \mathrm{pair} \Rightarrow r_{q} \,\, \mathrm{pair} & \mathbf{p} \leftrightarrow \mathcal{O}_{\mathbf{p}}  & Z_{\O_{N}(\C)}(u)_{\red} \simeq \prod_{{\substack{q \in \mathbf{p} \\ q \,\, \mathrm{impair}}}} \O_{r_q}(\C)  \times \prod_{{\substack{q \in \mathbf{p} \\ q \,\, \mathrm{pair}}}} \Sp_{r_q}(\C)  \\ \cline{2-3}
& \forall q \in \mathbf{p}, \,\, q \,\,\mathrm{et}\,\,r_q\,\,\mathrm{pair} & \mathbf{p} \leftrightarrow \left\{ \begin{array}{l}
\mathcal{O}_{\mathbf{p}}^{I}\\
\mathcal{O}_{\mathbf{p}}^{II}
\end{array} \right.   & Z_{\SO_{N}(\C)}(u)_{\red} \simeq S\left(\prod_{{\substack{q \in \mathbf{p} \\ q \,\, \mathrm{impair}}}} \O_{r_q}(\C) \right) \times \prod_{{\substack{q \in \mathbf{p} \\ q \,\, \mathrm{pair}}}} \Sp_{r_q}(\C) \\ \hline
\end{array}$$\captionof{table}{Orbites unipotentes des groupes classiques\\} 

Où l'on a noté $Z_{H}(u)_{\red}$ le quotient de $Z_{H}(u)$ par son radical unipotent et : $$S\left(\prod_{{\substack{q \in \mathbf{p} \\ q \,\, \mathrm{impair}}}} \O_{r_q}(\C) \right)=\left\{ (x_q)\in \prod_{{\substack{q \in \mathbf{p} \\ q \,\, \mathrm{impair}}}} \O_{r_q}(\C) \Biggm| \prod_{{\substack{q \in \mathbf{p} \\ q \,\, \mathrm{impair}}}} \det(x_q)=1 \right\}.$$

Les paires cuspidales sont assez rares. Elles apparaissent comme les blocs fondamentaux dans la correspondance de Springer généralisée. Dans \cite{Lusztig:1984}, Lusztig classifie complètement les paires cuspidales pour les groupes presque simples et simplement connexes. Plus précisément, dans \cite[\textsection 10, 12.4 \& 13.4]{Lusztig:1984} et \cite{Lusztig:2014}, il classifie les paires cuspidales pour les groupes classiques. Nous résumons les résultats dans le tableau ci-dessous.

$${\renewcommand{\arraystretch}{1.3}
\begin{array}{|c|c|c|c|c|}
\hline
H & \text{condition} & \text{orbite unipotente} & A_{H}(u) & \text{paire cuspidale} \;\; (\mathcal{C},\varepsilon) \\ 
\hline 
\hline 
\GL_{N}(\C) & N=1 & \mathcal{O}_{(1)} & \{1\} & (\mathcal{O}_{(1)},1) \\ 
\hline 
\Sp_{N}(\C) & N=d(d+1)  & \mathcal{O}_{(2,4,\ldots,2d-2,2d)} & (\Z/2\Z)^{d} & (\mathcal{O}_{(2,4,\ldots,2d-2,2d)},\varepsilon_{N}^{\S}) \\ 
\hline 
\SO_{N}(\C) & N=d^2  & \mathcal{O}_{(1,3,\ldots,2d-3,2d-1)} & (\Z/2\Z)^{d-1} & (\mathcal{O}_{(1,3,\ldots,2d-3,2d-1)},\varepsilon_{N}^{\O}) \\ 
\hline 
\end{array}}$$
\captionof{table}{Paires cuspidales des groupes classiques}\phantomsection\label{pairecusp}

Précisons les notations employées. Dans les cas des groupes symplectiques et orthogonaux, les orbites unipotentes qui interviennent dans les paires cuspidales sont paramétrées par des partitions dont les parts sont de multiplicité $1$. Par suite, le groupe des composantes est un produit de $\Z/2\Z$ indexé par les parts paires (resp. impaires) pour le cas symplectique (resp. orthogonal). \\ Dans le cas symplectique, soit $u \in \mathcal{O}_{(2,4,\ldots,2d-2,2d)}$ et pour tout $a \in \{2,\ldots,2d\}$, notons $z_a \in A_{\Sp_{N}(\C)}(u)$ tels que : $$A_{\Sp_{N}(\C)}(u) = \prod_{i=1}^{d} \langle z_{2i} \rangle \simeq (\Z / 2\Z)^{d}. $$ La représentation cuspidale $\varepsilon_{N}^{\S}$ de $A_{\Sp_{N}(\C)}(u)$ vérifie : pour tout $i \in \llbracket 1,d \rrbracket $, $$\varepsilon_{N}^{\S}(z_{2i})=(-1)^{i}.$$ Dans le cas orthogonal, soit $u \in \mathcal{O}_{(1,3,\ldots,2d-3,2d-1)}$ et pour tout $a \in \{1,\ldots,2d-1\}$, notons $z_a \in A_{\O_{N}(\C)}(u)$ tels que : $$
A_{\O_{N}(\C)}(u) = \prod_{i=1}^{d} \langle z_{2i-1} \rangle \simeq (\Z/ 2\Z)^{d} \quad \text{et} \quad A_{\SO_{N}(\C)}(u)  = \prod_{i=1}^{d-1} \langle z_{2i-1}z_{2i+1} \rangle \simeq (\Z/ 2\Z)^{d-1}.$$ La représentation cuspidale $\varepsilon_{N}^{\O}$ de $A_{\SO_{N}(\C)}(u)$ vérifie : pour tout $i \in \llbracket 1,d-1 \rrbracket $, $$\varepsilon_{N}^{\O}(z_{2i-1}z_{2i+1})=-1.$$ Les représentations irréductibles ${\varepsilon_{N}^{\O +}}$ et ${\varepsilon_{N}^{\O -}}$ de $A_{\O_{N}(\C)}(u)$  telles que leurs restrictions à $A_{\SO_{N}(\C)}(u)$ est $\varepsilon_{N}^{\O}$ vérifient : pour tout  $i \in \llbracket 1,d \rrbracket $, $${\varepsilon_{N}^{\O +}}(z_{2i-1})=(-1)^{i+1} \quad \text{et} \quad {\varepsilon_{N}^{\O -}} (z_{2i-1})=(-1)^{i}.$$ 

\subsection{Correspondance de Springer tordue}\phantomsection\label{springertordue}

Rappelons brièvement la définition des algèbres de Hecke graduées associées à des triplets cuspidaux et le théorème de classification qui s'ensuit.\\

Soient $H$ un groupe réductif connexe complexe, $X$ une variété algébrique sur laquelle $H$ agit algébriquement et $\mathcal{E}$ un système local $H$-équivariant sur $X$. On considère les espaces de cohomologies équivariantes $H^{j}_{H}(X,\mathcal{E})$ et d'homologies équivariantes $H^{H}_{j}(X,\mathcal{E})$ que Lusztig considère dans \cite[\textsection 1]{Lusztig:1988}. On notera $H^{\bullet}_{H}(X,\mathcal{E})= \bigoplus_{j \in \Z} H^{j}_{H}(X,\mathcal{E})$, $H^{H}_{\bullet}(X,\mathcal{E})= \bigoplus_{j \in \Z} H^{H}_{j} (X,\mathcal{E})$. Pour le système local constant, on écrira plus simplement $H^{\bullet}_{H}(X)=H^{\bullet}_{H}(X,\C)$, $H^{H}_{\bullet}(X)=H^{H}_{\bullet}(X,\C)$ et pour tout $\mathrm{point} \in X$, $H^{\bullet}_{H}=H^{\bullet}_{H}(\{\mathrm{point}\})$ et $H^{H}_{\bullet}=H^{H}_{\bullet}(\{\mathrm{point}\})$. \\

Soient $P=LU$ un sous-groupe parabolique de $H$ et $\mathfrak{p}=\mathfrak{l} \oplus \mathfrak{u}$ son algèbre de Lie. Soit $\mathcal{C} \subset \mathfrak{l}$ une $L$-orbite nilpotente supportant un système local irréductible cuspidal $L$-équivariant noté $\mathcal{L}$. Notons $\mathfrak{t}=[L,\mathcal{C},\mathcal{L}] \in \mathcal{S}_H$ le triplet cuspidal correspondant et $W_L^H=N_H(L)/L$ le groupe de Weyl relatif associé.\\

Le groupe $H \times \C^{\times}$ agit sur l'algèbre de Lie $\mathfrak{h}$ de $H$ par : $$\forall (h,t) \in H \times \C^{\times}, \forall x \in \mathfrak{h}, \;\; (h,t) \cdot x = t^{-2} \Ad(h) x.$$ On note $M_H(x)=\{(h,t) \in H \times \C^{\times} \mid \Ad(h)x=t^2 x \}$ le stabilisateur de $x \in \mathfrak{h}$ pour cette action. D'après \cite[2.1.f)]{Lusztig:1988}, on a un isomorphisme $M_H(x)/M_H(x)^{\circ} \simeq Z_H(x) / Z_H(x)^{\circ} =A_H(x)$.\\

Soit $x \in \mathfrak{h}$ un élément nilpotent et considérons la variété $\mathcal{P}_x=\{hP \in H/P \mid \Ad(h^{-1})x \in \mathcal{C}+\mathfrak{u} \}$. Le groupe $M_H(x)$ agit sur $\mathcal{P}_x$ par $(h,t) \cdot (h'P) =(hh')P$. Soit $\dot{\mathcal{L}}$ le système local $M_H(x)^{\circ}$-équivariant sur $\mathcal{P}_{x}$ obtenu à partir de $\mathcal{L}$ et défini en \cite[3.4,8.1.e)]{Lusztig:1988}. On note $\mathcal{Z}$ la variété des classes de conjugaisons sous $M_H(x)^{\circ}$ d'éléments semi-simples de $\mathfrak{m}_{H}(x)=\{(s,r_0) \in \mathfrak{h} \oplus \C \mid [s,x]=2r_{0}x \}$. D'après \cite[8.7]{Lusztig:1988} l'anneau des fonctions régulières sur $\mathcal{Z}$ est isomorphe à $H_{M_H(x)^{\circ}}^{\bullet}$.\\

Soit $(s,r_{0}) \in \mathfrak{m}_{H}(x)$ un élément semi-simple. On note $A_H(s,x)$ le groupe des composantes de $Z_H(x,s)=Z_H(x) \cap Z_H(s)$. L'évaluation au point $(s,r_0)$ définit une structure de $H_{M_H(x)^{\circ}}^{\bullet}$-algèbre sur $\C$ notée $\C_{(s,r_0)}$. On considère le module standard de Lusztig défini en \cite[\textsection 8]{Lusztig:1988} et \cite[10.8]{Lusztig:1995a} par : $$M(x,s,r_0)=\C_{(s,r_0)} \otimes_{H_{M_H(x)^{\circ}}^{\bullet}} H^{M_H(x)^{\circ}}_{\bullet}(\mathcal{P}_x, \dot{\mathcal{L}}) .$$ L'espace $M(x,s,r_0)$ est muni d'une action de $W_L^H \times A_H(x,s)$ (plus précisément il n'est pas seulement muni d'une action de $W_{L}^{H}$ mais c'est un module sur une algèbre de Hecke graduée). Pour tout $\eta \in \Irr(A_H(x,s))$, on note $$M(x,s,r_0,\eta)=\Hom_{A_H(s,x)}(\eta,M(x,s,r_0)).$$

Notons $\Irr(A_H(x))_{\mathfrak{t}}$ l'ensemble des représentations irréductibles $\widetilde{\eta}$ de $A_H(x)$ telles que $\Psi_{H}(\mathcal{C}_X^{H},\widetilde{\eta})=\mathfrak{t}$. Par ailleurs, on a une injection de $A_{H}(x,s)$ dans $A_{H}(x)$ et on note $\Irr(A_{H}(x,s))_{\mathfrak{t}}$ l'ensemble des représentations irréductibles de $A_{H}(x,s)$ apparaissant dans la restriction à $A_{H}(x,s)$ d'une représentation $\widetilde{\eta} \in \Irr(A_{H}(x))_{\mathfrak{t}}$.\\

\begin{prop}[{\cite[8.10]{Lusztig:1988}}]
En reprenant les notations précédentes, $M(x,s,r_0,\eta) \neq \{0\}$, si et seulement si, $\eta \in \Irr(A_{H}(x,s))_{\mathfrak{t}}$.
\end{prop}

Ajoutons une propriété qui nous sera utile pour la suite. D'après \cite[10.12.(d)]{Lusztig:1995a}, on a : $$M(x,s,r_0)=\C_{(s,r_0)} \otimes_{H_{M_H(x)^{\circ}}^{\bullet}} H^{M_H(x)^{\circ}}_{\bullet}(\mathcal{P}_x, \dot{\mathcal{L}})=H_{\bullet}^{\{1\}} (\mathcal{P}_x^{s},\widetilde{\mathcal{L}}),$$ où $\mathcal{P}_x^{s}=\{P' \in \mathcal{P}_x \mid s \in \mathfrak{p}'\}$. En particulier, si $M(x,s,r_0,\eta) \neq 0$ alors quitte à conjuguer, on peut supposer $s \in \mathfrak{p}$ et $x \in \mathcal{C}+\mathfrak{u}$. L'élément $s \in \mathfrak{p}$ étant semi-simple, il appartient à une sous-algèbre de Levi de $\mathfrak{p}$ qu'on note temporairement $\mathfrak{l}'$. Soit $x=y+n$ la décomposition de $x \in \mathfrak{l}+\mathfrak{u}$ et $y=y'+u' \in \mathfrak{l}' \oplus \mathfrak{u}$ la décomposition de $y \in \mathfrak{p}$. Par conséquent, puisque $P/U \simeq L \simeq L'$ et la $L$-orbite nilpotente de $y$ est isomorphe à la $L'$-orbite nilpotente de $y'$ et $s \in \mathfrak{l}'$. Ainsi, quitte à faire ces considérations, on peut supposer que $s \in \mathfrak{l}$. 

\begin{prop}\phantomsection\label{actionss}
En reprenant les notations précédentes, si $\eta \in \Irr(A_{H}(x,s))_{\mathfrak{t}}$, alors on peut supposer que $s \in \mathfrak{l}$ et on a : $$[s,y]=2r_0 y.$$
\end{prop}

\subsection{Correspondance de Langlands}

Soient $F$ un corps $p$-adique, $v_F : F \twoheadrightarrow\Z \cup \{ \infty\}$ sa valuation discrète normalisée et $q$ le cardinal de son corps résiduel. Soit $G$ (les $F$-points d'un) groupe réductif connexe défini et \emph{déployé} sur $F$. On notera $X^{*}(G)$ (resp. $X_{*}(G)$) le groupe des caractères (cocaractères) rationnels de $G$. La catégorie des représentations (complexes) lisses de $G$ est noté $\Rep(G)$ et l'ensemble des (classes de) représentations irréductibles de $G$ est noté $\Irr(G)$. Si $M$ est un sous-groupe de $G$ et $\rho$ une représentation de $M$, pour tout $g \in G$, on note ${}^{g} M=gMg^{-1}=\{gmg^{-1}, m \in M\}$ et $\rho^{g}$ la représentation de ${}^{g} M$ définie pour tout $h \in {}^{g} M$, par $\rho^{g}(h)=\rho(g^{-1}hg)$.\\
On note $W_F$ le groupe de Weil de $F$ et $W_F'=W_F \times \SL_2(\C)$ (resp. $WD_F=W_F \ltimes \C$) le groupe de Weil-Deligne (resp. le groupe Weil-Deligne « originel ») de $F$. Le groupe réductif complexe dual de Langlands de $G$ est noté $\widehat{G}$.\\
La correspondance de Langlands locale prévoit un paramétrage des représentations irréductibles admissibles de $G$ par certaines représentations du groupe de Weil-Deligne ou plus précisément, par des paramètres de Langlands. Un \emph{paramètre de Langlands de $G$} est un morphisme de groupes continu $\phi : W_F' \longrightarrow \widehat{G}$ tel que : 
\begin{itemize}
\item la restriction $\restriction{\phi}{\SL_2(\C)} : \SL_2(\C) \longrightarrow \widehat{G}$ est un morphisme de groupes algébriques ;
\item l'ensemble $\phi(W_F)$ est constitué d'éléments semi-simples.
\end{itemize}
Notons $\Phi(G)$ l'ensemble des classes de conjugaison par $\widehat{G}$ de paramètres de Langlands de $G$. La plupart du temps on ne fera pas de distinction entre un paramètre et sa classe de conjugaison par $\widehat{G}$. Lorsque qu'on veut distinguer ces objets, on notera $(\phi)_{\widehat{G}}$ la classe de conjugaison par $\widehat{G}$ d'un paramètre de Langlands $\phi$ de $G$.

La correspondance de Langlands locale prédit l'existence d'une surjection à fibres finies : $$\rec_{G} : \Irr(G) \longrightarrow \Phi(G).$$ 

Soit $\phi \in \Phi(G)$ un paramètre de Langlands de $G$. On note $\Pi_{\phi}(G)$ le $L$-paquet défini par $\phi$, c'est-à-dire l'ensemble fini des représentations irréductibles de $G$ ayant pour image $\phi$ par $\rec_{G}$.  Conjecturalement, $\Pi_{\phi}(G)$  est paramétré \textit{grosso modo} par les représentations irréductibles du groupe des composantes du centralisateur dans $\widehat{G}$ de l'image de $\phi$. Plus rigoureusement, considérons $Z_{\widehat{G}}(\phi)$ le centralisateur dans $\widehat{G}$ de l'image de $\phi$. On note
\begin{itemize}
\item $A_{\widehat{G}}(\phi)=\pi_{0}(Z_{\widehat{G}}(\phi)) \simeq Z_{\widehat{G}}(\phi)/ Z_{\widehat{G}}(\phi)^{\circ}$, le groupe des composantes de $Z_{\widehat{G}}(\phi)$ ;
\item $\mathcal{S}_{\phi}^{G}=\pi_{0}(Z_{\widehat{G}}(\phi)/Z_{\widehat{G}}) \simeq Z_{\widehat{G}}(\phi) / \left( Z_{\widehat{G}}\cdot Z_{\widehat{G}}(\phi)^{\circ} \right)$, le groupe des composantes de $Z_{\widehat{G}}(\phi)/Z_{\widehat{G}}$ ;
\item $\Irr(G)_{\cusp}$ l'ensemble des (classes de) représentations irréductibles \emph{cuspidales} de $G$ ;
\item $\Irr(G)_{2}$ l'ensemble des (classes de) représentations irréductibles \emph{essentiellement de carré intégrable }de $G$ ;
\item $\Irr(G)_{\temp}$ l'ensemble des (classes de) représentations irréductibles \emph{tempérées} $G$ ;
\item $\Phi(G)_{2}$ l'ensemble des (classes de) paramètres \emph{discrets} de $G$, c'est-à-dire les paramètres de Langlands dont l'image n'est contenue dans aucun sous-groupe Levi propre de $\widehat{G}$ ;
\item $\Phi(G)_{\bdd}$ l'ensemble des (classes de) paramètres de Langlands tels que l'image de $W_F$ est bornée ;
\item $\chi : W_F \longrightarrow Z_{\widehat{G}}$ le paramètre associé à un caractère $\chi$ de $G$ par la correspondance de Langlands pour les caractères (on les notera donc de la même façon).
\end{itemize}

On appelle \emph{paramètre de Langlands enrichi de $G$} un couple $(\phi,\eta)$ formé d'un paramètre de Langlands $\phi$ de $G$ et d'une représentation irréductible $\eta$ de $\mathcal{S}_{\phi}^{G}$. On note $\Phi_{e}(G)$ l'ensemble des (classes de conjugaison par $\widehat{G}$ de) paramètres de Langlands enrichis de $G$ : $$\Phi_{e}(G)= \left\{ (\phi,\eta) \mid \phi \in \Phi(G), \,\, \eta \in \Irr(\mathcal{S}_{\phi}^{G}) \right\}_{\widehat{G}-\rm{conj}}.$$

\begin{conj}[Correspondance de Langlands locale]
Avec les notations précédentes (on suppose donc $G$ déployé), il existe une surjection à fibres finies $$\rec_{G} : \Irr(G) \longrightarrow \Phi(G),$$ vérifiant les propriétés suivantes : \begin{itemize}
\item pour tout caractère $\chi$ de $G$ et pour tout $\pi \in \Irr(G)$, $\rec_{G}(\pi \otimes \chi)=\rec_{G}(\pi)\chi$;
\item pour tout $\phi \in \Phi(G)$, les conditions suivantes sont équivalentes  \begin{enumerate}[label=(\roman*)]
\item un élément de $\Pi_{\phi}(G)$ est une représentation essentiellement de carré intégrable modulo le centre;
\item tous les éléments de $\Pi_{\phi}(G)$ sont des représentations essentiellement de carré intégrable modulo le centre;
\item $\phi \in \Phi(G)_{2}$.
\end{enumerate}
\item pour tout $\phi \in \Phi(G)$, les conditions suivantes sont équivalentes \begin{enumerate}[label=(\roman*)]
\item un élément de $\Pi_{\phi}(G)$ est une représentation tempérée;
\item tous les éléments de $\Pi_{\phi}(G)$ sont des représentations tempérées;
\item $\phi \in \Phi(G)_{\bdd}$.
\end{enumerate}
\item pour tout $\phi \in \Phi(G)$,  il y a une bijection $\Pi_{\phi}(G) \leftrightarrow \Irr(\mathcal{S}_{\phi}^{G})$. Ainsi, la correspondance de Langlands se prolonge en une bijection $$\rec_{G}^{e} : \Irr(G) \xlongrightarrow{\sim} \Phi_{e}(G).$$
\end{itemize}
\end{conj}

Le dernier point permet d'identifier une représentation irréductible de $G$ à un paramètre de Langlands enrichi $(\phi,\eta)$. De plus, le second point énonce que la caractérisation, en terme de paramètres de Langlands, d'être essentiellement de carré intégrable pour une représentation, ne se traduit que sur le paramètre de Langlands (pas sur la représentation du $S$-groupe). Ainsi, conjecturalement, un paramètre de Langlands de $G$ discret définit un $L$-paquet constitué uniquement de représentations essentiellement de carré intégrable et réciproquement, une représentation essentiellement de carré intégrable est associée à un paramètre discret. On a le même phénomène concernant les représentations tempérées. En revanche, il est remarquable qu'on puisse avoir dans un même $L$-paquet des représentations irréductibles supercuspidales et des représentations  irréductibles non-supercuspidales. Par conséquent, la caractérisation des paramètres de Langlands des représentations supercuspidales porte nécessairement une condition sur $\Irr(\mathcal{S}_{\phi}^{G})$. On reviendra sur ce point plus tard (voir conjecture \ref{conjcusp}). \\ 

\subsection{Paramètres de Langlands des supercuspidales des groupes classiques d'après M\oe glin}

À présent, on se propose de rappeler succinctement la notion de bloc de Jordan, introduite par M\oe glin, pour les représentations et les paramètres de Langlands des groupes classiques. Ceci nous permettra d'énoncer le théorème de M\oe glin sur le paramétrage de Langlands des représentations irréductibles supercuspidales des groupes classiques d'après la construction d'Arthur.\\

Soit $G$ l'un des groupes déployés $\Sp_{N}(F)$ ou $\SO_{N}(F)$. Les sous-groupes de Levi de $G$ sont de la forme $\GL_{n_1}(F) \times \ldots \times \GL_{n_r}(F) \times G'$, où $G'$ est un groupe de même type que $G$ mais de rang semi-simple inférieur. Par conséquent, si $\pi_{i}$ est une représentation de $\GL_{n_i}(F)$ et $\sigma'$ une représentation de $G'$, comme il est d'usage de noter, nous noterons $\pi_{1} \times \ldots \times \pi_{r} \rtimes \sigma'$ pour l'induite parabolique normalisée de $\pi_{1} \boxtimes \ldots \boxtimes \pi_{r} \boxtimes \sigma$, relativement à un sous-groupe parabolique standard (pour le sous-groupe de Borel des matrices triangulaires supérieurs et du tore maximal diagonal). Enfin, si $\pi$ est une représentation irréductible supercuspidale unitaire d'un groupe linéaire $\GL_{d}(F)$ et $a \geqslant 1$ un entier, la représentation induite $$\pi |\, \, \,|^{\frac{a-1}{2}} \times \pi |\, \, \,|^{\frac{a-3}{2}} \times \ldots \times \pi | \, \, \,|^{\frac{1-a}{2}}$$ admet une unique sous-représentation irréductible, c'est une représentation de la série discrète de $\GL_{ad}(F)$ et on la note $\St(\pi,a)$.\\ Soit $\sigma$ une représentation irréductible de la série discrète de $G$. On appelle bloc de Jordan de $\sigma$ et on note $\Jord(\sigma)$ l'ensemble des couples $(\pi,a)$ formés d'une représentation irréductible supercuspidale irréductible unitaire $\pi$ d'un groupe $\GL_{n_{\pi}}(F)$ et d'un entier $a \geqslant 1$ tel qu'il existe un entier $a' \geqslant 1$ et $$\left\{\begin{array}{l}
a \equiv a' \mod 2  \\  
\St(\pi,a) \rtimes \sigma \;\;  \text{irréductible} \\  
\St(\pi,a') \rtimes \sigma \;\;   \text{réductible}
\end{array} \right.$$ De plus, Arthur a associé à $\sigma'$ un caractère d'un certain groupe fini que l'on note $\varepsilon_{\sigma'}$ (voir \cite[2.5]{Moeglin:2011}).\\ 

Notons $\Std_{G} : \widehat{G} \hookrightarrow \GL_{N}(\C)$ la « représentation standard » de $\widehat{G}$, c'est-à-dire :
$$
\renewcommand{\arraystretch}{1.3}
\begin{array}{|c|c|}
\hline
G & \Std_{G}  \\ \hline\hline
\SO_{2n+1}(F) & \Sp_{2n}(\C) \hookrightarrow \GL_{2n}(\C) \\  \hline
\Sp_{2n}(F) & \SO_{2n+1}(\C) \hookrightarrow \GL_{2n+1}(\C) \\  \hline
\SO_{2n}(F) & \SO_{2n}(\C) \hookrightarrow \GL_{2n}(\C) \\ \hline 
\end{array}$$

Soit $\varphi \in \Phi(G)_{2}$ un paramètre discret de $G$. Considérons la décomposition en composantes isotypiques de $\Std_{G} \circ \varphi$ : $$\Std_{G} \circ \varphi=\bigoplus_{\pi \in I} \bigoplus_{a \in \Jord(\varphi)_{\pi}} \pi \boxtimes S_{a},$$ où $S_a$ est la représentation irréductible de dimension $a$ de $\SL_2(\C)$, $I$ est un ensemble de (classes de) représentations irréductibles de $W_F$ et pour $\pi \in I$, $\Jord(\varphi)_{\pi}$ l'ensemble des entiers $a \geqslant 1$ tels que $\pi \boxtimes S_a$ soit une sous-représentation irréductible de $\Std_{G} \circ \varphi$. La discrétion du paramètre $\varphi$ implique qu'il n'y a pas de multiplicité. On appelle bloc de Jordan de $\varphi$ l'ensemble $\Jord(\varphi)=\left\{(\pi,a) \mid \pi \in I, a \in \Jord(\varphi)_{\pi} \right\}$. On dira que $\Jord(\varphi)$ est sans trou si pour tout $(\pi,a) \in \Jord(\varphi)$ avec $a \geqslant 3$ alors $(\pi,a-2) \in \Jord(\varphi)$.\\ Pour $(\pi,a) \in \Jord(\varphi)$ est associé un élément $z_{\pi,a} \in Z_{\widehat{G}}(\varphi)$ qui agit par le scalaire $-1$ sur l'espace associé à $\pi \boxtimes S_a$ et par $1$ ailleurs. On notera également $z_{\pi,a}$ son image $A_{\widehat{G}}(\varphi)$. Le groupe $A_{\widehat{G}}(\varphi)$ est engendré par $$\left\{\begin{array}{ll}
z_{\pi,a} & \mbox{pour $(\pi,a) \in \Jord(\varphi)$ et $a$ pair}\\  
z_{\pi,a}z_{\pi,a'} &  \mbox{pour $(\pi,a), (\pi,a') \in \Jord(\varphi)$ sans hypothèse de parité sur $a$ et $a'$}   
\end{array} \right.$$ Pour $(\pi,a), (\pi,a') \in \Jord(\varphi)$, avec $a < a$, on dit qu'ils sont consécutifs si pour tout $b \in \llbracket a+1,a'-1 \rrbracket$, $(\pi,b) \not \in \Jord(\varphi)$. Enfin, on note $a_{\pi,\min}$ le plus petit entier $a \geqslant 1$ tel que $(\pi,a) \in \Jord(\varphi)$. Un caractère $\varepsilon$ de $A_{\widehat{G}}(\varphi)$ sera dit alterné si pour tout $(\pi,a), (\pi,a') \in \Jord(\varphi)$ consécutifs, $\varepsilon(z_{\pi,a}z_{\pi,a'} )=-1$ et si pour tout $(\pi,a_{\pi,\min}) \in \Jord(\varphi)$ avec $a_{\pi,\min}$ pair, $\varepsilon(z_{\pi,a_{\pi,\min}})=-1$.\\

À présent, nous pouvons énoncer le paramétrage des représentations irréductibles supercuspidales de $G$.

\begin{theo}[M\oe glin, {\cite[2.51]{Moeglin:2011}}]\phantomsection\label{thmmoeglin}
La classification de Langlands des séries discrètes de $G$ telle qu'établie par Arthur, induit une bijection entre l'ensemble des classes des représentations irréductibles supercuspidales de $G$ et l'ensemble des couples $(\varphi,\varepsilon)$ tel que $\Jord(\varphi)$ est sans trou et $\varepsilon$ est alternée ; la bijection $\sigma \mapsto (\varphi,\varepsilon)$ est définie par le fait que $\Jord(\varphi)=\Jord(\sigma)$ et $\varepsilon=\varepsilon_{\sigma}$.
\end{theo}

Ajoutons pour conclure un théorème dont on se servira par la suite qui permet de calculer les points de réductibilités d'une induite de supercuspidale.

\begin{theo}[M\oe glin, {\cite[\textsection 3.2]{Moeglin:2014}}]\phantomsection\label{thmmoeglinreduc}
En reprenant les notations précédentes, si $\pi$ est une représentation supercuspidale irréductible unitaire de $\GL_{d}(F)$ alors le réel $x \in \R^{+}$ tel que $\pi | \cdot |^{x} \rtimes \sigma$ est réductible,  vaut $$x = \left\{
\begin{array}{ll}
\frac{a_{\pi}+1}{2} & \mbox{si $\pi \in \Jord(\sigma)$, avec $a_{\pi}=\max \{a \in \N, (\pi ,a) \in \Jord(\sigma)\}$} \\
\frac{1}{2} & \mbox{si $\pi \not \in \Jord(\sigma)$ et ($\pi $ et $\widehat{G}$) sont de même type} \\
0 & \mbox{si $\pi \not \in \Jord(\sigma)$ et ($\pi $ et $\widehat{G}$) sont de type différent} 
\end{array} \right.$$
\end{theo}

\subsection{Centre de Bernstein}\phantomsection\label{centredebernstein}

Soit $G$ (les $F$-points d'un) groupe réductif connexe défini et déployé sur un corps $p$-adique $F$. Soit $P=MN$ un sous-groupe parabolique de $G$, de facteur de Levi $M$. Le foncteur d'induction normalisé $i_P^G : \Rep(M) \longrightarrow \Rep(G)$ et le foncteur de Jacquet $r_P^G : \Rep(G) \longrightarrow \Rep(M)$ sont deux foncteurs essentiels en théorie des représentations. En effet, l'induction permet de construire des représentations de $G$ à partir des représentations d'un sous-groupe de ce dernier. Dès lors, l'étude des représentations de $G$ se décompose en l'étude des représentations qu'on obtient par le procédé d'induction et l'étude des représentations qui ne sont pas obtenues ainsi. C'est l'objet de la théorie du centre de Bernstein. On obtient une décomposition de $\Rep(G)$ en produit de sous-catégories pleines indécomposables dont le centre (de chacune de ces sous-catégories) est isomorphe à l'algèbre des fonctions régulières du quotient d'un tore par l'action d'un groupe fini.\\

Notons $H_{G} : G \longrightarrow X_{*}(G)$ l'application défini par la formule $$\forall g \in G, \chi \in X^{*}(G), \,\,  \langle \chi,H_{G} \rangle = v_{F}(\chi(g)).$$ On note $G^{1}=\{g \in G \, \mid \, \forall \chi \in X^{*}(G), \, v_F(\chi(g))=0 \}$ le noyau de cette application et $\Lambda(G) \subset X_{*}(G)$ son image. Soient $\mathcal{X}(G)=\Hom(G/G^{1},\C^{\times})$ le groupe des caractères non-ramifiés de $G$ et $\mathfrak{a}_{G,\C}^{*}=X^{*}(G) \otimes_{\Z} \C$. On a une surjection \begin{equation*}
\begin{array}{ccc}
\mathfrak{a}_{G,\C}^{*} & \twoheadrightarrow &  \mathcal{X}(G) \\
\chi \otimes s & \longmapsto & \left[ g \mapsto | \chi(g) |^{s} \right]
\end{array},
\phantomsection\label{appcarnr}
\end{equation*} de noyau de la forme $2\pi i /(\log q) R$, où $R$ est un certain réseau. Ceci définit une structure de tore algébrique complexe sur $\mathcal{X}(G)$. De plus, pour tout $\nu \in \mathfrak{a}_{G,\C}^{*}$, on notera $\chi_{\nu}$ le caractère non-ramifié de $G$ associé par la surjection précédente. \\

Soit $\pi$ une représentation irréductible lisse de $G$. Soit $P$ un sous-groupe parabolique de $G$ de facteur de Levi $M$ tel que $r_{P}^{G}(\pi) \neq 0$ et minimal pour cette propriété. Soit $\sigma$ un sous-quotient irréductible de $r_{P}^{G}(\pi)$. Alors $\sigma$ est une représentation irréductible supercuspidale de $M$. De plus, si $(M',\sigma')$ est un couple obtenu de la même façon pour le choix d'un autre sous-groupe parabolique $P'$ vérifiant la même propriété, alors il existe $g \in G$ tel que $M'={}^g M$ et $\sigma \simeq \sigma'^{g}$ (voir \cite[Lemme VI.7.1]{Renard:2010}). On appelle \emph{support cuspidal de $\pi$} la classe de conjugaison par $G$ du couple $(M,\sigma)$.\\

Considérons l'ensemble des \emph{données cuspidales de $G$}, c'est-à-dire l'ensemble des couples $(M,\sigma)$, où $M$ est un sous-groupe de Levi de $G$ et $\sigma$ une représentation irréductible supercuspidale de $M$. On définit les relations d'équivalences suivantes sur l'ensemble des données cuspidales de $G$. Soient $(M_1,\sigma_1)$ et $(M_2,\sigma_2)$ deux données cuspidales de $G$. On dit qu'elles sont : \begin{enumerate}[label=(\roman*)]
\item équivalentes s'il existe $g \in G$ tel que : ${}^g M_1 = M_2 $ et $\sigma_2 \simeq \sigma_1^{g}$ ;
\item inertiellement équivalentes s'il existe $g \in G$ et $\chi_2 \in \mathcal{X}(M_2)$ tels que : ${}^g M_1 = M_2 $ et $\sigma_2 \simeq \sigma_1^{g} \otimes \chi_2$.
\end{enumerate}

On appelle \emph{paire cuspidale} (resp. \emph{paire inertielle}) de $G$ une classe d'équivalence pour la relation (i) (resp. (ii)) et on notera $\Omega(G)$ (resp. $\mathcal{B}(G)$) l'ensemble des classes d'équivalence pour la relation (i) (resp. (ii)). De plus, on notera $(M,\sigma) \in \Omega(G)$ la paire cuspidale (resp. $[M,\sigma] \in \mathcal{B}(G)$ la paire inertielle) définie à partir de $M$ et $\sigma$.\\

On a vu précédemment comment associer à toute représentation irréductible $\pi$ de $G$ son support cuspidal $(M,\sigma) \in \Omega(G)$ : c'est la classe de conjugaison par $G$ d'une donnée cuspidale $(M,\sigma)$ telle que $\sigma$ soit un facteur de composition de $r_P^G(\pi)$, où $P$ est un sous-groupe parabolique de facteur de Levi $M$ ou de façon équivalente, telle que $\pi$ soit un facteur de composition de $i_P^G(\sigma)$. Ceci définit des applications de support cuspidal et support inertiel : $$\Sc : \Irr(G) \longrightarrow \Omega(G) \quad \text{et} \quad \Si : \Irr(G) \longrightarrow \Omega(G) \twoheadrightarrow \mathcal{B}(G).$$

Soit $\mathfrak{s}=[M,\sigma] \in \mathcal{B}(G)$ une paire inertielle de $G$. On définit les objets suivants : \begin{itemize}
\item un tore algébrique $T_{\mathfrak{s}}=\{ \rho \otimes \chi, \chi \in \mathcal{X}(M)\} \simeq \mathcal{X}(M)/\mathcal{X}(M)(\rho)$, où $\mathcal{X}(M)(\rho)=\{ \chi \in \mathcal{X}(M) \mid \rho \simeq \rho \otimes \chi \}$ ;
\item un groupe fini $W_{\mathfrak{s}}=N_G(\mathfrak{s})/M=\{g \in N_G(M) \mid \exists \chi \in \mathcal{X}(M), \rho^{w} \simeq \rho \otimes \chi \}/M$.
\end{itemize}

Le groupe fini $W_\mathfrak{s}$ agit sur $T_\mathfrak{s}$ et le quotient $T_\mathfrak{s} / W_\mathfrak{s}$ s'identifie aux paires cuspidales de $G$ dans $\mathfrak{s}$, c'est-à-dire à la fibre au-dessus de $\mathfrak{s}$ par la projection $\Omega(G) \twoheadrightarrow \mathcal{B}(G)$. Ainsi, $\Omega(G)$ est muni d'une structure de variété algébrique dont les composantes connexes, indexées par $\mathfrak{s} \in \mathcal{B}(G)$, sont des quotients de tores algébriques complexes par l'action de groupes finis (voir \cite[Théorème VI.7.1]{Renard:2010}). Ceci s'écrit $$\Omega(G) = \bigsqcup_{\mathfrak{s} \in \mathcal{B}(G)} T_{\mathfrak{s}}/W_{\mathfrak{s}}.$$

L'application $\Si$ induit une partition de $\Irr(G)$ suivant le support inertiel $$\Irr(G)=\bigsqcup_{\mathfrak{s} \in \mathcal{B}(G)} \Irr(G)_{\mathfrak{s}},$$ où $\Irr(G)_{\mathfrak{s}}=\Si^{-1}(\mathfrak{s})$. Plus concrètement, $\Irr(G)_{\mathfrak{s}}$ est l'ensemble des sous-quotients irréductibles des induites $i_P^G(\rho \otimes \chi)$, pour $\chi$ parcourant $\mathcal{X}(M)$, c'est-à-dire $$\Irr(G)_{\mathfrak{s}}=\bigsqcup_{\rho \otimes \chi \in T_\mathfrak{s}/W_\mathfrak{s}} \mathcal{JH}(i_P^G(\rho \otimes \chi)),$$ où $\mathcal{JH}(i_P^G(\rho \otimes \chi))$ désigne l'ensemble des sous-quotients irréductibles de $i_P^G(\rho \otimes \chi)$.\\

A présent, pour tout $\mathfrak{s} \in \mathcal{B}(G)$, notons $\Rep(G)_{\mathfrak{s}}$ la sous-catégorie pleine de $\Rep(G)$ des représentations dont tous les sous-quotients irréductibles sont dans $\Irr(G)_{\mathfrak{s}}$, $\mathfrak{Z}(G)$ (resp. $\mathfrak{Z}(G)_{\mathfrak{s}}$) le centre de la catégorie $\Rep(G)$ (resp. $\Rep(G)_{\mathfrak{s}}$).

\begin{theo} [Bernstein,{{\cite[2.10, 2.13]{Bernstein:1984}}, \cite[Théorème VI.7.2, VI.10.3]{Renard:2010}}]
Toute représentation $\pi$ de $G$ est scindée selon $\Omega(G)$. La catégorie $\Rep(G)$ se décompose en produit de catégories $$\Rep(G)=\prod_{\mathfrak{s} \in \mathcal{B}(G)}  \Rep(G)_{\mathfrak{s}}.$$ De plus, $$\mathfrak{Z}(G)_{\mathfrak{s}} \simeq \C[T_{\mathfrak{s}}/W_{\mathfrak{s}}].$$ Ainsi, $$\mathfrak{Z}(G) \simeq \C[\Omega(G)]$$
\end{theo}

\subsection{Centre de Bernstein stable}

On garde les notations de la section précédente, ainsi $G$ désigne un groupe réductif connexe défini et déployé sur un corps $p$-adique $F$. On dispose de deux paramétrages des représentations irréductibles de $G$. L'un par la décomposition de Bernstein, l'autre par la correspondance (conjecturale en général) de Langlands. \textit{A priori}, il n'y a pas de bonne compatibilité entre ces paramétrages. Afin d'énoncer un analogue des résultats qu'il a obtenu dans le cas réel au cas $p$-adique, Vogan introduit un analogue « galoisien » du centre de Bernstein. Le but de cette partie est de rappeler la construction de cet analogue « galoisien » du centre de Bernstein par Haines, puis de le relier au centre de Bernstein via la correspondance de Langlands. On s'intéressera à décrire l'analogue du tore et du groupe fini. Pour cela, on doit définir les analogues « galoisiens » d'une représentation supercuspidale, du tore des caractères non ramifiés d'un Levi, du groupe fini qui agit sur le tore et enfin du support cuspidal. On suit l'article de Haines \cite{Haines:2014} dont on a modifié quelques notations et définitions. En particulier, on ne traite que le cas où $G$ est déployé alors que \cite{Haines:2014} traite le cas général (c'est-à-dire non quasi-déployé). Les définitions changent aussi légèrement, voir \cite[Rem. 5.3.5 et 5.3.6]{Haines:2014} et la note de bas de page p.10.

\begin{defi}[{\cite[5.1]{Haines:2014}}]
On appelle \emph{caractère infinitésimal} de $G$, d'un paramètre de Langlands de $G$ de la forme $$\lambda : W_F \longrightarrow \widehat{G},$$ ou sa classe de conjugaison par $\widehat{G}$ et on le notera $(\lambda)_{\widehat{G}}$.
\end{defi}

Pour tout $w \in W_F$, on note $d_w=\left( \begin{smallmatrix} |w|^{1/2} & 0 \\0 &  |w|^{-1/2} \end{smallmatrix} \right) \in \SL_2(\C)$, où $|\cdot |$ désigne la valeur absolue définie pour tout $w \in W_F$ par $|w|=q^{-\nu_F(w)}$ et $\nu_F : W_F \longrightarrow \Z$ la valuation qui envoie tout Frobenius géométrique sur $1$. À tout paramètre de Langlands de $G$ est associé un caractère infinitésimal de la manière suivante. 

\begin{defi}
Soit $\phi : W_F' \longrightarrow \widehat{G}$ un paramètre de Langlands de $G$. On appelle \emph{caractère infinitésimal} de $\phi$, le morphisme $\lambda_{\phi} : W_F \longrightarrow \widehat{G}$ (ou sa classe de conjugaison par $\widehat{G}$) défini pour tout $w \in W_F$ par $\lambda_{\phi}(w)=\phi(w,d_w)$.
\end{defi}

\begin{rema}
Tel qu'il est défini ici, le caractère infinitésimal d'un paramètre de Langlands correspond à la restriction au groupe de Weil du paramètre de Langlands pour le groupe de Weil-Deligne « originel ». Nous reviendrons sur ce point dans la section \ref{groupesweildeligne}
\end{rema}

Soit $M$ un sous-groupe de Levi de $G$ et posons $\Lambda=\left(X^{*}(Z_{\widehat{M}})_{I_F}\right)^{\langle \Fr \rangle}=X^{*} \left( \left({Z_{\widehat{M}}}^{I_F}\right)_{\langle \Fr \rangle} \right)$. Dans \cite{Kottwitz:1997}, Kottwitz a défini un morphisme surjectif $\kappa_M : M \twoheadrightarrow \Lambda$, tel que $M^{1}$ (voir section \ref{centredebernstein}) est le noyau du morphisme $M \twoheadrightarrow \Lambda/\Lambda_{\mathrm{tors}}$.  Par suite, on obtient une bijection $$\mathcal{X}(M) \leftrightarrow \left(\left({Z_{\widehat{M}}}^{I_F}\right)_{\langle \Fr \rangle}\right)^{\circ}.$$ Posons $\mathcal{X}(\widehat{M})=H^{1}\left(\langle \Fr \rangle, \left({Z_{\widehat{M}}}^{I_F}\right)^{\circ}\right)$. Puisqu'on suppose $G$ déployé, on a alors $\mathcal{X}(\widehat{M})=\Hom(\langle \Fr \rangle,Z_{\widehat{M}}^{\circ})$. Tout élément de $\mathcal{X}(\widehat{M})$ est identifié à un morphisme $W_F \longrightarrow \widehat{M}$, trivial sur $I_F$, à valeur dans $Z_{\widehat{M}}^{\circ}$. On obtient ainsi une bijection entre $\mathcal{X}(M)$ et $\mathcal{X}(\widehat{M})$, compatible avec la correspondance de Langlands pour les caractères d'après \cite[4.5.2]{Kaletha:2012}.

\begin{defi}[{\cite[5.3.3]{Haines:2014}}]
On appelle \emph{donnée cuspidale de $\widehat{G}$} un couple $(\widehat{M},\lambda)$ formé d'un sous-groupe de Levi $\widehat{M}$ de $\widehat{G}$ et d'un paramètre de Langlands $\lambda : W_F \longrightarrow \widehat{M}$ discret de $M$. On définit les relations d'équivalences suivantes sur l'ensemble des données cuspidales de $\widehat{G}$. Soient $(\widehat{M}_1,\lambda_1)$ et $(\widehat{M}_2,\lambda_2)$ deux données cuspidales de $\widehat{G}$. On dit qu'elles sont : \begin{enumerate}[label=(\roman*)]
\item équivalentes s'il existe $g \in \widehat{G}$ tel que : ${}^g  \widehat{M}_1 =\widehat{M}_2 \,\, \text{et} \,\, \lambda_2= {}^g \lambda_1$ ;
\item inertiellement équivalentes s'il existe $g \in \widehat{G}$ et $\chi_2  \in \mathcal{X}(\widehat{M_2})$ tels que : ${}^g \widehat{M}_1=\widehat{M}_2 \,\, \text{et} \,\, \lambda_2 ={}^g \lambda_1 \chi_2$.
\end{enumerate}
\end{defi}

On appelle \emph{paire cuspidale} (resp. \emph{paire inertielle}) de $\widehat{G}$ une classe d'équivalence pour la relation (i) (resp. (ii)) et on notera $\Omega^{\st}(G)$ (resp. $\mathcal{B}^{\st}(G)$) l'ensemble des classes d'équivalence pour la relation (i) (resp. (ii)). De plus, on notera $(\widehat{M},\lambda) \in \Omega^{\st}(G)$ la paire cuspidale (resp. $[\widehat{M},\lambda] \in \mathcal{B}^{\st}(G)$ la paire inertielle) définie à partir de $\widehat{M}$ et $\lambda$.\\

À tout paramètre de Langlands on peut associer son \emph{support cuspidal stable} de la façon suivante. Soit $\phi \in \Phi(G)$ un paramètre de Langlands de $G$. Considérons son caractère infinitésimal $\lambda_{\phi}$. Le centralisateur dans $\widehat{G}$ d'un tore maximal de $Z_{\widehat{G}}(\lambda_{\phi})^{\circ}$ est un sous-groupe de Levi $\widehat{M}_{\lambda_{\phi}}$ de $\widehat{G}$ qui contient l'image de $\lambda_{\phi}$ et qui est minimal pour cette propriété. Tous les sous-groupes de Levi de $\widehat{G}$ qui contiennent minimalement l'image de $\lambda_{\phi}$ sont obtenus de cette façon. En particulier, ils sont tous conjugués sous l'action de $Z_{\widehat{G}}(\lambda_{\phi})^{\circ}$. Par conséquent, l'application $$ \cSc^{\st} : \begin{array}[t]{rcl} \Phi(G) & \longrightarrow & \Omega^{\st}(G) \\ \phi & \longmapsto & (\widehat{M}_{\lambda_{\phi}},\lambda_{\phi}) \end{array},$$ est bien définie et on notera $\Sil^{\st}$ la composée de $\cSc^{\st}$ avec la projection $\Omega^{\st}(G) \twoheadrightarrow \mathcal{B}^{\st}(G)$.

Comme pour $\Omega(G)$, l'ensemble $\Omega^{\st}(G)$ est muni d'une structure de variété algébrique dont les composantes connexes sont indexées par $\mathcal{B}^{\st}(G)$ et sont des quotients de tores complexes par l'action de groupes finis.

Soit $\wii=[\widehat{M},\lambda] \in \mathcal{B}^{\st}(G)$ une paire inertielle de $\widehat{G}$. On définit les objets suivants : \begin{itemize}
\item un tore algébrique $\mathcal{T}_{\ci}=\left\{(\lambda \chi)_{\widehat{M}}, \, \chi \in \mathcal{X}(\widehat{M})\right\} \simeq \mathcal{X}(\widehat{M})/\mathcal{X}(\widehat{M})(\lambda)$, où $\mathcal{X}(\widehat{M})(\lambda)=\{\chi \in \mathcal{X}(\widehat{M}) \mid (\lambda)_{\widehat{M}}=(\lambda \chi)_{\widehat{M}} \}$ ;
\item un groupe fini $\mathcal{W}_\ci=\left\{ w \in N_{\widehat{G}}(\widehat{M}) \mid \exists \chi \in  \mathcal{X}(\widehat{M}), ({}^w \lambda)_{\widehat{M}}=(\lambda \chi)_{\widehat{M}} \right\}/\widehat{M}.$
\end{itemize}

Le groupe fini $\mathcal{W}_\ci$ agit sur $\mathcal{T}_{\ci}$ et le quotient $\mathcal{T}_\ci/\mathcal{W}_\ci$ s'identifie aux paires cuspidales de $\widehat{G}$ dans $\wii$, c'est-à-dire à la fibre au-dessus de $\wii$ par la projection naturelle $\Omega^{\st}(G) \twoheadrightarrow \mathcal{B}^{\st}(G)$. Ainsi, $\Omega^{\st}(G)$ est muni d'une structure de variété algébrique dont les composantes connexes, indexées par $\wii \in \mathcal{B}^{\st}(G)$, sont des quotients de tores algébriques complexes par l'action de groupes finis. Ceci s'écrit $$\Omega^{\st}(G) = \bigsqcup_{\wii \in \mathcal{B}^{\st}(G)} \mathcal{T}_\ci/\mathcal{W}_\ci.$$

L'application $\Sil^{\st}$ induit une partition de $\Phi(G)$ suivant le support inertiel stable $$\Phi(G)= \bigsqcup_{ \ci \in \mathcal{B}^{\st}(G)} \Phi(G)_{\ci},$$ où $\Phi(G)_{\ci}={\Sil^{\st}}^{-1}(\wii)$ désigne l'ensemble des classes de paramètres de Langlands de $G$ qui admettent un support cuspidal stable dans $\wii$.

\begin{defi}[{\cite[p.15]{Haines:2014}}]
On appelle centre de Bernstein stable de $G$ et on note $\mathfrak{Z}^{\st}(G)$ l'anneau des fonctions régulières sur $\Omega^{\st}(G)$ : $$\mathfrak{Z}^{\st}(G)=\C[\Omega^{\st}(G)].$$
\end{defi} 

La terminologie « stable » trouve son origine dans la volonté de faire agir un sous-anneau du centre de Bernstein sur les distributions stables de $G$ (voir \cite[\textsection 7]{Vogan:1993}, \cite[Rem. 5.5.4]{Haines:2014} et \cite[\textsection 6]{Scholze:2013}).\\ 

\begin{conj}[de compatibilité de la correspondance de Langlands avec l'induction parabolique, {Vogan \cite[7.18]{Vogan:1993}, Haines \cite[5.2.2]{Haines:2014}} ]\phantomsection\label{conjindpar}
Soit $P$ un sous-groupe parabolique de $G$ de facteur de Levi $M$. Soient $\sigma$ une représentation irréductible supercuspidale de $M$ et $\pi$ un sous-quotient irréductible de $i_P^G(\sigma)$. Soient $\phi_\sigma : W_F' \longrightarrow \widehat{M}$ et $\phi_\pi : W_F' \longrightarrow \widehat{G}$ les paramètres de Langlands respectifs de $\sigma$ et $\pi$. À $\widehat{G}$-conjugaison près, on a un plongement naturel $\widehat{M} \hookrightarrow \widehat{G}$ et on peut voir $\phi_\sigma : W_F' \longrightarrow \widehat{M} \hookrightarrow \widehat{G}$ comme un paramètre de Langlands de $G$. La correspondance de Langlands devrait être compatible avec l'égalité des caractères infinitésimaux suivante : $$(\lambda_{\phi_\pi})_{\widehat{G}}=(\lambda_{\phi_\sigma})_{\widehat{G}}.$$
\end{conj} 
 
\begin{defi}[{\cite[5.1 \& 5.3.3]{Haines:2014}}]
Soit $\ci=[\widehat{M},\lambda] \in \mathcal{B}^{\st}(G)$ une paire inertielle de $\widehat{G}$. On appelle \emph{classe infinitésimale} (ou \emph{paquet infinitésimal}) de $\lambda$ la réunion des $L$-paquets de paramètres admettant $\lambda$ pour caractère infinitésimal et on note $$\Pi_{\lambda}^{+}(G)=\bigsqcup_{\substack{\phi \in \Phi(G) \\ (\lambda_{\phi})_{\widehat{G}}=(\lambda)_{\widehat{G}}}} \Pi_{\phi}(G).$$ On appelle \emph{paquet inertiel de $\ci$} la réunion des $L$-paquets de paramètres admettant $\lambda$ pour caractère infinitésimal à un cocaractère non-ramifié près de $\widehat{M}$ et on note $$\Pi_{\ci}^{+}(G)=\bigsqcup_{\substack{\phi \in \Phi(G) \\ (\lambda_{\phi})_{\widehat{G}}=(\lambda\chi)_{\widehat{G}} \\ \chi \in \mathcal{X}(\widehat{M})}} \Pi_{\phi}(G)=\bigsqcup_{(\lambda\chi)_{\widehat{M}} \in \mathcal{T}_{\ci}/\mathcal{W}_{\ci}} \Pi_{\lambda\chi}^{+}(G).$$
\end{defi}

\section{Centre de Bernstein dual}

Dans cette section nous allons conjecturer une propriété qui devrait caractériser les paramètres de Langlands enrichis des représentations supercuspidales des groupes déployés. Le principal ingrédient de cette conjecture et des constructions qui suivront est la correspondance de Springer généralisée. Nous énonçons donc notre conjecture sur le paramétrage des représentations supercuspidales, nous vérifions que cette conjecture est vraie pour les cas connus de la correspondance de Langlands, puis nous définissons une application de support cuspidal sur les paramètres de Langlands enrichis. Ainsi, nous sommes en mesure de définir l'analogue galoisien de la décomposition et du centre de Bernstein : le centre de Bernstein dual.\\

En reprenant les notations de la conjecture \ref{conjindpar}, la conjecture de compatibilité de la correspondance de Langlands avec l'induction parabolique implique que le caractère infinitésimal de $\phi_\pi$ ne dépend que du support cuspidal de $\pi$. Bien que $\sigma$ soit supercuspidale, la restriction à $\SL_2(\C)$ de $\phi_\sigma$ n'est pas nécessairement triviale. Cette situation n'a pas lieu pour $\GL_n(F)$ et $\SL_n(F)$ mais dans $\GSp_4(F), \Sp_4(F)$ et $\SO_5(F)$ elle peut se produire. Donnons un exemple.

\begin{exem} Considérons le groupe $G=\SO_5(F)$ et le paramètre $\varphi \in \Phi(G)$ donné par $\Std_{G} \circ \varphi=\mu_1 \boxtimes S_2 \oplus \mu_2 \boxtimes S_2$ de $\SO_5(F)$, où $\mu_1$ et $\mu_2$ sont deux caractères quadratiques distincts de $W_F$ et $S_2$ la représentation irréductible de dimension $2$ de $\SL_2(\C)$. Il définit un $L$-paquet constitué d'une représentation de carré intégrable qui est un sous-quotient irréductible de l'induite $\nu^{1/2}\mu_1 \times \mu_2 \nu^{1/2} \rtimes 1$ et d'une représentation supercuspidale de $\SO_5(F)$. \end{exem}

Reprenons les notations de la section précédentes et soit $(\widehat{M},\lambda) \in \Omega^{\st}(G)$. Notons $\mathcal{L}(\widehat{G})$ l'ensemble des (classes de conjugasion par $\widehat{G}$ de) sous-groupes de Levi de $\widehat{G}$. Pour tout sous-groupe de Levi $L$ de $G$ et pour tout paramètre de Langlands $\varphi \in \Phi(L)$ de $L$, notons $\Pi_{\varphi}(L)_{\cusp}$ les représentations irréductibles supercuspidales dans le paquet $\Pi_{\varphi}(L)$.\\

Supposons que la conjecture \ref{conjindpar} soit vraie. Soient $\pi \in \Pi_{\lambda}^{+}(G)$ et $\phi_{\pi} \in \Phi(G)$ le paramètre de Langlands de $\pi$. Soient $(L,\tau) \in \Omega(G)$ le support cuspidal de $\pi$ et $\varphi_{\tau}\in \Phi(L)$ le paramètre de Langlands de $\tau$. Par hypothèse, $(\lambda_{\phi_{\pi}})_{\widehat{G}}=(\lambda_{\varphi_{\tau}})_{\widehat{G}}=(\lambda)_{\widehat{G}}$ et ainsi $\mathcal{JH}(i_P^G(\tau)) \subseteq \Pi_{\lambda}^{+}(G)$. On obtient donc $$\Pi_{\lambda}^{+}(G)\subseteq \bigsqcup_{\widehat{L} \in \mathcal{L}(\widehat{G})} \bigsqcup_{\substack{ \varphi \in \Phi(L)\\ (\lambda_{\varphi})_{\widehat{G}}=(\lambda)_{\widehat{G}}}} \bigsqcup_{\sigma \in \Pi_{\varphi}(L)_{\cusp}} \mathcal{JH}(i_{LU}^{G}(\sigma)).$$ L'inclusion réciproque étant immédiate, et on obtient : 

\begin{prop}\phantomsection\label{compatibilitebernsteinLLC}
La conjecture de compatibilité de la correspondance de Langlands pour $G$ avec l'induction parabolique (conjecture \ref{conjindpar}) est équivalente à ce que pour tout sous-groupe de Levi $\widehat{M}$ de $\widehat{G}$, pour tout caractère infinitésimal discret $\lambda : W_F \longrightarrow \LL{M}$, on a : $$\Pi_{\lambda}^{+}(G) = \bigsqcup_{\widehat{L} \in \mathcal{L}(\widehat{G})} \bigsqcup_{\substack{ \varphi \in \Phi(L)\\ (\lambda_{\varphi})_{\widehat{G}}=(\lambda)_{\widehat{G}}}} \bigsqcup_{\sigma \in \Pi_{\varphi}(L)_{\cusp}} \mathcal{JH}(i_{LU}^{G}(\sigma)).$$ 
\end{prop}

La correspondance de Langlands prédit un paramétrage d'un paquet $\Pi_\phi(G)$ par les représentations irréductibles du groupe fini $\mathcal{S}_{\phi}^{G}$, qui est essentiellement le groupe des composantes du centralisateur dans $\widehat{G}$ de l'image $\phi$. Notre but dans ce qui suit va être, en se donnant un caractère infinitésimal discret $\lambda : W_F \longrightarrow \widehat{M}$, de déterminer « explicitement » cette décomposition, c'est-à-dire de préciser les sous-groupes de Levi, les paramètres de Langlands qui apparaissent en terme de paramètre de Langlands enrichis. 

\subsection{Conjecture sur les paramètres de Langlands des représentations supercuspidales}

Dans cette section, comme son titre l'indique, nous conjecturons la propriété qui caractérise les paramètres de Langlands enrichis des représentations supercuspidales. Pour commencer, remarquons les faits et les tautologies suivants. La correspondance de Langlands prédit qu'une représentation (essentiellement) de la série discrète $\sigma$ de $G$ admet un paramètre de Langlands $\phi_{\sigma} \in \Phi(G)$ discret, c'est-à-dire dont l'image ne se factorise dans un sous-groupe de Levi propre de $\widehat{G}$. Puisque cette caractérisation ne porte que sur le paramètre de Langlands $\phi_{\sigma}$, tous les éléments du $L$-paquet $\Pi_{\phi_{\sigma}}(G)$ sont des représentations de la série discrète. Le même phénomène a lieu quand on remplace série discrète par tempérée. En revanche, on sait (et on a vu) qu'il y a des paquets qui contiennent des représentations supercuspidales et des non-supercuspidales. Ainsi, si on veut caractériser les paramètres de Langlands $\varphi$ des représentations supercuspidales, on devra nécessairement prendre en compte les représentations irréductibles de $\Irr(\mathcal{S}_{\varphi}^{G})$. Dans ce qui suit on énonce une conjecture concernant les paramètres de Langlands des représentations supercuspidales. Puisque qu'une représentation supercuspidale est en particulier une représentation essentiellement de la série discrète, son paramètre est discret.\\

Soient $\varphi : W_F' \longrightarrow \widehat{G}$ un paramètre de Langlands de $G$ discret et $\varepsilon$ une représentation irréductible de $\mathcal{S}_{\varphi}^{G}$. Définissons $$H_{\varphi}^{G}=Z_{\widehat{G}}(\restriction{\varphi}{W_F}).$$ On a l'égalité remarquable suivante \begin{align*}
Z_{\widehat{G}}(\varphi) =& Z_{\widehat{G}}(\restriction{\varphi}{W_F},\restriction{\varphi}{\SL_2(\C)}) \\
=& Z_{Z_{\widehat{G}}(\restriction{\varphi}{W_F})}(\restriction{\varphi}{\SL_2(\C)}) \\
=& Z_{H_{\varphi}^{G}} (\restriction{\varphi}{\SL_2(\C)}).
\end{align*} Ainsi, $A_{\widehat{G}}(\varphi)=A_{H_{\varphi}^{G}}(\restriction{\varphi}{\SL_2(\C)})$ et $A_{{(H_{\varphi}^{G})}^{\circ}}(\restriction{\varphi}{\SL_2(\C)})$ est un sous-groupe distingué de ce dernier. Considérons le diagramme suivant 

$$
\begin{tikzcd}
A_{H_{\varphi}^{G}}(\restriction{\varphi}{\SL_2(\C)}) \ar[r,twoheadrightarrow] & \mathcal{S}_{\varphi}^{G}   \\
A_{{(H_{\varphi}^{G})}^{\circ}}(\restriction{\varphi}{\SL_2(\C)}) \ar[u,hook]  & \\
\end{tikzcd}
$$

Notons $\widetilde{\varepsilon}$ la représentation irréductible de $A_{H_{\varphi}^{G}}(\restriction{\varphi}{\SL_2(\C)})$ qui est la composée de $\varepsilon$ par la surjection $A_{H_{\varphi}^{G}}(\restriction{\varphi}{\SL_2(\C)}) \twoheadrightarrow \mathcal{S}_{\varphi}^{G}$. D'après \cite[10.1.1]{Kottwitz:1984}, $H_{\varphi}^{G}$ est un groupe réductif et un théorème de Kostant \cite[3.7.3 \& 3.7.23]{Chriss:2010} montre qu'en notant $\displaystyle u_{\varphi}=\varphi\left(1,\left(\begin{smallmatrix} 1 & 1 \\  0 & 1\end{smallmatrix} \right)\right) \in {(H_{\varphi}^{G})}^{\circ}$, alors $Z_{H_{\varphi}^{G}}(\restriction{\varphi}{\SL_2(\C)})$ est un sous-groupe réductif maximal de $Z_{H_{\varphi}^{G}}(u_{\varphi})$. En particulier, on a : $$A_{H_{\varphi}^{G}}(\restriction{\varphi}{\SL_2(\C)}) \simeq A_{H_{\varphi}^{G}}(u_{\varphi}).$$ 

\begin{defi}\phantomsection\label{defcusp}
Soient $\varphi \in \Phi(G)$ un paramètre discret de $G$, $\varepsilon \in \Irr(\mathcal{S}_{\varphi}^{G})$ et $\widetilde{\varepsilon}$ la représentation irréductible de $A_{\widehat{G}}(\varphi) \simeq A_{H_{\varphi}^{G}}(\restriction{\varphi}{\SL_2(\C)}) \simeq A_{H_{\varphi}^{G}}(u_{\varphi})$ décrit ci-dessus. On dit que \begin{itemize}
\item $\varepsilon$ est \emph{cuspidale} lorsque toutes les représentations irréductibles de $ A_{(H_{\varphi}^{G})^{\circ}}(u_{\varphi})$ apparaissant dans la restriction de $\widetilde{\varepsilon}$ à $A_{(H_{\varphi}^{G})^{\circ}}(u_{\varphi})$ sont cuspidales au sens de Lusztig (définition \ref{defsyslocusp}) et on notera $\Irr(\mathcal{S}_{\varphi}^{G})_{\cusp}$ l'ensemble des représentations irréductibles cuspidales de $\mathcal{S}_{\varphi}^{G}$ ;
\item $\varphi$ un paramètre de Langlands \emph{cuspidal} de $G$ lorsque $\Irr(\mathcal{S}_{\varphi}^{G})_{\cusp}$ est non vide.\end{itemize}
\end{defi}

\begin{conj}\phantomsection\label{conjcusp}
Soit $\varphi \in \Phi(G)$ un paramètre de Langlands de $G$. Le $L$-paquet $\Pi_{\varphi}(G)$ contient des représentations supercuspidales de $G$, si et seulement si, $\varphi$ est un paramètre de Langlands cuspidal. Si tel est le cas, les représentations supercuspidales de $\Pi_{\varphi}(G)$ sont paramétrées par $\Irr(\mathcal{S}_{\varphi}^{G})_{\cusp}$. Autrement dit, il existe une bijection : $$\Pi_{\varphi}(G)_{\cusp} \leftrightarrow \Irr(\mathcal{S}_{\varphi}^{G})_{\cusp}.$$
\end{conj}

\begin{rema}
La condition de discrétion de $\varphi$ est nécessaire. En effet, prenons $G=\GL_2(F)$ et soient $\chi_1, \chi_2$ deux caractères distincts de $W_F$. Définissons $\varphi : W_F' \longrightarrow \GL_2(\C)$, pour tout $(w,x) \in W_F \times \SL_2(\C)$, par $\varphi(w,x)=\diag(\chi_1(w),\chi_2(w))$. Alors $\varphi$ n'est pas discret (son image se factorise à travers un tore maximal de $\GL_2(\C)$) et $H_{\varphi}^{G} \simeq (\C^{\times})^{2}$. La représentation triviale de $\mathcal{S}_{\varphi}^{G}=A_{H_{\varphi}}(1)=\{1\}$ est cuspidale (pour l'orbite unipotente $\mathcal{C}_{1}^{H_{\varphi}^{G}})$ mais $\varphi$ n'est pas le paramètre d'une représentation supercuspidale de $\GL_2(F)$.
\end{rema}

On se propose de décrire la forme des paramètres de Langlands cuspidaux dans les cas du groupe linéaire, symplectique et spécial orthogonal. Pour décrire la forme des paramètres et le calcul des divers centralisateurs, on se réfère à \cite{Gan:2012}. On notera $I^{\O}$ (resp. $I^{\S}$) un certain ensemble de représentations irréductibles de $W_F$ de type orthogonal (resp. symplectique) et $S_d$ la représentation irréductible de dimension $d$ de $\SL_2(\C)$.

\begin{prop}\phantomsection\label{verifparamcusp}
Pour un groupe linéaire ou un groupe classique déployé, les paramètres de Langlands cuspidaux (définition \ref{defcusp}) sont : \begin{itemize}
\item pour $\GL_n(F)$, $$\varphi : W_F \longrightarrow \GL_n(\C), \,\, \text{irréductible (ou de façon équivalente discret)} \,\, ;$$
\item pour $\SO_{2n+1}(F)$, $$\Std_{G} \circ \varphi=\bigoplus_{\pi \in I^{\O}} \bigoplus_{a=1}^{d_{\pi}} \pi \boxtimes  S_{2a} \bigoplus_{\pi \in I^{\S}} \bigoplus_{a=1}^{d_{\pi}} \pi \boxtimes  S_{2a-1}, \,\, \forall \pi \in I^{\O}, d_{\pi} \in \N, \,\, \forall \pi \in I^{\S}, d_{\pi} \in \N^{*} ;$$
\item pour $\Sp_{2n}(F)$ ou $\SO_{2n}(F)$, $$\Std_{G} \circ \varphi=\bigoplus_{\pi \in I^{\S}} \bigoplus_{a=1}^{d_{\pi}} \pi \boxtimes  S_{2a} \bigoplus_{\pi \in I^{\O}} \bigoplus_{a=1}^{d_{\pi}} \pi \boxtimes  S_{2a-1}, \,\, \forall \pi \in I^{\O}, d_{\pi} \in \N^{*}, \,\, \forall \pi \in I^{\S}, d_{\pi} \in \N.$$
\end{itemize}
De plus, d'après la construction de la correspondance de Langlands pour $\GL_n$ par Harris-Taylor et Henniart et le théorème \ref{thmmoeglin} de M\oe glin qui décrit les paramètres de Langlands des représentations supercuspidales pour les groupes classiques, les représentations supercuspidales de $G$ sont paramétrées par $(\varphi,\varepsilon)$ avec $\varphi$ un paramètre de Langlands cuspidal de $G$ et $\varepsilon \in \Irr(\mathcal{S}_{\varphi}^{G})_{\cusp}$. Autrement dit, la conjecture \ref{conjcusp} est vraie lorsque $G$ est un groupe linéaire ou un groupe classique déployé.
\end{prop}

\begin{proof}
Soit $\varphi : W_F' \longrightarrow \GL_n(\C)$ un paramètre cuspidal de $G=\GL_{n}(F)$. Écrivons la décomposition en composante isotypique de la restriction à $W_F$ de $\varphi$, $$\restriction{\varphi}{W_F}=\bigoplus_{\pi \in I} \pi \otimes M_{\pi},$$ où $I$ est un ensemble fini de représention irréductible de $W_F$. D'où,  \begin{align*}
H_{\varphi}^{G} &= Z_{\GL_n(\C)}(\restriction{\varphi}{W_F})\\
     & \simeq \prod_{\pi \in I} \GL_{m_{\pi}}(\C),
\end{align*} où $m_{\pi}=\dim M_{\pi}$. On peut donc écrire $u_{\varphi}=(u_{\pi})_{\pi \in I}$, où $u_{\pi} \in \GL_{m_{\pi}}(\C)$ et se ramener à un étudier un seul facteur pour l'existence d'une paire cuspidale. Or, d'après la classification des paires cuspidales (voir table \ref{pairecusp}), ceci est équivalent au fait que pour tout $\pi \in I, \, m_{\pi}=1$ et $u_{\pi}=1$. D'où $Z_{\GL_n(\C)}(\varphi) = \prod_{\pi \in I} \GL_1(\C)$. À présent, la discrétion du paramètre impose que $Z(\GL_n(\C))^{\circ} = \GL_1(\C)$ soit un tore maximal de $Z_{\GL_n(\C)}(\varphi) $, par suite $I$ est un singleton et $Z_{\GL_n(\C)}(\varphi) =\GL_{1}(\C)$.\\
 
Ceci montre que, $\varphi$ est un paramètre cuspidal de $\GL_n$, si et seulement si, $\varphi$ est un paramètre de Langlands discret trivial sur $\SL_2(\C)$.\\

Soit $\varphi : W_F' \longrightarrow \Sp_{2n}(\C)$ un paramètre cuspidal de $G=\SO_{2n+1}(F)$. Écrivons la décomposition en composantes isotypiques de la restriction à $W_F$ de $\varphi$, $$\restriction{\varphi}{W_F}=\bigoplus_{\pi \in I^{\O}} \pi \otimes M_{\pi} \bigoplus_{\pi \in I^{\S}} \pi \otimes M_{\pi} \bigoplus_{\pi \in I^{\GL}} (\pi \oplus \pi^{\vee}) \otimes M_{\pi}.$$  Ainsi, $$H_{\varphi}^{G} \simeq \prod_{\pi \in I^{\O}} \Sp_{m_{\pi}}(\C) \times \prod_{\pi \in I^{\S}} \O_{m_{\pi}}(\C) \times \prod_{\pi \in I^{\GL}} \GL_{m_{\pi}}(\C). $$ D'après la classification des paires cuspidales (table \ref{pairecusp}), on obtient les conditions suivantes : \begin{itemize}
\item pour $\pi \in I^{\O}, \, m_{\pi}=d_{\pi}(d_{\pi}+1), \,\, d_{\pi} \in \N$ ;
\item pour $\pi \in I^{\S}, \, m_{\pi}=d_{\pi}^{2}, \,\, d_{\pi} \in \N^{*}$ ;
\item pour $\pi \in I^{\GL}, \, m_{\pi}=1$.
\end{itemize} De plus, $\varphi$ est de la forme (on rappelle que $S_d$ désigne la représentation irréductible de $\SL_2(\C)$ de dimension $d$), $$\Std_{G} \circ \varphi=\bigoplus_{\pi \in I^{\O}} \bigoplus_{a=1}^{d_{\pi}} \pi \boxtimes  S_{2a} \bigoplus_{\pi \in I^{\S}} \bigoplus_{a=1}^{d_{\pi}} \pi \boxtimes  S_{2a-1} \bigoplus_{\pi \in I^{\GL}} (\pi \oplus \pi^{\vee}).$$ Par suite, on trouve $$Z_{\Sp_{2n}(\C)}(\varphi)=\prod_{\pi \in I^{\O}}\prod_{a=1}^{d_{\pi}} (\Z/2\Z) \times \prod_{\pi \in I^{\S}}\prod_{a=1}^{d_{\pi}} (\Z/2\Z) \times \prod_{\pi \in I^{\GL}} \GL_1(\C).$$ Puisque $\varphi$ est discret, $Z_{\Sp_{2n}(\C)}(\varphi)$ ne contient aucun tore non trivial. Par suite, $I^{\GL}$ est vide et donc $$Z_{\Sp_{2n}(\C)}(\varphi)=\prod_{\pi \in I^{\O} \sqcup I^{\S}} (\Z/2\Z)^{d_{\pi}},$$ et $$\Std_{G} \circ \varphi=\bigoplus_{\pi \in I^{\O}} \bigoplus_{a=1}^{d_{\pi}}\pi \boxtimes  S_{2a} \bigoplus_{\pi \in I^{\S}} \bigoplus_{a=1}^{d_{\pi}} \pi \boxtimes  S_{2a-1}.$$ Écrivons le bloc de Jordan correspondant à ce paramètre de Langlands $$\Jord(\varphi)=\left\{(\pi,2),\ldots,(\pi,2d_{\pi}-2),(\pi,2d_{\pi}), \pi \in I^{\O}\right\} \sqcup \left\{(\pi,1),\ldots,(\pi,2d_{\pi}-3),(\pi,2d_{\pi}-1), \pi \in I^{\S}\right\}.$$ Ainsi, $\Jord(\varphi)$ est sans trou et de plus, nous avons décrit à la suite de la table \ref{pairecusp} les caractères cuspidaux $\varepsilon_{d}^{\S}, \,\, {\varepsilon_{d}^{\O}} '$ et ${\varepsilon_{d}^{\O}}''$. Ceux-ci correspondent exactement aux caractères alternés du théorème \ref{thmmoeglin} de M\oe glin . Réciproquement, un bloc de Jordan sans trou correspond à ces partitions d'unipotents dans $\Sp_{m_{\pi}}(\C)$ et $\SO_{m_{\pi}}(\C)$ et les caractères alternées à ces caractères cuspidaux.\\

Pour un groupe symplectique ou un groupe spécial orthogonal pair, c'est essentiellement la même démonstration.
\end{proof}

Supposons temporairement que $G$ est un groupe réductif connexe défini et quasi-déployé sur $F$. La définition \ref{defcusp} et la conjecture \ref{conjcusp} ont toujours un sens dans ce contexte.

\begin{lemm}
Soient $M$ un sous-groupe de Levi de $G$ et $\lambda : W_F \longrightarrow \LL{M}$ un paramètre de Langlands de $M$ discret. Alors, $Z_{\widehat{M}}(\lambda)^{\circ}=\left( Z_{\widehat{M}}^{\Gamma_F}\right)^{\circ}$. En particulier,  $Z_{\widehat{M}}(\lambda)^{\circ}$ est un tore. 
\end{lemm}

\begin{proof}
D'après \cite[10.3.1]{Kottwitz:1984}, $Z_{\widehat{M}}(\lambda)^{\circ} \subseteq Z_{\widehat{M}}$. De plus, $Z_{\widehat{M}}(\lambda) \cap Z_{\widehat{M}}= Z_{\widehat{M}}^{\Gamma_F}$, d'où $Z_{\widehat{M}}(\lambda)^{\circ} \subseteq \left( Z_{\widehat{M}}^{\Gamma_F}\right)^{\circ}$. L'autre inclusion étant évidente, il y a égalité.
\end{proof}

\begin{prop}\phantomsection\label{cocaracusp}
On reprend les notations du lemme précédent. Si $\phi \in \Phi(M)$ est un paramètre de Langlands de $M$ de caractère infinitésimal $(\lambda)_{\widehat{M}}$, alors $(\phi)_{\widehat{M}}=(\lambda)_{\widehat{M}}$. Ceci implique que le paquet infinitésimal et $L$-paquet défini par $\lambda$ coïncident : $$\Pi_{\lambda}^{+}(M)=\Pi_{\lambda}(M).$$  De plus, toutes les représentations irréductibles de $\mathcal{S}_{\lambda}(M)$ sont cuspidales, c'est-à-dire $$\Irr(\mathcal{S}_{\lambda}^{M})=\Irr(\mathcal{S}_{\lambda}^{M})_{\cusp}.$$
\end{prop}

\begin{proof}
Soit $\phi : W_F' \longrightarrow \LL{M}$ un paramètre de Langlands de $M$ de caractère infinitésimal $(\lambda)_{\widehat{M}}$, quitte à conjuger par élement de $\widehat{M}$, ceci signifie que pour tout $w \in W_F, \,\, \phi(w,d_w)=\lambda(w)$. \\ Par connexité, $\phi(\SL_2(\C)) \subseteq Z_{\widehat{M}}(\restriction{\phi}{W_F})^{\circ}$. L'image par $\phi$ de $T_{\SL_2}=\left\{ \left(\begin{smallmatrix} t &  \\  & t^{-1}\end{smallmatrix} \right), t \in \C^{\times} \right\}$, tore maximal de $\SL_2(\C)$, est un tore (éventuellement trivial) de $Z_{\widehat{M}}(\restriction{\phi}{W_F})^{\circ}$. Soit $A$ un tore maximal de $Z_{\widehat{M}}(\restriction{\phi}{W_F})^{\circ}$ contenant $\phi(T_{\SL_2})$. Notons $\LL{L}=Z_{\LL{M}}(A)$ ; d'après \cite[3.6]{Borel:1979} c'est un sous-groupe de Levi de $\LL{M}$ qui contient minimalement l'image de $\restriction{\phi}{W_F}$. Puisque pour tout $w \in W_F, \,\, \phi(1,d_w) \in A \subset \widehat{L}$, on a pour $w \in W_F, \lambda(w)=\phi(w,d_w) \in \LL{L}$. Par discrétion de $\lambda$, $\LL{L}=\LL{M}$ et $\restriction{\phi}{W_F}$ est un paramètre discret de $M$. D'après ce qui précède, $Z_{\widehat{M}}(\restriction{\phi}{W_F})^{\circ}$ est un tore. L'image de $\SL_2(\C)$ par $\phi$ est contenue dans un tore, cette image est donc triviale. D'où $\phi=\lambda$.\\ Soient $\varepsilon \in \Irr(\mathcal{S}_{\lambda}^{M})$ et $\widetilde{\varepsilon} \in \Irr(A_{\widehat{M}}(\lambda))$. Puisque ${(H_{\lambda}^{M})}^{\circ}=Z_{\widehat{M}}(\lambda)^{\circ}$ est un tore, la seule sous-représentation irréductible de la restriction de $\widetilde{\varepsilon}$ à $A_{{(H_{\lambda}^{M})}^{\circ}}(1)$ est la représentation triviale et la paire $(\mathcal{C}^{{(H_{\lambda}^{M})}^{\circ}}_{1},\triv)$ est automatiquement cuspidale. D'où $\Irr(\mathcal{S}_{\lambda}^{M})=\Irr(\mathcal{S}_{\lambda}^{M})_{\cusp}$.
\end{proof}

\begin{rema} Cette proposition montre que la conjecture \ref{conjcusp} sur le paramétrage des représentations supercuspidales est compatible avec une propriété de la correspondance de Langlands, à savoir que le $L$-paquet d'un paramètre discret $\lambda : W_F \longrightarrow \LL{M}$ n'est constitué que de représentations supercuspidales de $M$. Notons par ailleurs, qu'Heiermann dans \cite{Heiermann:2006} construit sous diverses hypothèses les paramètres de séries discrètes non-cuspidales à partir du paramètre de leurs support cuspidal. Sa construction montre que nécessairement la restriction à $\SL_2(\C)$ est non triviale pour ces séries discrètes. Ainsi un paramètre discret trivial sur $\SL_2(\C)$ ne peut contenir que des représentations supercuspidales. Dans \cite[5.6.1]{Haines:2014}, Haines prouve, en supposant (essentiellement) la conjecture de compatibilité \ref{conjindpar}, que $\Pi_{\lambda}^{+}(G)$ ne contient que des représentations supercuspidales, si et seulement $G=M$. \end{rema}

Nous introduisons la définition suivante afin d'éviter la répétition des phrases du type : «une supercuspidale dont le paramètre de Langlands n'est pas trivial sur le facteur $\SL_2$». 

\begin{defi}
Soit $\sigma$ une représentation irréductible supercuspidale de $G$. On dira que $\sigma$ est \emph{ordinaire} lorsque son paramètre de Langlands est de la forme $\lambda : W_F \longrightarrow {}^L G$, c'est-à-dire trivial sur le facteur $\SL_2(\C)$. Dans le cas contraire, on dira que $\sigma$ est \emph{orpheline}. 
\end{defi}

\begin{rema}
Par conséquent, les représentations supercuspidales ordinaires sont (conjecturalement) dans des $L$-paquets ne contenant que des supercuspidales et les représentations supercuspidales orphelines sont dans des $L$-paquets pouvant contenir des séries discrètes qui ne sont pas supercuspidales. La conjecture \ref{conjcusp} caractérise essentiellement les représentations supercuspidales orphelines.
\end{rema}

On suppose que $G$ est un groupe réductif connexe défini et déployé sur $F$.

\begin{defi}
On appelle \emph{donnée cuspidale enrichie de $\widehat{G}$} un triplet $(\widehat{L},\varphi,\varepsilon)$ formé d'un sous-groupe de Levi $\widehat{L}$ de $\widehat{G}$, d'un paramètre de Langlands cuspidal $\varphi: W_F' \longrightarrow \widehat{L}$ de $L$ et d'une représentation irréductible cuspidale $\varepsilon$ de $\mathcal{S}_{\varphi}^{L}$. On définit les relations d'équivalences suivantes sur l'ensemble des données cuspidales enrichies de $\widehat{G}$. Soient $(\widehat{L}_1,\varphi_1,\varepsilon_1)$ et $(\widehat{L}_2,\varphi_2, \varepsilon_2)$ deux données cuspidales enrichies de $\widehat{G}$. On dit qu'elles sont : \begin{enumerate}[label=(\roman*)]
\item équivalentes s'il existe $g \in \widehat{G}$ tel que : ${}^g  \widehat{L}_1 =\widehat{L}_2, \,\, \varphi_2= {}^g \varphi_1  \,\, \text{et} \,\, \varepsilon_2 \simeq \varepsilon_1^{g}$ ;
\item inertiellement équivalentes s'il existe $g \in \widehat{G}$ et $\chi_2  \in \mathcal{X}(\widehat{L_2})$ tels que : ${}^g \widehat{L}_1=\widehat{L}_2, \,\, \varphi_2 ={}^g \varphi_1 \chi_2 \,\, \text{et} \,\, \varepsilon_2 \simeq \varepsilon_1^{g}$.
\end{enumerate}
\end{defi}

On appelle \emph{triplet cuspidal} (resp. \emph{triplet inertiel}) de $\widehat{G}$ une classe d'équivalence pour la relation (i) (resp. (ii)) et on notera $\Omega_{e}^{\st}(G)$ (resp. $\mathcal{B_{e}}^{\st}(G)$) l'ensemble des classes d'équivalence pour la relation (i) (resp. (ii)). De plus, on notera $(\widehat{L},\varphi,\varepsilon) \in \Omega_{e}^{\st}(G)$ le triplet cuspidal (resp. $[\widehat{L},\varphi,\varepsilon] \in \mathcal{B}_{e}^{\st}(G)$ le triplet inertiel) défini à partir de $\widehat{L}$, $\varphi$ et $\varepsilon$.\\

Comme pour $\Omega(G)$, l'ensemble $\Omega_{e}^{\st}(G)$ est muni d'une structure de variété algébrique dont les composantes connexes sont indexées par $\mathcal{B}_{e}^{\st}(G)$ et sont des quotients de tores complexes par l'action de groupes finis.

Soit $\cj=[\widehat{L},\varphi,\varepsilon] \in \mathcal{B}_{e}^{\st}(G)$ un triplet inertiel de $\widehat{G}$. On définit les objets suivants : \begin{itemize}
\item un tore algébrique $\mathcal{T}_{\cjp}=\left\{(\varphi \chi)_{\widehat{L}}, \, \chi \in \mathcal{X}(\widehat{L})\right\} \simeq \mathcal{X}(\widehat{L})/\mathcal{X}(\widehat{L})(\varphi)$, où $\mathcal{X}(\widehat{L})(\varphi)=\{\chi \in \mathcal{X}(\widehat{L}) \mid (\varphi)_{\widehat{L}}=(\varphi\chi)_{\widehat{L}} \}$ ;
\item un groupe fini $\mathcal{W}_{\cjp}=\left\{ w \in N_{\widehat{G}}(\widehat{L}) \mid \exists \chi \in  \mathcal{X}(\widehat{L}), ({}^w \varphi)_{\widehat{L}}=(\varphi \chi)_{\widehat{L}} , \,\, \varepsilon^w \simeq \varepsilon  \right\}/\widehat{L}.$
\end{itemize}

Le groupe fini $\mathcal{W}_{\cjp}$ agit sur $\mathcal{T}_{\cjp}$ et le quotient $\mathcal{T}_{\cjp}/\mathcal{W}_{\cjp}$ s'identifie aux triplets cuspidaux de $\widehat{G}$ dans $\cj$, c'est-à-dire à la fibre au-dessus de $\cj$ par la projection naturelle $\Omega_{e}^{\st}(G) \twoheadrightarrow \mathcal{B}_{e}^{\st}(G)$. Ainsi, $\Omega_{e}^{\st}(G)$ est muni d'une structure de variété algébrique dont les composantes connexes, indexées par $\cjp \in \mathcal{B}_{e}^{\st}(G)$, sont des quotients de tores algébriques complexes par l'action de groupes finis. Ceci s'écrit $$\Omega_{e}^{\st}(G) = \bigsqcup_{\cjp \in \mathcal{B}_{e}^{\st}(G)} \mathcal{T}_{\cjp}/\mathcal{W}_{\cjp}.$$

\begin{defi}
On appelle \emph{centre de Bernstein dual} de $G$ et on note $\mathfrak{Z}_{e}^{\st}(G)$ l'anneau des fonctions régulières sur $\Omega_{e}^{\st}(G)$ : $$\mathfrak{Z}_{e}^{\st}(G)=\C[\Omega_{e}^{\st}(G)].$$
\end{defi} 

Contrairement à la décomposition de la section précédente, pour le moment, on ne sait pas associer à tout paramètre de Langlands enrichi $(\phi,\eta) \in \Phi_{e}(G)$ un triplet cuspidal de $\widehat{G}$. Ce sera l'objet de la section suivante. La conjecture ci-dessous établit le lien entre le centre de Bernstein et le centre de Bernstein dual.
\begin{conj}\phantomsection\label{compatibiliteaction}
Supposons que pour tous les sous-groupes de Levi $L$ de $G$, les paramètres de Langlands enrichis cuspidaux de $L$ sont les paramètres des représentations irréductibles supercuspidales de $L$. Soient $L$ un sous-groupe de Levi de $G$, $\varphi : W_F' \longrightarrow \widehat{L}$ un paramètre de Langlands cuspidal de $L$, $\sigma \in \Pi_\varphi(L)_{\cusp}$ un représentation irréductible supercuspidale paramétrée par $\varepsilon \in \Irr(\mathcal{S}_{\varphi}^{L})_{\cusp}$. On note $\mathfrak{s}=[L,\sigma] \in \mathcal{B}(G)$ et $\cj=[\widehat{L},\varphi,\varepsilon]\in \mathcal{B}_{e}^{\st}(G)$. La correspondance de Langlands induit des isomorphismes : $$\begin{array}{ccc}
T_{\mathfrak{s}} & \longrightarrow & \mathcal{T}_{\cjp} \\
t & \longmapsto & \widehat{t}
\end{array}, \quad \begin{array}{ccc}
W_{\mathfrak{s}} & \longrightarrow & \mathcal{W}_{\cjp} \\
w & \longmapsto & \widehat{w}
\end{array},$$
tels que pour tout $t \in T_{\mathfrak{s}}, \,\, w \in W_{\mathfrak{s}}$ : $$\widehat{w \cdot t}=\widehat{w} \cdot \widehat{t}.$$
\end{conj}
 
Cette conjecture énonce principalement que la correspondance de Langlands pour les représentation supercuspidales des sous-groupes de Levi de $G$ produit un isomorphisme $$\Omega(G) \simeq \Omega_{e}^{\st}(G),$$ et donc $$\mathfrak{Z}(G) \simeq \mathfrak{Z}_{e}^{\st}(G),$$ le membre de droite étant défini en terme de paramètres de Langlands. Par ailleurs, si $\mathfrak{s}=[L,\sigma] \in \mathcal{B}(G)$, l'isomorphisme $T_{\mathfrak{s}} \simeq \mathcal{T}_{\cjp}$ résulte de la correspondance de Langlands pour les caractères et la compatibilité de la correspondance de Langlands pour $L$ avec les caractères, c'est-à-dire, pour tout caractère $\chi$ de $L$, $\rec_L(\sigma \otimes \chi)=\rec_{L}(\sigma) \chi$.\\ La conjecture \ref{compatibiliteaction} sera prouvée pour les groupes classiques au théorème \ref{verifcompatibilite}.

\subsection{Les groupes de Weil-Deligne}\phantomsection\label{groupesweildeligne}
Commençons par un rappel sur les deux notions de groupes de Weil-Deligne. À l'origine, Deligne a introduit dans \cite[8.3.6]{Deligne:1973} le groupe de Weil-Deligne comme le produit semi-direct $$WD_F=\C \rtimes W_F,$$ l'action de $W_F$ sur $\C$ étant définie de la façon suivante : pour tout $w \in W_F, z \in \C, \,\, wzw^{-1}=|w| z$. Ainsi, pour tout $(z,w), \, (z',w') \in WD_F$, $$(z,w)(z',w')=(z+|w| z',ww').$$ On reprend les notations précédentes et on note $\widehat{\mathfrak{g}}$ l'algèbre de Lie de $\widehat{G}$. Dans ce contexte, un paramètre de Langlands de $G$ est la donnée d'un couple $(\lambda,N)$ avec \begin{itemize}
\item $\lambda: W_F \longrightarrow \widehat{G}$, un morphisme continu et dont l'image est constitué d'éléments semi-simples ;
\item $N \in \widehat{\mathfrak{g}}$, un élément nilpotent tel que pour tout $w \in W_F, \,\, \Ad(\lambda(w))N=|w| N$.
\end{itemize} La relation d'équivalence sur ces couples est toujours la conjugaison par $\widehat{G}$, c'est-à-dire $(\lambda,N)$ est équivalent à $(\lambda',N')$, si et seulement si, il existe $g \in \widehat{G}$ tel que $\lambda'={}^{g}\lambda$ et $N'=\Ad(g)N$. Pour distinguer les deux notions de paramètres de Langlands, on appelera les couples définis précédemment des paramètres de Langlands originels de $G$.\\ 

Rappelons que pour tout $w \in W_F$, on note $d_w=\left( \begin{smallmatrix} |w|^{1/2} & 0 \\0 &  |w|^{-1/2} \end{smallmatrix} \right) \in \SL_2(\C)$. À un paramètre de Langlands de $G$ de la forme $\phi : W_F' \longrightarrow \widehat{G}$, on définit un paramètre de Langlands originel $(\lambda_{\phi},N_{\phi})$ de $G$ de la façon suivante : pour tout $w \in W_F$, $$\lambda_{\phi}(w)=\phi(w,d_w) \quad \mathrm{et} \quad N_{\phi}=d\restriction{\phi}{\SL_2(\C)}\left( \begin{smallmatrix} 0 & 1 \\  0 & 0\end{smallmatrix} \right).$$
Réciproquement, à un paramètre de Langlands originel $(\lambda,N)$ de $G$ on peut associer un paramètre de Langlands de $G$ de la forme $\phi : W_F' \longrightarrow \widehat{G}$. Pour cela, on note $K_{\lambda}^{G}=Z_{\widehat{G}}(\restriction{\lambda}{I_F})$ le centralisateur dans $\widehat{G}$ de l'image de l'inertie par $\lambda$, $\mathfrak{k}_{\lambda}^{G} = \{ x \in \widehat{\mathfrak{g}} \mid \forall w \in I_F, \, \Ad(\lambda(w))x=x\}$ son algèbre de Lie et $f_{\lambda}=\lambda(\Fr)$. L'élément semi-simple $f_{\lambda}$ normalise $\lambda(I_F)$, il agit donc par adjonction sur $\mathfrak{k}_{\lambda}^{G}$. Pour $\alpha \in \C$, notons $\mathfrak{k}_{\lambda}^{G}(\alpha)$ l'espace propre relatif à l'action de $f_{\lambda}$ dans $\mathfrak{k}_{\lambda}^{G}$ associé à la valeur propre $\alpha$. La condition que $N$ doit vérifier est donc équivalente à $N \in \mathfrak{k}_{\lambda}^{G}(q^{-1})$. L'élément $N \in \mathfrak{k}_{\lambda}^{G} $ étant nilpotent, d'après le théorème de Jacobson-Morozov, il existe un $\mathfrak{sl}_2$-triplet $(x,y,z)$, c'est-à-dire des éléments $x,y,z \in \mathfrak{k}_{\lambda}^{G}$ vérifiant $$[z,x]=2x, \,\, [z,y]=-2y, \,\, [x,y]=z,$$ avec $x=N$. De plus, d'après un raffinement de Kostant du théorème de Jacobson-Morozov (voir \cite[lemma 2.1]{Gross:2010}), on peut supposer que le $\mathfrak{sl}_2$-triplet $(x,y,z)$ vérifie $$x \in \mathfrak{k}_{\lambda}^{G} (q^{-1}), \quad z \in \mathfrak{k}_{\lambda}^{G} (1) \quad \text{et} \quad y \in\mathfrak{k}_{\lambda}^{G} (q).$$ Soit $\gamma : \SL_2(\C) \longrightarrow \left(K_{\lambda}^{G} \right) ^{\circ}$ le morphisme défini par le $\mathfrak{sl}_{2}$-triplet $(x,y,z)$. En particulier, on a $\gamma \left( \begin{smallmatrix}1 & 1 \\ 0 & 1 \end{smallmatrix} \right)=\exp(N)$. On définit un morphisme $\chi_{\phi} : W_F \longrightarrow \left(K_{\lambda}^{G} \right) ^{\circ}$ pour tout $w \in W_F$, par $$\chi_{\phi}(w)=\gamma(d_w)^{-1}.$$ Il est clair que $\chi_{\phi}$ est entièrement déterminé par l'image de $\Fr$, c'est-à-dire par $s_{\gamma}=\gamma \left( \begin{smallmatrix} q^{1/2} & 0 \\  0 & q^{-1/2} \end{smallmatrix} \right)=\exp\left(\frac{\log q}{2}Z \right).$ Puisque $\frac{\log q}{2}Z \in \mathfrak{j}_{\lambda}(1)$, $s_{\lambda}=\lambda(\Fr)$ commute à $s_{\gamma}$. De plus, on a les relations $$\Ad(s_{\gamma})x=q x, \,\, \Ad(s_{\gamma})Y=q^{-1}Y.$$ Ceci montre que $x,Y$ et $Z$ sont fixés par l'action adjointe des éléments de $(\lambda \chi_{\phi})(W_F)$, donc $\gamma(\SL_2(\C))$ est contenu dans $Z_{\widehat{G}}((\lambda \chi_{\phi})(W_F))$. Ceci permet de définir $\phi : W_F' \longrightarrow \widehat{G}$ pour tout $w \in W_F, x \in \SL_2(\C)$, par $$\phi(w,x)=\lambda(w) \chi_{\phi}(w) \gamma(x).$$ On obtient ainsi une bijection entre les classes de paramètres de Langlands de $G$ et les classes de paramètres de Langlands originels de $G$ (voir \cite[proposition 2.2]{Gross:2010}).\\ 

Dans ce qui suit nous montrons que des paramètres de Langlands pour $W_F'$ et pour $WD_F$ associés, ont leurs groupes de composantes du centralisateur de leurs images isomorphes. D'après un théorème de Kostant, puisque $K_{\lambda}^{G}$ est réductif, $Z_{K_{\lambda}^{G}}(\gamma) $ est un sous-groupe réductif maximal de $Z_{K_{\lambda}^{G}}(N)$ et $ Z_{K_{\lambda}^{G}}(N) =Z_{K_{\lambda}^{G}}(\gamma) U$, où $U$ est le radical unipotent de $Z_{K_{\lambda}^{G}}(N)$. De plus, $N \in \mathfrak{k}_{\lambda}^{G}(q^{-1})$ et $f_{\lambda}$ agit par conjugaison sur $Z_{K_{\lambda}^{G}}(N)$. Le sous-groupe $U$ étant caractéristique on a $f_{\lambda} U f_{\lambda}^{-1}=U$ et d'après \cite[Theorem 18.3]{Humphreys:1975} $Z_{U}(f_{\lambda})$ est un groupe unipotent et connexe. Pour tout $k \in Z_{K_{\lambda}^{G}}(\gamma)$, l'élément $f_{\lambda} k f_{\lambda}^{-1}$ fixe le $\mathfrak{sl}_2$-triplet $(x,y,z)$ et cela montre que $Z_{K_{\lambda}^{G}}(\gamma)$ est stable par $f_{\lambda}$. Par suite, $Z_{\widehat{G}}(\lambda,N)=Z_{K_{\lambda}^{G}}(f_{\lambda},N)$ et $Z_{K_{\lambda}^{G}}(f_{\lambda},N)=Z_{K_{\lambda}^{G}}(f_{\lambda},\gamma) Z_{U}(f_{\lambda})$. D'où $$Z_{\widehat{G}}(\lambda,N)=Z_{\widehat{G}}(\phi) Z_U(f_{\lambda}) \quad \text{et} \quad A_{\widehat{G}}(\lambda,N) \simeq A_{\widehat{G}}(\phi).$$

Afin de pouvoir appliquer certain résultats (essentiellement ceux de Lusztig \cite{Lusztig:1988}, \cite{Lusztig:1995a} et \cite{Lusztig:1995a}) on souhaite remplacer $f_{\lambda}$ par un élément semi-simple de l'algèbre de Lie d'un certain groupe (qui n'est certainement pas $K_{\lambda}^{G}$ puisque en général $f_{\lambda} \not \in K_{\lambda}^{G}$). On note $f_{\lambda}=s_{\lambda} t_{\lambda}$ la décomposition de l'élément semi-simple $f_{\lambda}$ en produit d'un élément semi-simple hyperbolique $s_{\lambda}$ et elliptique $t_{\lambda}$. Ceci signifie par exemple que pour tout caractère rationnel $\chi$ d'un tore maximal qui contient $f_{\lambda}$, alors $\chi(s_{\lambda}) \in \R_{+}^{*}$ et $\chi(t_{\lambda}) \in \mathbf{S}^{1}$. Soit $N \in \widehat{\mathfrak{g}}$. Puisque $f_{\lambda}$ agit algébriquement sur $\mathfrak{k}_{\lambda}^{G}$, $N \in \mathfrak{k}_{\lambda}^{G}(q^{-1})$, si et seulement si, $\Ad(t_{\lambda})N=N$ et $\Ad(s_{\lambda})N=q^{-1}N$. Notons $$J_{\lambda}^{G}=Z_{\widehat{G}}(\restriction{\lambda}{I_F},t_{\lambda}) \,\, \text{et} \,\, \mathfrak{j}_{\lambda}^{G}= \{ x \in \widehat{\mathfrak{g}} \mid \forall w \in I_F, \, \Ad(\lambda(w))x=x \,\, \text{et} \,\, \Ad(t_{\lambda})x=x\}.$$ Puisque $f_{\lambda}$ normalise $\lambda(I_F)$, il normalise également $K_{\lambda}^{G}=Z_{\widehat{G}}(\restriction{\lambda}{I_F})$. Ainsi, il existe $n \in \N^{*}$ tel que $f_{\lambda}^{n} \in K_{\lambda}^{G}$ et en particulier $s_{\lambda}^{n} \in K_{\lambda}^{G} $.  L'élément semi-simple $s_{\lambda}$ commute à $t_{\lambda}$ et la composante neutre $\left(J_{\lambda}^{G} \right)^{\circ}$ est d'indice fini dans $J_{\lambda}^{G}$. Il existe donc une certaine puissance $s_{\lambda}^{m} \in \left(J_{\lambda}^{G} \right)^{\circ}$. Puisque $s_{\lambda}$ est semi-simple hyperbolique, ceci implique $s_{\lambda}\in \left(J_{\lambda}^{G} \right)^{\circ}$.\\

L'intérêt du paragraphe précédent est le suivant. À partir d'un paramètre de Langlands originel $(\lambda,N)$, nous avons défini trois groupes $H_{\lambda}^{G}$, $J_{\lambda}^{G}$ et $K_{\lambda}^{G}$ vérifiant $H_{\lambda}^{G} \subseteq J_{\lambda}^{G} \subseteq K_{\lambda}^{G}$. Les éléments semi-simples $f_{\lambda}$ et $s_{\lambda}$ agissent de façon identique sur $N$. Par ailleurs, $f_{\lambda} \not \in K_{\lambda}^{G}$ alors que $s_{\lambda} \in J_{\lambda}$ (il peut même se descendre à l'algèbre de Lie) et $H_{\lambda}^{G}=Z_{J_{\lambda}^{G}}(s_{\lambda})=Z_{K_{\lambda}^{G}}(f_{\lambda})$.

\subsection{Support cuspidal d'un paramètre de Langlands enrichi}

La correspondance de Langlands locale relie les (classes de) représentations irréductibles de $G$ et les (classes de) paramètres de Langlands enrichis de $G$. Nous avons énoncé (et vérifié pour les groupes linéaires et classiques déployés) précédemment une conjecture concernant les paramètres de Langlands enrichis des représentations irréductibles supercuspidales. Par la correspondance de Langlands locale, il correspond à l'application de support cuspidal (et de support inertiel) une application de support cuspidal (et de support inertiel) pour les paramètres de Langlands enrichis. Autrement dit, il existe une application $$\cSc : \Phi_{e}(G) \longrightarrow \Omega_{e}^{\st}(G),$$ qui à tout paramètre enrichi $(\phi,\eta) \in \Phi_{e}(G)$, associe un triplet $(\widehat{L},\varphi,\epsilon)$ avec $\widehat{L}$ un sous-groupe de Levi de $\widehat{G}$, $\varphi \in \Phi(L)_{\cusp}$ un paramètre de Langlands cuspidal de $L$ et $\varepsilon \in \Irr(\mathcal{S}_{\varphi}^{L})$ une représentation irréductible cuspidale de $\mathcal{S}_{\varphi}^{L}$, au sens de la définition \ref{defcusp}. Par ailleurs, la correspondance de Langlands permet de définir une application $\rec_{\Omega(G)}^{e} : \Omega(G) \longrightarrow \Omega_{e}^{\st}(G)$, qui à une paire cuspidale $(L,\sigma)$ associe $(\widehat{L},\rec_{L}^{e}(\sigma))$. Ainsi, l'application de support cuspidal pour les paramètres de Langlands enrichis devrait être compatible avec la commutativité du diagramme suivant : 
\begin{center}
\begin{tikzcd}
\Irr(G)  \arrow{d}[swap]{\Sc} \arrow{r}{{\rec_{G}^{e}}}  &  \Phi_{e}(G) \arrow{d}{\cScp} \\
\Omega(G)          \arrow{r}{{\rec_{\Omega(G)}^{e}}} &  \Omega_{e}^{\st}(G)
\end{tikzcd}
\end{center} c'est-à-dire, pour tout $\pi \in \Irr(G)$, en notant $\Sc(\pi)=(L,\sigma)$ et $\rec_L^{e}(\sigma)=(\varphi,\varepsilon)$, alors $\cSc(\rec_{G}^{e}(\pi))=(\widehat{L},\varphi,\varepsilon)$.\\ 

Cette section est consacré à la définition du support cuspidal d'un paramètre de Langlands enrichi $(\phi,\eta)$ « intrinsèquement », c'est-à-dire uniquement en terme de paramètre enrichi (sans supposer la correspondance de Langlands). Pour cela nous utilisons les techniques de Lusztig pour définir le sous-groupe de Levi et le paramètre de Langlands cuspidal. Nous retrouvons (et nous nous inspirons) d'une construction similaire chez Lusztig (\cite{Lusztig:1988}, \cite{Lusztig:1995}, \cite{Lusztig:1995a}) et Waldspurger \cite{Waldspurger:2004} dans leur travaux sur la classification des représentations de réductions unipotentes d'un groupe simple adjoint (voir notamment \cite[\textsection 2]{Waldspurger:2004}).\\

\begin{theo}\phantomsection\label{supportcuspidalpartiel}
Soit $G$ un groupe réductif connexe défini et déployé sur $F$. Soit $(\phi,\eta)\in \Phi_{e}(G)$ un paramètre de Langlands enrichi de $G$. Il existe un sous-groupe de Levi $\widehat{L}$ de $\widehat{G}$ et un paramètre de Langlands cuspidal $\varphi$ de $L$ tels que : \begin{itemize}
\item $\lambda_{\phi}=\lambda_{\varphi}$ ;
\item pour tout $w \in W_F, \,\, \chi_{\phi}(w) \in \widehat{L}$ ;
\item en posant pour tout $w \in W_F,\,\, \chi_c(w)=\chi_{\phi}(w)/\chi_{\varphi}(w)$ alors $\chi_c(W_F) \subset Z_{\widehat{L}}^{\circ}$.
\end{itemize}
De plus, si $\widehat{L}'$ et $\varphi'$ sont respectivement, un autre sous-groupe de Levi de $\widehat{G}$ et un paramètre de Langlands cuspidal de $L'$ alors ils sont conjugués par un même élément de $Z_{\widehat{G}}(\restriction{\phi}{W_F})^{\circ}$ à $\widehat{L}$ et $\varphi$.
\end{theo}

\begin{rema}
Nous verrons dans la section suivante que dans le cas des groupes classiques on peut définir un support cuspidal pour tous paramètres de Langlands enrichis, c'est-à-dire un triplet $(\widehat{L},\varphi,\varepsilon) \in \Omega_e^{\st}(G)$. Par ailleurs, la construction est définie afin de satisfaire la conjecture \ref{conjindpar} de compatibilité entre l'induction parabolique et la correspondance de de Langlands.
\end{rema}

\begin{proof}
Notons $\lambda$ le caractère infitinitésimal de $\phi$ et rappelons que pour tout $w \in W_F$, $\chi_{\phi}(w)=\phi \left( 1,\left( \begin{smallmatrix} |w|^{-1/2} & 0 \\ 0 & |w|^{1/2} \end{smallmatrix} \right) \right)$, si bien que $\restriction{\phi}{W_F}=\lambda \chi_{\phi}$. De plus, on notera $\lambda(\Fr)=s_{\lambda} t_{\lambda}$ la décomposition de $\lambda(\Fr)$ avec $s_{\lambda}$ (resp. $t_{\lambda}$) un élément semi-simple hyperbolique (resp. elliptique), $H_{\phi}^{G}=Z_{\widehat{G}}(\restriction{\phi}{W_F})$, $J_{\lambda}=Z_{\widehat{G}}(\restriction{\lambda}{I_F},s_{\lambda})$, $N_{\phi}=d\restriction{\phi}{\SL_2(\C)}\left( \begin{smallmatrix} 0 & 1 \\ 0 & 0\end{smallmatrix} \right)$, $s_{\phi}=s_{\lambda} \chi_{\phi}(\Fr)$ et $r_{0}=-\frac{\log q}{2}$. On a alors :
\begin{center}
\begin{tikzcd}
A_{\widehat{G}}(\phi)=A_{H_{\phi}^{G}}(N_{\phi}) \ar[r,twoheadrightarrow] & \mathcal{S}_{\phi}^{G}   \\
A_{{(H_{\phi}^{G})}^{\circ}}(N_{\phi}) \ar[u,hook]  & 
\end{tikzcd}
\end{center}
Soit $\widetilde{\eta}$ la représentation de $A_{H_{\phi}^{G}}(N_{\phi})$ obtenue en tirant en arrière $\eta$ et soit $\eta_{0}$ une sous-représentation irréductible de la restriction de $\widetilde{\eta}$ au sous-groupe distingué  $A_{{(H_{\phi}^{G})}^{\circ}}(N_{\phi})$. \\

La correspondance de Springer généralisée appliquée au groupe $(H_{\phi}^{G})^{\circ}$ et au couple $(\mathcal{C}_{N_{\phi}}^{(H_{\phi}^{G})^{\circ}},\eta_{0})$ associe un triplet $((H_{\phi}^{L})^{\circ},\mathcal{C}_{N_{\varphi}}^{(H_{\phi}^{L})^{\circ}},\varepsilon_0)$ avec $(H_{\phi}^{L})^{\circ}$ un sous-groupe de Levi de $(H_{\phi}^{G})^{\circ}$, $N_{\varphi} \in \mathfrak{h}_{\phi}^{L}$ et $\varepsilon \in \Irr(A_{(H_{\phi}^{L})^{\circ}} (N_{\varphi}))$ une représentation cuspidale. Par ailleurs, on a vu que $H_{\phi}^{G}=Z_{J_{\lambda}}(s_{\phi})$ et d'après le théorème de Steinberg, $(H_{\phi}^{G})^{\circ}=Z_{(J_{\lambda\chi_{\phi}})^{\circ}}(s_{\phi})$. Ainsi, on peut appliquer la proposition \ref{actionss} et d'après le théorème de Jacobson-Morozov-Kostant, il existe un morphisme $$\theta : \SL_2(\C) \longrightarrow (H_{\phi}^{L})^{\circ},$$ tel que $d\theta\left( \begin{smallmatrix} 0 & 1 \\  0 & 0 \end{smallmatrix} \right)=N_{\varphi} $ et pour tout $t \in \C^{\times}$, $\theta \left( \begin{smallmatrix} t & 0 \\ 0 & t^{-1} \end{smallmatrix} \right)$ commute à $s_{\phi}$.\\

Définissons les cocaractères $\chi_{\varphi}$ et $\chi_{c}$, pour tout $w\in W_F$, par : $$\chi_{\varphi}=\theta(d_w)^{-1}, \,\, \chi_{c}(w)=\chi_{\phi}(w)/\chi_{\varphi}(w),$$ si bien que l'on a une décomposition $\chi_{\phi}=\chi_{\varphi} \chi_c=\chi_c \chi_{\varphi}$. \\

Puisque pour tout $w \in W_F$, $$\Ad(\lambda \chi_{\phi}(w))N_{\varphi}=N_{\varphi}, \quad \Ad(\chi_{\phi}(w))N_{\varphi}=|w|^{-1}N_{\varphi}, \quad \Ad(\chi_{\varphi}(w))N_{\varphi}=|w|^{-1}N_{\varphi},$$ on en déduit que pour tout $w \in W_F$, on a : $$\Ad(\chi_c(w))N_{\varphi}= N_{\varphi}, \,\, \Ad(\lambda(w))N_{\varphi}=|w| N_{\varphi}.$$ Par suite, on peut définir un paramètre de Langlands $\varphi$ pour tout $(w,x) \in W_F'$, par $$\varphi(w,x)=\lambda(w)\chi_{\varphi}(w) \theta(x).$$ D'après \cite[2.8]{Lusztig:1984}, puisque $\mathcal{C}_{N_{\varphi}}^{(H_{\phi}^{L})^{\circ}}$ supporte un système local cuspidal, $N_{\varphi}$ est un élément nilpotent distingué de $\mathfrak{h}_{\phi}^{L}$. De plus, pour tout $w \in W_F$ l'élément semi-simple $\chi_c(w) \in (H_{\phi}^{L})^{\circ}$ commute à $N_{\varphi}$. Donc pour tout $w \in W_F, \,\, \chi_c(w) \in A_{\widehat{L}}$. On obtient ainsi, \begin{align*}
Z_{H_{\phi}^{G}}(A_{\widehat{L}}) &= Z_{\widehat{G}}(A_{\widehat{L}},\lambda \chi_{\phi})& (H_{\phi}^{L})^{\circ} &=Z_{(H_{\phi}^{G})^{\circ}}(A_{\widehat{L}})\\
 &=Z_{\widehat{L}}(\lambda \chi_{\phi})  & &=Z_{\widehat{L}}(\lambda \chi_{\varphi})^{\circ} \\
 &=Z_{\widehat{L}}(\lambda \chi_{\varphi} \chi_c)   & &=(H_{\varphi}^{L})^{\circ}\\
 & =Z_{\widehat{L}}(\lambda \chi_{\varphi}) \\
 &=H_{\varphi}^{L} 
\end{align*} De plus, d'après \cite[2.3.b)]{Lusztig:1988}, $A_{\widehat{L}}$ est un tore maximal de $Z_{H_{\varphi}}(\theta)^{\circ}=Z_{\widehat{L}}(\lambda \chi_{\varphi}, \theta)^{\circ}=Z_{\widehat{L}}(\varphi)^{\circ}$. Ceci prouve que $$\varphi : W_F' \longrightarrow \widehat{L},$$ est un paramètre de Langlands discret de $L$. De plus, par construction, il est cuspidal et de caractère infinitésimal $\lambda$.\\

Pour finir, expliquons pourquoi $\mathcal{C}$ ne dépend pas du choix de $\eta_0$. En effet, si on prend une autre sous-représentation irréductible de $\widetilde{\eta}$, alors elle est de la forme $\eta_0^{x}$, où $x \in A_{H_{\phi}^{G}}(N_{\phi})$. Or, la correspondance de Springer généralisée pour ${(H_{\phi}^{G})}^{\circ}$ est $H_{\phi}^{G}$-équivariante. Ainsi, l'orbite nilpotente associée à $({}^{x} \mathcal{O},\eta_0^{x})$ est ${}^{x}\mathcal{C}$. Pour le cas des groupes classiques qui nous intéressera dans la suite, c'est-à-dire lorsque ${(H_{\phi}^{G})}^{\circ}$ est un produit de groupes symplectiques, spéciaux orthogonaux et linéaires, il y a au plus une orbite nilpotente supportant un système local cuspidal. Nécessairement, ${}^{x}\mathcal{C}=\mathcal{C}$. En général, on peut supposer ${(H_{\phi}^{G})}^{\circ}$ simplement connexe, puis le décomposer en produit presque direct de groupes simples (et d'un tore central), on vérifie à l'aide de la classification des paires cuspidales décrite dans \cite{Lusztig:1984}),que ${}^{x}\mathcal{C}=\mathcal{C}$.\\
\end{proof}

Profitons d'être dans ce contexte pour introduire les définitions suivantes.

\begin{defi}\phantomsection\label{paramadapte}
On reprend les notations précédentes. Soient $\varphi : W_F' \longrightarrow \widehat{L}$ un paramètre de Langlands cuspidal de $L$ et $\phi : W_F' \longrightarrow \widehat{G}$ un paramètre de Langlands de $G$. On note $H_{\phi}^{G}=Z_{\widehat{G}}(\restriction{\phi}{W_F}), \,\, H_{\varphi}^{L}=Z_{\widehat{L}}(\restriction{\varphi}{W_F})$ et $\mathfrak{h}_{\phi}^{G}, \,\, \mathfrak{h}_{\varphi}^{L}$ leurs algèbres de Lie respectives.
On dit que $\phi$ est adapté à $\varphi$ si : \begin{itemize}
\item $\phi$ et $\varphi$ ont le même caractère infinitésimal $\lambda$ ;
\item pour tout $t \in \C^{\times}, \,\, \phi \left(1, \left( \begin{smallmatrix} t & 0 \\ 0 & t^{-1} \end{smallmatrix} \right)  \right) \in H_{\varphi}^{L}$ ;
\item il existe une sous-algèbre parabolique $\mathfrak{p}=\mathfrak{h}_{\varphi}^{L} \oplus \mathfrak{u}$ de $\mathfrak{h}_{\phi}^{G}$, admettant $\mathfrak{h}_{\varphi}^{L}$ pour sous-algèbre de Levi telle que : $N_{\phi} \in \mathfrak{p}$ et $N_{\phi}=N_{\varphi}+U$ (avec $U \in \mathfrak{u}$) ;
\end{itemize}
\end{defi}

\begin{defi}\label{cocaractercorrection}
Soit $\varphi : W_F' \longrightarrow \widehat{L}$. On appelle cocaractère de correction du paramètre de $\varphi$ dans $\widehat{G}$, tout cocaractère  $c : \C^{\times} \longrightarrow Z_{\widehat{L}}^{\circ}$ défini pour tout $t \in \C^{\times}$, par : $$c(t)=\frac{\phi \left( 1,\left( \begin{smallmatrix} t & 0 \\ 0 & t^{-1} \end{smallmatrix} \right) \right)}{\varphi\left( 1,\left( \begin{smallmatrix} t & 0 \\ 0 & t^{-1} \end{smallmatrix} \right) \right)},$$ où $\phi$ est un paramètre de Langlands de $G$ adapté à $\varphi$. On note alors $\chi_{c} : W_F \longrightarrow Z_{\widehat{L}}^{\circ}$ le cocaractère non ramifié de $W_F$ défini pour tout $w \in W_F$ par $\chi_{c}=c(|w|^{-1/2})$.
\end{defi}

\begin{rema}
Le fait que $c$ est à valeur dans $Z_{\widehat{L}}^{\circ}$ résulte du fait que les éléments de  $c(W_F) \subset {(H_{\varphi}^{L})}^{\circ}$ commutent à $N_{\varphi}$, sont semi-simples et que $N_{\varphi}$ est un élément nilpotent distingué dans $\mathfrak{h}_{\varphi}^{L}$.
\end{rema}

\begin{rema}
La définition de $c$ montre qu'il y a un nombre fini de cocaractère de correction du paramètre de $\varphi$ dans $\widehat{G}$. Plus précisément, ces cocaractères sont uniquement déterminés (à conjugaison près) par l'orbite nilpotente de $N_{\phi}$. Si on se place du point de vue du groupe de Weil-Deligne originel, notons $\lambda$ le caractère infinitésimal de $\varphi$. Les cocaractères de correction de $\varphi$ dans $\widehat{G}$ sont en bijection avec les $Z_{\widehat{G}}(\lambda)$-orbites nilpotentes $\mathcal{O} \subset \mathfrak{h}_{\Lambda}(q^{-1})$, tel qu'il existe un élément $N \in \mathcal{O}$ et une décomposition $N=N_{\varphi}+U$ où $U$ est un élément dans le radical unipotent d'une sous-algèbre parabolique admettant $\mathfrak{l}$ pour facteur de Levi.
\end{rema}

Pour définir de façon satisfaisante le support cuspidal d'un paramètre de Langlands enrichi, il nous faut associer non pas une représentation irréductible cuspidale de $A_{Z_{\widehat{L}}(\varphi)^{\circ}}(\restriction{\varphi}{\SL_2(\C)})$ mais une représentation irréductible cuspidale de $A_{Z_{\widehat{L}}(\varphi)}(\restriction{\varphi}{\SL_2(\C)})=A_{\widehat{L}}(\varphi)$. C'est l'objet de la section suivante dans le cas où $G$ est un groupe classique.

\subsection{Support cuspidal des paramètres de Langlands enrichis dans le cas des groupes classiques}\label{sec:suppcusp}

Explicitons cette construction dans le cas qui nous intéresse et introduisons quelques notations. Le groupe $G$ désigne l'un des groupes suivants $\Sp_{N}(F)$ ou $\SO_{N}(F)$, $\widehat{G}$ son dual de Langlands. Avant de commencer, introduisons et rappelons le calcul de centralisateurs des paramètres de Langlands. On note $\widehat{G}_{*}$ le groupe orthogonal associé si $\widehat{G}$ est un groupe spécial orthogonal et $\widehat{G}_{*}=\widehat{G}$ si $\widehat{G}$ est un groupe symplectique.\\

Soit $\phi \in \Phi(G)$ un paramètre de Langlands de $G$. On note $I$ l'ensemble des (classes de) représentations irréductibles de $W_F$ qui apparaissent dans $\Std_G \circ \restriction{\phi}{W_F}$ et qu'on décompose $I=I^{\O} \sqcup I^{\S} \sqcup I^{\GL}$, où $I^{\O}$ (resp. $I^{\S}$) désigne le sous-ensemble de $I$ formé des représentations de type orthogonales (resp. symplectiques) et $I^{\GL}$ un sous-ensemble maximal de $I$ formés de représentations qui ne sont pas autoduales et tel que si $\pi \in I^{\GL}$ alors $\pi^{\vee} \not\in I^{\GL}$. Si bien que : $$\Std_G \circ \restriction{\phi}{W_F} = \bigoplus_{\pi \in I^{\O}} \pi \otimes M_{\pi} \bigoplus_{\pi \in I^{\S}} \pi \otimes M_{\pi} \bigoplus_{\pi \in I^{\GL}} (\pi \oplus \pi^{\vee}) \otimes M_{\pi} ,$$ où pour tout $\pi \in I$, $M_{\pi}$ est l'espace de multiplicité de $\pi$.\\

En fonction du type de $\widehat{G}$ et de $\pi \in I$, $M_{\pi}$ est un espace orthogonal, symplectique ou totalement isotrope. Pour tout $\pi \in I$, notons $\widehat{G}_{\phi,\pi}$ le groupe des isométries de $M_{\pi}$. Pour être plus précis, en notant $m_{\pi}=\dim M_{\pi}$, on a : $$\widehat{G}_{\phi,\pi}= \left\{ \begin{tabular}{ll}
$\O_{m_{\pi}}(\C)$ & si ($\widehat{G}$ est un groupe spécial orthogonal et $\pi \in I^{\O}$) ou ($\widehat{G}$ est un groupe symplectique et $\pi \in I^{\S}$)\\
$\Sp_{m_{\pi}}(\C)$ & si ($\widehat{G}$ est un groupe spécial orthogonal et $\pi \in I^{\S}$) ou ($\widehat{G}$ est un groupe symplectique et $\pi \in I^{\O}$) \\
$\GL_{m_{\pi}}(\C)$ & si $\pi \in I^{\GL}$
\end{tabular} \right.$$

Par conséquent, $A_{\widehat{G}_{*}}(\restriction{\phi}{W_F}) \simeq \prod_{\pi \in I} \widehat{G}_{\phi,\pi}$. Notons $S_{\phi}=\restriction{\phi}{\SL_2(\C)}$ et remarquons qu'à travers ce dernier isomorphisme, il correspond à $S_{\phi}$ une famille de morphismes indexée par $\pi \in I $ qu'on note $S_{\phi,\pi} : \SL_2(\C) \rightarrow \widehat{G}_{\phi,\pi}$, si bien que $$A_{\widehat{G}_{*}}(\phi)=A_{Z_{\widehat{G}_{*}}(\restriction{\phi}{W_F})}(\restriction{\phi}{\SL_2(\C)}) \simeq \prod_{\pi \in I} A_{\widehat{G}_{\phi,\pi}}(S_{\phi,\pi}).$$

À présent, supposons que $\widehat{G}$ est un groupe spécial orthogonal. Notons $I^{\O,\imp}=\{ \pi \in I^{\O} \mid \dim \pi \equiv 1 \mod 2 \}$ et $I^{O,\pair}=\{ \pi \in I^{\O} \mid \dim \pi \equiv 0 \mod 2 \}$. On obtient ainsi les isomorphismes :\begin{align*}
Z_{\widehat{G}_{*}}(\restriction{\phi}{W_F}) & \simeq \prod_{\pi \in I^{\O}} \O_{m_{\pi}}(\C) \times \prod_{\pi \in I^{\S}} \Sp_{m_{\pi}}(\C) \times \prod_{\pi \in I^{\GL}}\GL_{m_{\pi}}(\C)  \\
Z_{\widehat{G}}(\restriction{\phi}{W_F}) & \simeq S\left( \prod_{\pi \in I^{\O}} \O_{m_{\pi}}(\C) \right) \times \prod_{\pi \in I^{\S}} \Sp_{m_{\pi}}(\C) \times \prod_{\pi \in I^{\GL}}\GL_{m_{\pi}}(\C) \\
																	 & \simeq S\left( \prod_{\pi \in  I^{\O,\imp}} \O_{m_{\pi}}(\C) \right) \times \prod_{\pi \in I^{\O,\pair}} \O_{m_{\pi}}(\C)  \times \prod_{\pi \in I^{\S}} \Sp_{m_{\pi}}(\C) \times \prod_{\pi \in I^{\GL}}\GL_{m_{\pi}}(\C)\\
																	 & \simeq \widehat{G}_{\phi,i} \times \prod_{\pi \in I^{\O,\pair} \sqcup I^{\S} \sqcup I^{\GL}} \widehat{G}_{\phi,\pi},
\end{align*} 

où l'on a noté $$\widehat{G}_{\phi,i}=S\left( \prod_{\pi \in I^{\O,\imp}} \O_{m_{\pi}}(\C) \right)=\left\{(z_{\pi}) \in \prod_{\pi \in I^{\O,\imp}} \O_{m_{\pi}}(\C) \Biggm| \prod_{\pi \in I^{\O,\imp}} \det(z_{\pi})=1 \right\}.$$ Comme précédemment, on note $S_{\phi,i} : \SL_2(\C) \rightarrow \widehat{G}_{\phi,i}$ le morphisme correspondant. En notant $I'=I$ si $\widehat{G}$ est un groupe symplectique et $I'=I^{\O,\pair} \sqcup I^{\S} \sqcup I^{\GL} \sqcup \{i \}$ si $\widehat{G}$ est un groupe spécial orthogonal, on obtient : $$A_{\widehat{G}}(\phi) \simeq \prod_{\pi \in I'} A_{\widehat{G}_{\phi,\pi}}(S_{\phi,\pi}).$$

\begin{theo}\phantomsection\label{theoremesupportcuspidal}
Soit $G$ un groupe linéaire ou un groupe classique déployé. Il existe une application de support cuspidal : $$\cSc : \begin{array}[t]{ccc}
\Phi_{e}(G) & \longrightarrow & \Omega_{e}^{\st}(G) \\
(\phi,\eta)  & \longmapsto & (\widehat{L}, \varphi,\varepsilon)
\end{array}.$$ Cette application est surjective et respecte l'induction parabolique, c'est-à-dire : si $\cSc \,(\phi,\eta)=(\widehat{L},\varphi,\varepsilon)$, alors $(\lambda_{\phi})_{\widehat{G}}=(\lambda_{\varphi})_{\widehat{G}}$.
\end{theo}
 
\begin{proof}[Démonstration du théorème \ref{theoremesupportcuspidal}]
Soient $(\phi,\eta) \in \Phi_e(G)$ un paramètre de Langlands enrichi de $G$ et $\widetilde{\eta} \in \Irr(A_{\widehat{G}}(\phi))$ la représentation irréductible de $A_{\widehat{G}}(\phi)$ obtenue en composant la projection $A_{\widehat{G}}(\phi) \twoheadrightarrow \mathcal{S}_{\phi}^{G}$ avec $\eta$.\\

D'après ce qui précède, l'isomorphisme $H_{\phi}^{G} \simeq \prod_{\pi \in I'} H_{\phi,\pi}^{G}$ induit les décompositions suivantes : $\widetilde{\eta} \simeq \boxtimes_{\pi \in I'} \widetilde{\eta}_{\pi}$ où $\widetilde{\eta}_{\pi} \in \Irr(A_{H_{\phi,\pi}^{G}}(S_{\phi,\pi}))$, $S_{\phi}=(S_{\phi,\pi})_{\pi \in I'}$ et $\chi_{\phi}=(\chi_{\phi,\pi})_{\pi \in I'}$.\\

Soit $\pi \in I'$. En utilisant la correspondance de Springer généralisée pour chacun de ces groupes (d'après Lusztig pour les groupes connexes et l'appendice pour les groupes de type orthogonal) et la construction décrite dans la démonstration du théorème \ref{supportcuspidalpartiel}, on associe à $S_{\phi,\pi}$ et $\widetilde{\eta}_{\pi}$ pour le groupe $H_{\phi,\pi}^{G}$ : \begin{itemize}
\item une sous-groupe de quasi-Levi $H_{\varphi,\pi}^{L}$ de $H_{\phi,\pi}^{G}$ ;
\item un morphisme $S_{\varphi,\pi}: \SL_2(\C) \rightarrow H_{\varphi,\pi}^{L}$ ;
\item une décomposition du cocaractère $\chi_{\phi,\pi} = \chi_{\varphi,\pi} \chi_{c_{\pi}}$ ;
\item une représentation irréductible cuspidale $\widetilde{\varepsilon}_{\pi} \in \Irr(A_{H_{\varphi,\pi}^{L}}(S_{\varphi,\pi}))$ ;
\item une représentation irréductible $\rho_{\pi} \in \Irr(W_{H_{\phi,\pi}^{G}}^{H_{\phi,\pi}^{G}})$.
\end{itemize}

Après avoir interprété la représentation irréductible $\widetilde{\eta}$ de $A_{\widehat{G}}(\phi)$ comme une représentation de $A_{Z_{\widehat{G}}(\restriction{\phi}{W_F})}(\restriction{\phi}{\SL_{2}(\C)})$, nous avons associé à $(\phi,\widetilde{\eta})$ un paramètre de Langlands cuspidal $\varphi$ pour un certain sous-groupe de Levi $L$ de $G$ et les données ci-dessus. À présent, nous allons parcourir le chemin inverse, c'est-à-dire interpréter les données associées ci-dessus en terme d'une représentation irréductible de $A_{\widehat{L}}(\varphi)$ (et d'une représentation d'un groupe de Weyl étendu).\\

Soit $A_{\pi}$ le plus grand tore central de $H_{\varphi,\pi}^{L}$. Notons $A \simeq \prod_{\pi \in I'} A_{\pi}$ le tore de $\widehat{G}$ correspondant et $\widehat{L}=Z_{\widehat{G}}(A)$ le sous-groupe de Levi de $\widehat{G}$ correspondant. On obtient ainsi une décomposition $H_{\varphi}^{L} \simeq \prod_{\pi \in I'} H_{\varphi,\pi}^{L}$ et on peut  définir $$S_{\varphi}=(S_{\varphi,\pi})_{\pi \in I'}, \;\; \chi_{\varphi}=(\chi_{\varphi,\pi})_{\pi \in I'} \;\; \text{et} \;\; \varphi : \begin{array}[t]{lcl}
W_F \times \SL_2(\C) & \rightarrow & \widehat{L}\\
(w,x)  & \mapsto & \lambda(w) \chi_{\varphi}(w) S_{\varphi}(x)
\end{array}.$$ On obtient alors $A_{\widehat{L}}(\varphi)=A_{H_{\varphi}^{L}}(S_{\varphi}) \simeq \prod_{\pi \in I'} A_{H_{\varphi,\pi}^{L}}(S_{\varphi,\pi})$. Ceci permet donc de définir une représentation irréductible cuspidale $\widetilde{\varepsilon}$ de $A_{\widehat{L}}(\varphi)$ correspondante à $\boxtimes_{\pi \in I'} \widetilde{\varepsilon}_{\pi}$. Il s'agit de voir à présent que $\varepsilon$ se factorise en une représentation de $\mathcal{S}_{\varphi}^{L}$. Lorsque $\widehat{G}=\SO_{2n+1}(\C)$, il n'y a rien à faire puisque son centre est trivial.\\ Soit $\pi \in I'$. Si $H_{\phi,\pi}^{G}$ est un groupe symplectique, d'après \cite[5.23]{Lusztig:1995}, $-1$ agit sur $\widetilde{\eta}_{\pi}$ par le même scalaire que sur $\widetilde{\varepsilon}_{\pi}$. Ainsi, $\eta_{\pi}(-1)=\varepsilon_{\pi}(-1)$. Si $H_{\phi,\pi}^{G}$ est un groupe de type orthogonal, la construction de la correspondance de Springer pour le groupe orthogonal de la section \ref{springerorth} permet d'affirmer que l'on a aussi $\eta_{\pi}(-1)=\varepsilon_{\pi}(-1)$ et $\widetilde{\eta}_{i}(-1)=\widetilde{\varepsilon}_{i}(-1)$. Ainsi, on a : $$\widetilde{\eta}(-1)=\prod_{\pi \in I'} \eta_{\pi}(-1) =\prod_{\pi \in I'} \varepsilon_{\pi}(-1)=\widetilde{\varepsilon}(-1).$$ Par suite, $\widetilde{\varepsilon}$ se factorise en une représentation irréductible $\varepsilon$ de $\mathcal{S}_{\varphi}^{L}$.\\

Revenons à présent sur la représentation du groupe de Weyl. On a l'isomorphisme : $$W_{H_{\phi}^{L}}^{H_{\phi}^{G}} \simeq \prod_{\pi \in I'} W_{H_{\phi,\pi}^{L}}^{H_{\phi,\pi}^{G}}.$$ Par suite, il correspond une unique représentation irréductible $\rho$ de $W_{H_{\phi}^{L}}^{H_{\phi}^{G}} $ correspondant à $\boxtimes_{\pi \in I'} \rho_{\pi}$.
\end{proof}

\begin{rema}
Comme on peut le constater dans la preuve, le théorème reste vrai pour les formes intérieures pures des groupes classiques.
\end{rema}

\section{Paramétrage de Langlands du dual admissible des groupes classiques}

Dans la section précédente, nous avons vu que les pour les groupes classiques, les paramètres de Langlands des représentations supercuspidales correspondent aux paramètres de Langlands cuspidaux (définition \ref{defcusp}). De plus, nous avons construit une application de support cuspidal pour les paramètres de Langlands enrichis. Dans cette section, on construit une paramétrisation de Langlands des représentations irréductibles. Cette construction est basée sur la comparaison de deux algèbres de Hecke qui paramètrent d'une part les représentations irréductibles dans un bloc de Bernstein et d'autre part les paramètres de Langlands enrichis. Pour la première paramétrisation, nous utilisons les résultats d'Heiermann qui établit une équivalence de catégorie entre les représentations lisses dans un bloc de Bernstein et les modules sur une algèbre de Hecke affine étendue. Pour la seconde, nous utilisons les résultats de Lusztig. Le lien entre les deux (plus précisément entre les algèbres de Hecke graduées associées) se fait par comparaison des fonctions paramètres des algèbres de Hecke. Ceci sera une conséquence des travaux de M\oe glin qui explicite les points de réductiblités de certaines induites en terme de paramètres de Langlands.

\subsection{Relation entre le centre de Bernstein et le centre de Bernstein dual}

Nous avons construit en terme de paramètres de Langlands enrichis une variété $\Omega_{e}^{\st}(G)$ qu'on conjecture être isomorphe à $\Omega(G)$. Nous vérifions dans cette section que c'est le cas pour les groupes classiques déployés en prouvant le théorème suivant.

\begin{theo}\phantomsection\label{verifcompatibilite}
Soit $G$ l'un des groupes $\SO_{N}(F)$ ou $\Sp_{N}(F)$. Soient $\mathfrak{s}=[L,\sigma] \in \mathcal{B}(G)$ et $\cj=[\widehat{L},\varphi,\varepsilon] \in \mathcal{B}_{e}^{\st}(G)$ le triplet inertiel de $\widehat{G}$ correspondant par la correspondance de Langlands. Notons $T_{\mathfrak{s}}, \mathcal{T}_{\cjp}, W_{\mathfrak{s}}, \mathcal{W}_{\cjp}$ les tores de Bernstein et groupes de Weyl définis dans les sections précédentes. La correspondance de Langlands induit un isomorphisme  $T_{\mathfrak{s}} \simeq \mathcal{T}_{\cjp}$. De plus, les groupes $W_{\mathfrak{s}}$ et $\mathcal{W}_{\cjp}$ sont isomorphes et les actions de chacun sur $\mathcal{T}_{\cjp}$ et $T_{\mathfrak{s}}$ est compatible avec cet isomorphisme. Plus précisément, la conjecture \ref{compatibiliteaction} est vraie pour $G$.
\end{theo}

\begin{proof}
Nous allons décrire explicitement $W_{\mathfrak{s}}$ et $\mathcal{W}_{\cjp}$. Commençons par décrire $W_{\mathfrak{s}}$, en suivant les résultats d'Heiermann \cite[1.12,1.13,1.15]{Heiermann:2011} et \cite[2.2]{Heiermann:2010}.\\

En reprenant les notations de l'énoncé, on peut supposer que $$L=\prod_{i=1}^r \GL_{n_i}(F)^{\ell_i} \times G' \quad \text{et}  \quad
 \sigma = \underbrace{\pi_1 \boxtimes \ldots \boxtimes \pi_1}_{\ell_1} \boxtimes \ldots \boxtimes \underbrace{\pi_r \boxtimes \ldots \boxtimes \pi_r}_{\ell_r} \boxtimes \sigma',$$ avec $\pi_i$ une représentation irréductible supercuspidale unitaire de $\GL_{n_i}(F)$ et $\sigma'$ une représentation irréductible supercuspidale d'un groupe $G'$ de même type que $G$ mais de rang semi-simple inférieur. Notons $T_{\pi_{i}}$ l'orbite inertielle de $\pi_{i}$, c'est-à-dire $T_{\pi_{i}}=\left\{ \pi_{i} \chi, \chi \in \mathcal{X}(\GL_{n_i}(F))\right\}$. On peut alors de plus supposer que (voir \cite[1.13]{Heiermann:2011} et \cite[4.1]{Heiermann:2012}) :
\begin{itemize}
\item pour tout $i \in \llbracket 1,r \rrbracket$, si $T_{\pi_i}=T_{\pi_i^{\vee}}$, alors $\pi_i \simeq \pi_i^{\vee}$ ;
\item pour tout $i,j \in \llbracket 1,r \rrbracket$, si $i \neq j $, alors $T_{\pi_i} \neq T_{\pi_j}$ et $T_{\pi_i} \neq T_{\pi_j^{\vee}}$ ;
\item pour tout $i \in \llbracket 1,r \rrbracket$, ou bien il existe $s \in \R_{+}^{*}$ tel que $\pi_i  |\,\,|^{s} \rtimes \sigma'$ est réductible, ou bien pour tout $s \in \R_{+}^{*}$, $\pi_i  |\,\,|^{s} \rtimes \sigma'$ est irréductible.
\end{itemize}

On peut ainsi distinguer trois cas de figure : \begin{enumerate}[label=(\roman*)]
\item $T_{\pi_i}=T_{\pi_i^{\vee}}$, $\pi_i \simeq \pi_i^{\vee}$ et il existe $s \in \R_{+}^{*}$ tel que $\pi_i  |\,\,|^{s} \rtimes \sigma'$ est réductible ;
\item $T_{\pi_i}=T_{\pi_i^{\vee}}$, $\pi_i \simeq \pi_i^{\vee}$ et pour tout $s \in \R_{+}^{*}$, $\pi_i  |\,\,|^{s} \rtimes \sigma'$ est irréductible ;
\item $T_{\pi_i} \neq T_{\pi_i^{\vee}}$.
\end{enumerate}

Un élément $m \in \prod_{i=1}^r \GL_{n_i}(F)^{\ell_i}$ peut être décomposé $(m_{i,j})_{i \in \llbracket 1, r \rrbracket, j \in\llbracket 1,\ell_i \rrbracket}$. Pour tout $i \in \llbracket 1, r \rrbracket, j \in\llbracket 1,\ell_i-1 \rrbracket$, notons $s_{i,j} \in N_{G}(L)/L$ dont l'action sur $m$ permute les éléments $m_{i,j}$ et $m_{i,j+1}$. Notons $s_{i,\ell_i} \in N_{G}(L)/L$ lorsque $G$ est un groupe symplectique ou spécial orthogonal impair et $s_{i,\ell_i} \in N_{\O_N(F)}(L)/L$ lorsque $G$ est un groupe spécial orthogonal pair, l'élément dont l'action sur $m$ échange $m_{i,\ell_{i}}$ et ${}^{\tau} m_{i,\ell_{i}}^{-1}$.\\

On a une décomposition  $W_{\mathfrak{s}}=W_{\mathfrak{s}}^{\circ} \rtimes R_{\mathfrak{s}}$. Le groupe $W_{\mathfrak{s}}^{\circ}$ est un produit direct de groupe de Weyl $W_{\mathfrak{s},i}^{\circ}$, avec pour tout $i \in \llbracket 1,r \rrbracket$ :

\begin{equation} \label{Ws}
W_{\mathfrak{s},i}^{\circ} = \left\{ \begin{array}{lll}
\langle s_{i,j}, j \in \llbracket 1,\ell_{i} \rrbracket \rangle & \mbox{si $\pi_i$ vérifie (i),} & \mbox{groupe de Weyl de type $B_{\ell_i}/C_{\ell_i}$} \\
\langle s_{i,j}, j \in \llbracket 1,\ell_{i}-1 \rrbracket,\,\, s_{i,\ell_{i}}s_{i,\ell_{i}-1}s_{i,\ell_{i}}^{-1} \rangle & \mbox{si $\pi_i$ vérifie (ii),} & \mbox{groupe de Weyl de type $D_{\ell_i}$} \\
\langle s_{i,j}, j \in \llbracket 1,\ell_{i}-1 \rrbracket \rangle & \mbox{si $\pi_i$ vérifie (iii),} &\mbox{groupe de Weyl de type $A_{\ell_i-1}$} \\
\end{array} \right.
\end{equation}

Décrivons le groupe $R_{\mathfrak{s}}$ (voir \cite[1.5]{Goldberg:2011}) et pour cela, notons $$
C =\{ i \in \llbracket 1,r \rrbracket \mid \pi_{i} \mbox{ vérifie (ii)} \},  \;\; C_{\pair} = \{ i \in C \mid n_i \equiv 0 \mod 2 \} \;\; \text{et} \;\; C_{\imp} = \{ i \in C \mid n_i \equiv 1 \mod 2 \}.$$
Il vient alors :
\begin{equation} \label{Rs}
R_{\mathfrak{s}} = \left\{ \begin{array}{ll}
\prod_{i \in C} \langle s_{i,\ell_{i}} \rangle & \text{si $G=\Sp_{N}(F)$ ou $G=\SO_{N}(F)$ avec $N$ impair} \\
\prod_{i \in C} \langle s_{i,\ell_{i}} \rangle & \text{si $G=\SO_{N}(F)$ et $L=\GL_{n_1}(F)^{\ell_1} \times \ldots \times \GL_{n_r}(F)^{\ell_r} \times \SO_{N'}(F)$ avec $N$ pair et $N' \geqslant 4$} \\
\prod_{i \in C_{\pair}} \langle s_{i,\ell_{i}} \rangle  \times \langle s_{i,\ell_{i}}s_{j,\ell_{j}} \mid i,j \in C_{\imp} \rangle & \text{si $G=\SO_{N}(F)$ et $L=\GL_{n_1}(F)^{\ell_1} \times \ldots \times \GL_{n_r}(F)^{\ell_r}$ avec $N$ pair} \\
\end{array} \right.
\end{equation}

On s'intéresse à présent au groupe $\mathcal{W}_{\cjp}$. Notons $A_{\widehat{L}}=Z_{\widehat{L}}^{\circ}$ et posons $$D_{\varphi}^{G}=\{ g \in \widehat{G} \mid \exists \chi \in \mathcal{X}(\widehat{L}), {}^g \restriction{\varphi}{W_F}=\restriction{\varphi}{W_F}\chi \}  =\{g \in Z_{\widehat{G}}(\restriction{\varphi}{I_F})  \mid \, g\varphi(\Fr)g^{-1}\varphi(\Fr) ^{-1} \in A_{\widehat{L}}  \}.$$ Alors, ${(D_{\varphi}^{G})}^{\circ}={(H_{\varphi}^{G})}^{\circ}$ car $H_{\varphi}^{G} \subset D_{\varphi}^{G}$ et ils ont la même algèbre de Lie. Le groupe $Z_{D_{\varphi}^{G}}(\restriction{\varphi}{\SL_2(\C)})$ est l'ensemble des éléments $g \in \widehat{G}$ tels qu'il existe $\chi \in \mathcal{X}(\widehat{L})$ et ${}^g \varphi= \varphi \chi$. Considérons $$N_{Z_{D_{\varphi}^{G}}(\restriction{\varphi}{\SL_2(\C)})}(A_{\widehat{L}},\varepsilon)=\{g \in N_{\widehat{G}}(A_{\widehat{L}})  \mid \exists \chi \in \mathcal{X}(\widehat{L}), \, {}^{g}\varphi=\varphi \chi, \,\, {}^{g}\varepsilon \simeq \varepsilon \}.$$ Les propriétés de $\varepsilon$ (unicité de sa restriction à $A_{(H_{\varphi}^{G})^{\circ}}(\restriction{\varphi}{\SL_2(\C)})$ et le fait que ce soit un caractère) entraînent que la dernière condition est automatiquement vérifiée. On a un morphisme évident $N_{Z_{D_{\varphi}^{G}}(\restriction{\varphi}{\SL_2(\C)})}(A_{\widehat{L}}) \rightarrow \mathcal{W}_{\cjp}$ induit un isomorphisme $\mathcal{W}_{\cjp} \simeq N_{Z_{D_{\varphi}^{G}}(\restriction{\varphi}{\SL_2(\C)})}(A_{\widehat{L}},\varepsilon)/Z_{D_{\varphi}^{L}}(\restriction{\varphi}{\SL_2(\C)})$.

Notons \begin{align*}
\mathcal{W}_{\cjp,\varphi}^{\circ} &=N_{Z_{D_{\varphi}^{G}}(\restriction{\varphi}{\SL_2(\C)})^{\circ}}(A_{\widehat{L}})/Z_{D_{\varphi}^{L}}(\restriction{\varphi}{\SL_2(\C)})^{\circ} \\
&\simeq N_{Z_{(H_{\varphi}^{G})^{\circ}}(\restriction{\varphi}{\SL_2(\C)})^{\circ}}(A_{\widehat{L}})/Z_{(H_{\varphi}^{L})^{\circ}}(\restriction{\varphi}{\SL_2(\C)})^{\circ} \\
&\simeq N_{Z_{\widehat{G}}(\varphi)^{\circ}}(A_{\widehat{L}})/A_{\widehat{L}}.
\end{align*}

Remarquons que $\mathcal{W}_{\cjp,\varphi}^{\circ}$ est un sous-groupe distingué de $\mathcal{W}_{\cjp}$. De plus, $\mathcal{W}_{\cjp,\varphi}^{\circ}$ est le groupe de Weyl du groupe réductif $Z_{\widehat{G}}(\varphi)^{\circ}$, admettant $A_{\widehat{L}}$ pour tore maximal. Notons $\Sigma_{\cjp,\varphi}$ le système de racines associé à $(Z_{\widehat{G}}(\varphi)^{\circ},A_{\widehat{L}})$ et $\Sigma_{\cjp,\varphi}^{+}$ un sous-ensemble de racines positives. Soit $$\mathcal{R}_{\cjp,\varphi}=\left\{w \in \mathcal{W}_{\cjp} \mid w\Sigma_{\cjp,\varphi}^{+}=\Sigma_{\cjp,\varphi}^{+} \right\}.$$ Puisque $\mathcal{W}_{\cjp,\varphi}^{\circ}$ agit simplement transitivement sur les systèmes de racines positives, on a la décomposition suivante $$\mathcal{W}_{\cjp} = \mathcal{W}_{\cjp,\varphi}^{\circ} \rtimes \mathcal{R}_{\cjp,\varphi}.$$

On peut choisir $\varphi$ tel que pour tout $\chi \in \mathcal{X}(\widehat{L})$, $\Sigma_{\cjp,\varphi\chi} \subseteq \Sigma_{\cjp,\varphi}$, ou de façon à ce que $\mathcal{W}_{\cjp,\varphi}^{\circ}$ soit maximal. Comme précédemment on peut supposer avoir regroupé les représentations dans la même orbite inertielle. Notons $\widehat{L}=\prod_{i=1}^r \GL_{n_i}(\C)^{\ell_i} \times \widehat{G}'$ le dual de Langlands de $L$. Pour tout $i \in \llbracket 1,r \rrbracket$, soient \begin{itemize}
\item $\pi_i : W_F \longrightarrow \GL_{n_i}(\C)$ un paramètre de Langlands cuspidal de $\pi_{i}$ ; 
\item $\varphi' : W_F' \longrightarrow \widehat{G}'$ un paramètre de Langlands cuspidal de $\sigma'$ ;
\item $\varphi: W_F' \longrightarrow \widehat{L}$ un paramètre de Langlands de $\sigma$.
\end{itemize} Décomposons $\Std_G \circ \varphi$  : \begin{align*}
\Std_G \circ \varphi &=\underbrace{(\pi_1 \oplus \pi_1^{\vee}) \oplus \ldots \oplus  (\pi_1 \oplus \pi_1^{\vee})}_{\ell_1} \oplus  \ldots \oplus \underbrace{(\pi_r \oplus \pi_r^{\vee}) \oplus \ldots \oplus (\pi_r \oplus \pi_r^{\vee})}_{\ell_r} \oplus \varphi' \\
&= \bigoplus_{i=1}^{r} \ell_{i} (\pi_i \oplus \pi_i^{\vee}) \bigoplus \varphi'
\end{align*}

Pour tout $i \in \llbracket 1,r \rrbracket$, notons $\mathcal{T}_{\pi_i}=\left\{ \pi_{i} \chi, \chi \in \mathcal{X}(\GL_{n_i}(\C)) \right\}$ l'orbite inertielle de $\pi_i$. \\

Soit $I$ l'ensemble des représentations irréductibles de $W_F$ apparaissant dans $\restriction{\varphi}{W_F}$. On a vu qu'on peut décomposer $I=I^{\O} \sqcup I^{\S} \sqcup I^{\GL}$. Décomposons de la même manière $\Std_{G'} \circ \varphi'$ : \begin{align*}
\Std_{G'}\circ \varphi' &= \bigoplus_{(\pi,q) \in \Jord(\varphi')} \pi \boxtimes S_{q} \\
&= \bigoplus_{\pi \in I_{\varphi'}} \pi \boxtimes S_{\pi},
\end{align*} où $I_{\varphi'}$ est l'ensemble des représentations irréductibles de $W_F$ apparaissant dans $\restriction{\varphi'}{W_F}$ (elles sont de type orthogonal ou symplectique) et $\displaystyle S_{\pi}=\bigoplus_{\substack{q \in \N \\ (\pi,q) \in \Jord(\varphi')}} S_q$.\\

Pour tout $\pi \in I$, notons $$\ell_{\pi}=\left\{ \begin{array}{ll}
\ell_{i} & \text{si $\pi \in \left\{\pi_i, i \in \llbracket 1,r \rrbracket \right\} $} \\
0 & \text{sinon}
\end{array} \right. \;\; \text{et} \;\; m_{\pi}'=\left\{ \begin{array}{ll}
\displaystyle \sum_{\substack{q \in \N \\ (\pi,q) \in \Jord(\varphi')}} q& \text{si $\pi \in I_{\varphi'}$} \\
0 & \text{sinon}
\end{array} \right.$$ Puisque $\varphi'$ est sans multiplicité, il y a au plus un $\pi_i$ qui apparait dans cette décomposition. Ainsi, nous pouvons écrire : \begin{align*}
\Std_G \circ \varphi&=\bigoplus_{\pi \in I^{\O} \sqcup I^{\S}} \left(2\ell_{\pi} \pi \oplus \pi \boxtimes S_{\pi} \right) \bigoplus_{\pi \in I^{\GL}}  \ell_{\pi}(\pi \oplus \pi^{\vee}) \\
\Std_G \circ \restriction{\varphi}{W_F}&=\bigoplus_{\pi \in I^{\O} \sqcup I^{\S}} (2\ell_{\pi}+m_{\pi}') \pi  \bigoplus_{\pi \in I^{\GL}} \ell_{\pi} (\pi \oplus \pi^{\vee})
\end{align*}

Pour tout $\pi \in I^{\O} \sqcup I^{\S}$, notons $m_{\pi}=2 \ell_{\pi}+m_{\pi}'$. On obtient ainsi :

$$Z_{\widehat{G}_{*}}(\restriction{\varphi}{W_F}) \simeq \prod_{\pi \in I} \widehat{G}_{\pi} \quad \text{et} \quad
Z_{\widehat{L}_{*}}(\restriction{\varphi}{W_F}) \simeq \prod_{\pi \in I} \widehat{L}_{\pi},$$

avec $$\widehat{G}_{\pi}= \left\{ \begin{tabular}{ll}
$\O_{m_{\pi}}(\C)$ & si ($\widehat{G}$ est un groupe spécial orthogonal et $\pi \in I^{\O}$) ou ($\widehat{G}$ est un groupe symplectique et $\pi \in I^{\S}$)\\
$\Sp_{m_{\pi}}(\C)$ & si ($\widehat{G}$ est un groupe spécial orthogonal et $\pi \in I^{\S}$) ou ($\widehat{G}$ est un groupe symplectique et $\pi \in I^{\O}$) \\
$\GL_{\ell_{\pi}}(\C)$ & si $\pi \in I^{\GL}$
\end{tabular} \right.$$

$$\widehat{L}_{\pi}= \left\{ \begin{tabular}{ll}
$(\C^{\times})^{\ell_{\pi}} \times \O_{m_{\pi}'}(\C)$ & si ($\widehat{G}$ est un groupe spécial orthogonal et $\pi \in I^{\O}$) ou ($\widehat{G}$ est un groupe symplectique et $\pi \in I^{\S}$)\\
$(\C^{\times})^{\ell_{\pi}} \times \Sp_{m_{\pi}'}(\C)$ & si ($\widehat{G}$ est un groupe spécial orthogonal et $\pi \in I^{\S}$) ou ($\widehat{G}$ est un groupe symplectique et $\pi \in I^{\O}$) \\
$(\C^{\times})^{\ell_{\pi}}$ & si $\pi \in I^{\GL}$
\end{tabular} \right.$$

Remarquons que $\pi_i \not \in \Jord(\varphi')$, si et seulement si, $m_{\pi_i}' =0$. Décrivons dans le tableau ci-dessous le système de racines, ainsi que le groupe de Weyl (étendue) associé par la correspondance de Springer généralisée et les travaux de Lusztig.

$$
\renewcommand{\arraystretch}{1.3}
\begin{array}{|*7{>{\displaystyle}c|} c}
\cline{1-7}
\widehat{G}_{\pi} & \widehat{L}_{\pi} & \mbox{condition} & R & R_{\red} & W_{\widehat{L}_{\pi}^{\circ}}^{\widehat{G}_{\pi}^{\circ}}& W_{\widehat{L}_{\pi}}^{\widehat{G}_{\pi}} & \\ 
\cline{1-7} \cline{1-7}
\multirow{3}*{$\Sp_{m}(\C)$} & \multirow{3}*{$(\C^{\times})^\ell \times \Sp_{m'}(\C)$} & \ell=0 & \varnothing & \varnothing & \{1\} & \{1\}  & \\  \cline{3-7}
 &  & \ell\neq 0, m=0 & C_{\ell} & C_{\ell} & W_{C_{\ell}} & W_{C_{\ell}} & \\ \cline{3-7}
 &  & \ell\neq 0, m \neq 0 & BC_{\ell} & B_{\ell} & W_{B_{\ell}} & W_{B_{\ell}} &  \\ 
\cline{1-7}
\multirow{3}*{$\O_{m}(\C)$} & \multirow{3}*{$(\C^{\times})^\ell \times \O_{m'}(\C)$} & \ell=0 & \varnothing & \varnothing & \{1\} & \{1\} & \\ \cline{3-7}
 &  & \ell\neq 0, m=0 & D_{\ell} & D_{\ell} & W_{D_{\ell}} & W_{D_{\ell}} \rtimes (\Z/2\Z) & (*) \\ \cline{3-7}
 &  & \ell\neq 0, m \neq 0 & BC_{\ell} & B_{\ell} & W_{B_{\ell}} & W_{B_{\ell}} &  \\ 
\cline{1-7}
\multirow{2}*{$\GL_{\ell}(\C)$} & \multirow{2}*{$(\C^{\times})^{\ell}$} & \ell \leqslant 1 & \varnothing & \varnothing & \{1\} & \{1\} & \\ \cline{3-7}
&  & \ell \geqslant 2 & A_{\ell-1} & A_{\ell-1} & W_{A_{\ell-1}} & W_{A_{\ell-1}} & \\ 
\cline{1-7}
\end{array} $$

Si $\pi_{i}^{\vee} \simeq \pi_{i}$, soit $\zeta_{i}$ un cocaractère non-ramifié, tel que $(\pi_{i} \zeta_{i})^{\vee} \simeq \pi_{i} \zeta_{i}$. On voit donc que le choix de $\pi_{i}$ (resp. $\pi_{i}\zeta_{i}$) contribue à un groupe de Weyl de type $B_{2\ell_{\pi_{i}}+m_{\pi_{i}}'}$ (resp. $B_{2\ell_{\pi_{i} \zeta_{i}}+m_{\pi_{i}\zeta_{i}}'}$) (ou $C$ ou $D$ en fonction du type). Ainsi pour maximiser $\mathcal{W}_{\cjp,\varphi}^{\circ}$, on pourra supposer $m_{\pi_{i}}' \geqslant m_{\pi_{i}\zeta_{i}}'$. C'est ce qui correspond à la condition $a_{s_{\alpha}} \geqslant b_{s_{\alpha}}$ de \cite[1.7]{Heiermann:2011}. Enfin, si $\mathcal{T}_{\pi_i}=\mathcal{T}_{\pi_i^{\vee}}$, on voit que si $(\pi_{i} \chi)^{\vee} \not \simeq (\pi_{i} \chi)^{\vee} $, alors ce choix contribue à un groupe de Weyl $W_{A_{\ell_{\pi_i}-1}}$ tandis que le choix de $\pi_{i}^{\vee} \simeq \pi_{i}$ y contribue pour un groupe de Weyl $W_{B_{\ell_{\pi_i}}} \supseteq W_{A_{\ell_{\pi_i}-1}}$  (ou $C$ ou $D$). Ainsi, si $\mathcal{T}_{\pi_i}=\mathcal{T}_{\pi_i^{\vee}}$, alors on peut supposer $\pi_i \simeq \pi_i^{\vee}$.

La condition imposée sur la forme de $\sigma$, se traduit sur $\varphi$ par : \begin{itemize}
\item pour tout $i \in \llbracket 1,r\rrbracket$, si $\mathcal{T}_{\pi_i}=\mathcal{T}_{\pi_i^{\vee}}$, alors $\pi_{i}=\pi_{i}^{\vee}$, i.e. $\pi$ est de type symplectique ou orthogonal ;
\item pour tout $i,j \in \llbracket 1,r\rrbracket$, si $i \neq j$, alors $\pi_{i} \neq \pi_{j}$ et $\pi_{i} \neq \pi_{j}^{\vee}$ ;
\item pour tout $i \in \llbracket 1,r\rrbracket$, si $\pi_{i}=\pi_{i}^{\vee}$, alors $m_{\pi_{i}}' \geqslant m_{\pi_{i}\zeta_{i}}'$.
\end{itemize}

La forme particulière qu'on a imposé à $\varphi$ montre que $\mathcal{W}_{\cjp} \simeq N_{Z_{H_{\varphi}^{G}}(\restriction{\varphi}{\SL_2(\C)})}/Z_{H_{\varphi}^{L}}(\restriction{\varphi}{\SL_2(\C)})$. Écrivons $a=(a_{i,j})_{i \in \llbracket 1, r \rrbracket, j \in\llbracket 1,\ell_i \rrbracket}$, pour tout $i \in \llbracket 1, r \rrbracket, j \in\llbracket 1,\ell_i-1 \rrbracket$, notons $\widehat{s}_{i,j} \in N_{Z_{H_{\varphi}^{G}}(\restriction{\varphi}{\SL_2(\C)})}(A_{\widehat{L}})/Z_{H_{\varphi}^{L}}(\restriction{\varphi}{\SL_2(\C)})$ dont l'action sur $a$ permute les éléments $a_{i,j}$ et $a_{i,j+1}$ et $\widehat{s}_{i,\ell_i} \in N_{Z_{H_{\varphi}^{G_*}}(\restriction{\varphi}{\SL_2(\C)})}(A_{\widehat{L}})/Z_{H_{\varphi}^{L_*}}(\restriction{\varphi}{\SL_2(\C)})$ dont l'action sur $a$ permute $a_{i,\ell_{i}}$ et $a_{i,\ell_{i}}^{-1}$. L'action de ces éléments sur un paramètre de Langlands de $L$ de la forme $$\bigoplus_{i=1}^{r}\bigoplus_{j=1}^{\ell_i} \left(\pi_{i,j} \oplus \pi_{i,j}^{\vee} \right) \bigoplus \varphi',$$ est la suivante : \begin{itemize}
\item pour tout $i \in \llbracket 1, r \rrbracket, j \in\llbracket 1,\ell_i-1 \rrbracket$, $\widehat{s}_{i,j}$ permute $\pi_{i,j}$ et $\pi_{i,j+1}$ (et  $\pi_{i,j}^{\vee}$ et $\pi_{i,j+1}^{\vee}$) ;
\item pour tout $i \in \llbracket 1, r \rrbracket$, $\widehat{s}_{i,\ell_{i}}$ permute $\pi_{i,\ell_i}$ et $\pi_{i,\ell_i}^{\vee}$.
\end{itemize}
La table précédente nous montre en particulier que $\mathcal{W}_{\cjp}^{\circ}$ est le produit direct de $\mathcal{W}_{\cjp,i}^{\circ} (\simeq W_{\widehat{L}_{\pi_i}^{\circ}}^{\widehat{G}_{\pi_{i}}^{\circ}})$ et 
\begin{equation} \label{Wj}
\mathcal{W}_{\cjp,i}^{\circ} = \left\{ \begin{array}{ll}
\langle \widehat{s}_{i,j}, j \in \llbracket 1,\ell_{i} \rrbracket \rangle & \mbox{si $W_{\widehat{L}_{\pi_i}^{\circ}}^{\widehat{G}_{\pi_{i}}^{\circ}} $ est de type $B_{\ell_i}/C_{\ell_i}$} \\
\langle \widehat{s}_{i,j}, j \in \llbracket 1,\ell_{i}-1 \rrbracket,\,\, \widehat{s}_{i,\ell_{i}}\widehat{s}_{i,\ell_{i}-1}\widehat{s}_{i,\ell_{i}}^{-1} \rangle &  \mbox{si $W_{\widehat{L}_{\pi_i}^{\circ}}^{\widehat{G}_{\pi_{i}}^{\circ}}$ est de type $D_{\ell_i}$} \\
\langle \widehat{s}_{i,j}, j \in \llbracket 1,\ell_{i}-1 \rrbracket \rangle &\mbox{si $W_{\widehat{L}_{\pi_i}^{\circ}}^{\widehat{G}_{\pi_{i}}^{\circ}}$ est de type $A_{\ell_i-1}$} \\
\end{array} \right.
\end{equation}
Concernant $\mathcal{R}_{\cjp}$, notons $$
\widehat{C} =\{ i \in \llbracket 1,r \rrbracket \mid \pi_{i} \text{ vérifie $(*)$ dans le tableau précédent} \}, \;\;
\widehat{C}_{\pair} = \{ i \in \widehat{C} \mid n_i \equiv 0 \mod 2 \}, \;\;
\widehat{C}_{\imp} = \{ i \in \widehat{C} \mid n_i \equiv 1 \mod 2 \}.
$$ Il vient alors :
\begin{equation} \label{Rj}
\mathcal{R}_{\cjp}=\left\{ \begin{array}{ll}
\prod_{i \in \widehat{C}} \langle \widehat{s}_{i,\ell_{i}} \rangle & \text{si $G=\Sp_{N}(F)$ ou $G=\SO_{N}(F)$ avec $N$ impair} \\
\prod_{i \in \widehat{C}} \langle \widehat{s}_{i,\ell_{i}} \rangle & \text{si $G=\SO_{N}$ et $L=\GL_{d_1}(F)^{\ell_1} \times \ldots \times \GL_{d_r}(F)^{\ell_r} \times \SO_{N'}(F)$ avec $N$ pair et $N' \geqslant 4$} \\
\prod_{i \in \widehat{C}_{\pair}} \langle \widehat{s}_{i,\ell_{i}} \rangle  \times \langle \widehat{s}_{i,\ell_{i}}\widehat{s}_{j,\ell_{j}} \mid i,j \in \widehat{C}_{\imp} \rangle & \text{si $G=\SO_{N}(F)$ et $L=\GL_{d_1}(F)^{\ell_1} \times \ldots \times \GL_{d_r}(F)^{\ell_r}$ avec $N$ pair} \\
\end{array} \right.
\end{equation} 

On a un isomorphisme naturel $N_G(M)/M \simeq N_{\widehat{G}}(\widehat{M})/\widehat{M}$ et en comparant $(\ref{Ws})$ et $(\ref{Wj})$, puis $(\ref{Rs})$ et $(\ref{Rj})$, on voit qu'il se restreint en un isomorphisme $W_{\mathfrak{s}} \simeq \mathcal{W}_{\cjp}$. De plus, d'après la correspondance de Langlands pour le groupe linéaire, on a $\mathcal{X}(\GL_{n_i}(F))(\pi_{i}) \simeq \mathcal{X}(\GL_{n_i}(\C))(\pi_{i})$. Ainsi, la correspondance de Langlands pour les caractères et ce qui précède montre qu'on a une bijection $T_{\mathfrak{s}} \simeq \mathcal{T}_{\cjp}$ envoyant $\sigma\chi$ sur $\phi \widehat\chi$. Il reste à vérifier que l'action de cette bijection est équivariante pour les actions de $W_{\mathfrak{s}}$ sur $T_{\mathfrak{s}}$ et de $\mathcal{W}_{\cjp}$ sur $\mathcal{T}_{\cjp}$.

Si $\chi$ est un caractère non-ramifié de $L$ qu'on décompose en $(\chi_{i,j})$, où $\chi_{i,j}$ est un caractère non-ramifié de $\GL_{n_i}(F)$, alors \begin{itemize}
\item pour tout $i \in \llbracket 1,r \rrbracket, j \in \llbracket 1, \ell_{i}-1 \rrbracket$,\begin{align*}
 (\sigma \otimes \chi)^{s_{i,j}} &\simeq \pi_{1}\chi_{1,1} \boxtimes \ldots  \boxtimes \pi_{i} \chi_{i,j+1} \boxtimes \pi_{i} \chi_{i,j} \boxtimes \ldots \boxtimes \pi_{r} \chi_{r,\ell_{r}} \boxtimes \sigma' \\
 {}^{\widehat{s}_{i,j}} \left(\varphi \chi \right) &= \pi_{1} \chi_{1,1} \oplus \ldots	\oplus \pi_{i} \chi_{i,j+1} \oplus \pi_{i} \chi_{i,j} \oplus \ldots \oplus \pi_{r} \chi_{r,\ell_{r}} \oplus \varphi' \\
 &\quad \oplus \pi_{r}^{\vee} \chi_{r,\ell_{r}}^{-1} \oplus \ldots \oplus \pi_{i}^{\vee} \chi_{i,j} \oplus  \pi_{i}^{\vee} \chi_{i,j+1}^{-1}  \oplus  \ldots \pi_{1}^{\vee} \chi_{1,1}^{-1}
 \end{align*}
 \item pour tout $i \in \llbracket 1,r \rrbracket$, \begin{align*}
 (\sigma \otimes \chi)^{s_{i,\ell_i}} &\simeq \pi_{1}\chi_{1,1} \boxtimes \ldots  \boxtimes \pi_{i} \chi_{i,\ell_{i}-1} \boxtimes \pi_{i}^{\vee} \chi_{i,\ell_{i}}^{-1} \boxtimes \ldots \boxtimes \pi_{r} \chi_{r,\ell_{r}} \boxtimes \sigma' \\
 {}^{\widehat{s}_{i,\ell_{i}}} \left(\varphi \chi \right) &= \pi_{1} \chi_{1,1} \oplus \ldots	\oplus \pi_{i} \chi_{i,\ell_{i}-1} \oplus \pi_{i}^{\vee} \chi_{i,\ell_{i}}^{-1} \oplus \ldots \oplus \pi_{r} \chi_{r,\ell_{r}} \oplus \varphi' \\
 &\quad \oplus \pi_{r}^{\vee} \chi_{r,\ell_{r}}^{-1} \oplus \ldots \oplus \pi_{i} \chi_{i,\ell_{i}} \oplus  \pi_{i}^{\vee} \chi_{i,\ell_{i}-1}^{-1}  \oplus  \ldots \pi_{1}^{\vee} \chi_{1,1}^{-1} 
\end{align*}
\end{itemize}

Ceci montre que la bijection $T_{\mathfrak{s}} \simeq \mathcal{T}_{\cjp}$ est équivariante pour les actions de $W_{\mathfrak{s}}$ et $\mathcal{W}_{\cjp}$. Ce qui achève la démonstration.
\end{proof}

\subsection{Représentations et algèbres de Hecke affines}\phantomsection\label{equivcat}

Soit $\mathfrak{s}=[L,\sigma] \in \mathcal{B}(G)$. On suppose qu'on a décomposé $L$ et $\sigma$ comme dans la démonstration du théorème précédent. Dans \cite{Heiermann:2011}, Heiermann associe à la paire inertielle $\mathfrak{s}$, une donnée radicielle basée $\Psi_{\mathfrak{s}}=(\Lambda_{\mathfrak{s}},\Sigma_{\mathfrak{s}},\Lambda_{\mathfrak{s}}^{\vee},\Sigma_{\mathfrak{s}}^{\vee},\Delta_{\mathfrak{s}})$ et des fonctions paramètres (au sens des algèbres de Hecke affine) $(\lambda_{\mathfrak{s}},\lambda_{\mathfrak{s}}^{*})$, où $\Lambda_{\mathfrak{s}}$ est un sous-réseau de $\Lambda(L)$, $\Sigma_{\mathfrak{s}}$ un système de racines et $\Delta_{\mathfrak{s}}$ une base de racines simples. De plus, l'image de $\Lambda_{\mathfrak{s}}^{\vee} \otimes_{\Z} \C$ par l'application (\ref{appcarnr}) est isomorphe à $T_{\mathfrak{s}}$.\\

On a des décompositions $\Lambda_{\mathfrak{s}}=\bigoplus_{i=1}^{r} \Lambda_{\mathfrak{s},i}$ et $\Sigma_{\mathfrak{s}}=\sqcup_{i=1}^{r} \Sigma_{\mathfrak{s},i}$ en systèmes de racines irréductibles. De plus, le groupe de Weyl associé à $\Sigma_{\mathfrak{s},i}$ est $W_{\mathfrak{s},i}^{\circ}$ et $\Sigma_{\mathfrak{s},i}$ est de type $A,B,C$ ou $D$ (voir \cite[1.13]{Heiermann:2011}).\\ Soit $i \in \llbracket 1,r \rrbracket$ et notons $\Delta_{\mathfrak{s},i}=\Delta_{\mathfrak{s}} \cap \Sigma_{\mathfrak{s},i}$. Si $\sigma_i$ est autoduale, soit $\zeta_i$ un caractère non-ramifié de $\GL_{n_i}(F)$ tel que $\sigma_i \zeta_i$ soit autoduale, non isomorphe à $\sigma_i$. Soit $x_i^{+} \in \R_{+} $ (resp. $x_i^{-} \in \R_{+} $), l'unique réel positif tel que $\sigma_i  |\,\,|^{x_i^+} \rtimes \sigma'$ (resp. $\sigma_i \zeta_i  |\,\,|^{x_i^-} \rtimes \sigma'$) soit réductible. On peut supposer que $x_i^+ \geqslant x_i^-$. \begin{itemize}
\item Si $\Sigma_{\mathfrak{s},i}$ est de type $A,C$ ou $D$, pour tout $\alpha \in \Delta_{\mathfrak{s},i}$, $$\lambda_{\mathfrak{s}}(\alpha)=1.$$
\item Si $\Sigma_{\mathfrak{s},i}$ est de type $B$, pour tout $\alpha \in \Delta_{\mathfrak{s},i}$ racine longue, $$\lambda_{\mathfrak{s}}(\alpha)=1.$$  Si $\alpha_i  \in \Delta_{\mathfrak{s},i}$ est la racine courte, alors $$\lambda_{\mathfrak{s}}(\alpha_i)=x_i^+ + x_i^-, \quad \lambda_{\mathfrak{s}}^{*}(\alpha_i)=x_i^+-x_i^-.$$
\end{itemize}

On peut dès lors considérer l'algèbre de Hecke affine $\mathcal{H}_{i}$ définie par la donnée radicielle basée $\Psi_{\mathfrak{s},i}=(\Lambda_{\mathfrak{s},i},\Sigma_{\mathfrak{s},i},\Lambda_{\mathfrak{s},i}^{\vee},\Sigma_{\mathfrak{s},i}^{\vee},\Delta_{\mathfrak{s},i})$. C'est une $\C[v_{i}^{ \pm 1}]$-algèbre associative unitaire (où $v_{i}$ est une indéterminée) définie par les générateurs $(T_w)_{ w \in W_{\mathfrak{s},i}^{\circ}}$ et $(\theta_x)_{x \in \Lambda_{\mathfrak{s},i}}$ vérifiant certaines relations qu'il n'est pas nécessaire d'expliciter (voir \cite[7.1]{Heiermann:2011}).\\

Soient $t_{i}$ l'ordre du groupe cyclique $\mathcal{X}(\GL_{n_i}(F))(\sigma_{i})=\left\{ \chi \in \mathcal{X}(\GL_{n_i}(F)) \mid \sigma \simeq \sigma \otimes \chi \right\}$ et $\restriction{\mathcal{H}_{i}}{v_{i}=q^{t_i/2}}$ l'algèbre obtenue par spécialisation de l'indéterminée $v_{i}$ en $q^{t_i/2}$. Notons enfin $$\mathcal{H}_{\mathfrak{s}}=\bigotimes_{i=1}^{r} \restriction{\mathcal{H}_{i}}{v_{i}=q^{t_i/2}} \,\, \text{et} \,\,  \mathcal{H}_{\mathfrak{s}}'=\mathcal{H}_{\mathfrak{s}} \rtimes \C[R_{\mathfrak{s}}].$$ En notant $\module\left(\mathcal{H}_{\mathfrak{s}}'\right)$ la catégorie des modules à droite sur $\mathcal{H}_{\mathfrak{s}}'$, on a :
\begin{theo}[Heiermann,{\cite[7.7,7.8]{Heiermann:2011},\cite{Heiermann:2012}}]\phantomsection\label{thmHeiermann}
Il y a une équivalence de catégories $$\Rep(G)_{\mathfrak{s}} \simeq \module\left(\mathcal{H}_{\mathfrak{s}}'\right).$$ De plus, cette équivalence de catégories préserve les objets tempérés et ceux de la série discrète.
\end{theo}

\subsection{Réduction à une algèbre de Hecke graduée}

En introduisant les algèbres de Hecke graduées, Lusztig a montré qu'on pouvait étudier certains modules d'une algèbre de Hecke affine en se ramenant à des modules d'une algèbre de Hecke graduée. Nous allons suivre ce procédé concernant l'algèbre de Hecke affine étendue du théorème précédent.\\

On reprend les notations de la section précédente. On rappelle que $T_{\mathfrak{s}} \simeq \Lambda_{\mathfrak{s}} \otimes_{\Z} \C^{\times}$. Notons $\mathfrak{t}_{\mathfrak{s}}=\Lambda_{\mathfrak{s}}^{\vee} \otimes_{\Z} \C$, $\mathfrak{t}_{\mathfrak{s}}^{*}=\Lambda_{\mathfrak{s}} \otimes_{\Z} \C$ et $\mathfrak{t}_{\mathfrak{s},\R}=\Lambda_{\mathfrak{s}}^{\vee}\otimes_{\Z} \R$. Soient $r$ une indéterminée et $S(\mathfrak{t}_{\mathfrak{s}}^{*})$ l'algèbre symétrique de $\mathfrak{t}_{\mathfrak{s}}^{*}$.\\

Soient $\zeta \in T_{\mathfrak{s}}$, $\mathcal{O}=W_{\mathfrak{s}} \cdot \zeta \in T_{\mathfrak{s}}/W_{\mathfrak{s}}$ l'orbite de $\zeta$ sous $W_{\mathfrak{s}}$. À la donnée radicielle $\mathcal{R}_{\mathfrak{s}}$ et à l'orbite $\mathcal{O}$, Lusztig a associé une algèbre de Hecke graduée $\mathbb{H}_{\mathfrak{s},\mathcal{O}}'$ qui est une $\C[r]$-algèbre engendrée par $(t_w)_{w \in W_{\mathfrak{s}}}$, $S(\mathfrak{t}_{\mathfrak{s}}^{*})$ et un ensemble d'idempotents orthogonaux $(E_{t})_{t \in \mathcal{O}}$, qui vérifient des relations faisant intervenir une fonction paramètre $\mu_{\mathcal{O}}$ qui s'exprime en fonction de $\lambda_{\mathfrak{s}}$ et $\lambda_{\mathfrak{s}}^{*}$. Pour plus de précisions, voir \cite{Lusztig:1989}, \cite{Barbasch:1993} et en particulier \cite{Barbasch:2013} qui nous a été très utile. \\

Par ailleurs, la construction suivante, permet en quelque sorte de supposer que l'orbite $\mathcal{O}$ est réduit à un singleton.\\

On définit la donnée radicielle basée suivante $\Psi_{\mathfrak{s},\zeta}=(\Lambda_{\mathfrak{s}},\Sigma_{\mathfrak{s},\zeta},\Lambda_{\mathfrak{s}}^{\vee},\Sigma_{\mathfrak{s},\zeta}^{\vee},\Delta_{\mathfrak{s},\zeta})$ par \begin{align*}
\Sigma_{\mathfrak{s},\zeta} &= \left\{ \alpha \in \Sigma_{\mathfrak{s}} \bigm| \theta_{\alpha}(\zeta)=  \left\{ \begin{array}{ll} 1 & \text{si $\alpha^{\vee} \not \in 2 \Lambda_{\mathfrak{s}}^{\vee}$} \\ \pm 1 & \text{si $\alpha^{\vee} \in 2 \Lambda_{\mathfrak{s}}^{\vee}$} \end{array} \right. \right\} ; \\
\Sigma_{\mathfrak{s},\zeta}^{+} &= \Sigma_{\mathfrak{s},\zeta}\cap \Sigma_{\mathfrak{s}}^{+} ; \\
\Sigma_{\mathfrak{s},\zeta}^{\vee} &= \left\{ \alpha^{\vee}, \alpha \in \Sigma_{\mathfrak{s},\zeta} \right\}  ; \\
\Delta_{\mathfrak{s},\zeta} &= \left\{ \alpha \in \Sigma_{\zeta}^{+} \mid \alpha \,\,\text{ n'est pas de la forme  $\alpha'+\alpha''$ avec $\alpha',\alpha'' \in \Sigma_{\mathfrak{s},\zeta}^{+}$}\right\} ;\\
W_{\mathfrak{s},\zeta} &= \langle s_{\alpha}, \alpha \in \Delta_{\mathfrak{s},\zeta} \rangle ;\\
R_{\mathfrak{s},\zeta} &= \left\{w \in Z_{W_{\mathfrak{s}}}(\zeta) \mid w\Sigma_{\mathfrak{s},\zeta}^{+}=\Sigma_{\mathfrak{s},\zeta}^{+}  \right\} 
\end{align*}

On peut considérer l'algèbre de Hecke graduée étendue $$\mathbb{H}_{\mathfrak{s},\zeta,\mu_{\zeta}}'=\mathbb{H}_{\mathfrak{s},\zeta,\mu_{\zeta}} \rtimes \C[R_{\mathfrak{s},\zeta}].$$ C'est une $\C[r]$-algèbre engendrée par $(t_w)_{w \in W_{\mathfrak{s},\zeta}}$, $(j_r)_{r \in R_{\mathfrak{s},\zeta}}$ et $S(\mathfrak{t}_{\mathfrak{s}}^{*})$ vérifiant les relations :
\begin{itemize}
\item pour tout $w,w' \in W_{\mathfrak{s},\zeta}$, $t_w t_{w'}=t_{w w'}$ ;
\item pour tout $\alpha \in \Delta_{\mathfrak{s},\zeta}$, $\gamma \in \mathfrak{t}_{\mathfrak{s}}^{*}$, $\gamma t_{s_{\alpha}}-t_{s_{\alpha}} s_{\alpha}(\gamma)=r \mu_{\zeta}(\alpha) \langle \gamma , \alpha^{\vee} \rangle$ où $$\mu_{\zeta}(\alpha)= \left\{ \begin{array}{*2{>{\displaystyle}l}}
2\lambda_{\mathfrak{s}}(\alpha) & \text{si $\alpha^{\vee} \not \in 2 \Lambda_{\mathfrak{s}}^{\vee}$}\\
\lambda_{\mathfrak{s}}(\alpha)+\lambda_{\mathfrak{s}}^{*}(\alpha) \theta_{-\alpha}(\zeta) & \text{si $\alpha^{\vee}  \in 2 \Lambda_{\mathfrak{s}}^{\vee}$}
\end{array}\right. .$$
\end{itemize}

Soit $\{w_1,\ldots,w_m\}$ un ensemble de représentants de $W_{\mathfrak{s}}/Z_{W_{\mathfrak{s}}}(\zeta)$ (avec $w_1=1$). Pour tout $i,j \in \llbracket 1,m \rrbracket$, posons $$E_{i,j}=t_{w_{i}^{-1}}E_{\zeta}t_{w_j}, \,\, \mathcal{M}_{n}=\C[E_{i,j}]_{1 \leqslant i,j \leqslant m}, \,\, \mathcal{C}_{n}=\sum_{i=1}^{n} \C E_{i,1} \,\, \text{et} \,\,\mathcal{E}=\C[E_{w_i \cdot \zeta}]_{1 \leqslant i \leqslant m }.$$

D'après \cite[4.5]{Lusztig:1989}, le centre de $\mathbb{H}_{\mathfrak{s},\mathcal{O}}$ est $Z=\left( S(\mathfrak{t}_{\mathfrak{s}}^{*}) \otimes_{\C} \C[r] \otimes_{\C} \mathcal{E} \right)^{W_\mathfrak{s}}$. Ainsi, tout caractère central $\omega : Z \rightarrow \C$ correspond à un élément $(\sigma,r_{0}) \in \mathfrak{t}_{\mathfrak{s}}/Z_{W_{\mathfrak{s}}}(\zeta) \times \C$. De plus, si $\omega : Z \rightarrow \C$ est un caractère central, nous noterons $\module(\mathbb{H}_{\mathfrak{s},\mathcal{O}})_{\omega}$ la catégorie des $\mathbb{H}_{\mathfrak{s},\mathcal{O}}$-modules admettant $\omega$ pour caractère central.\\

On suppose à présent que $\zeta \in T_{\mathfrak{s}}$ est un élément elliptique. L'application $$e_{\zeta}: \begin{array}[t]{rcl}
\mathfrak{t}_{\mathfrak{s},\R} \oplus \C &  \longrightarrow & T_{\mathfrak{s}} \times \C^{\times} \\
(\nu,r_{0}) & \longmapsto &(\zeta e^{\nu},e^{r_0})
\end{array},$$ est $Z_{W_{\mathfrak{s}}}(\zeta)$-invariante et elle fait correspondre $W_{\mathfrak{s}} \cdot (\zeta e^{\nu},e^{r_0})$ à $Z_{W_{\mathfrak{s}}}(\zeta) \cdot (\nu,r_{0})$. Ainsi, elle induit une bijection entre les caractères centraux de $\mathcal{H}_{\mathfrak{s}}'$ de partie elliptique dans $\mathcal{O}$ et les caractères centraux hyperboliques de $\mathbb{H}_{\mathfrak{s},\mathcal{O}}'$.

\begin{theo}[Lusztig]\phantomsection\label{theoaffingrad} Soient $\zeta \in T_{\mathfrak{s}}$ un élément elliptique, $\nu \in \mathfrak{t}_{\mathfrak{s},\R}$, $r_{0}>1$ un réel. Soient $\overline{\omega}=Z_{W_{\mathfrak{s}}}(\zeta) \cdot (\nu,r_{0})$ un caractère central hyperbolique de $\mathbb{H}_{\mathfrak{s},\mathcal{O}}'$ et $\omega=W_{\mathfrak{s}} \cdot (\zeta e^{\nu},e^{r_0})$ le caractère central de $\mathcal{H}_{\mathfrak{s}}'$ correspond à $\overline{\omega}$ par l'application $e_{\zeta}$.
\begin{enumerate}
\item \cite[8.6]{Lusztig:1989} L'algèbre $\mathcal{M}_{n}$ est un algèbre de matrices et on a un isomorphisme :  $$\mathbb{H}_{\mathfrak{s},\mathcal{O}}' \simeq \mathcal{M}_{n} \otimes_{\C} \mathbb{H}_{\mathfrak{s},\zeta,\mu_{\zeta}}' = \mathcal{M}_{n} \otimes_{\C} (\mathbb{H}_{\mathfrak{s},\zeta,\mu_{\zeta}} \rtimes \C[R_{\mathfrak{s},\zeta}]).$$
\item Le foncteur $$\mathcal{F} : \begin{array}[t]{rcl}
\module(\mathbb{H}_{\mathfrak{s},\zeta,\mu_{\zeta}}') & \longrightarrow & \module(\mathbb{H}_{\mathfrak{s},\mathcal{O}}') \\
\mathcal{V} & \longmapsto & \mathcal{C}_{n} \otimes_{\C} \mathcal{V}
\end{array},$$ est une équivalence de catégories et en tant que $\C[W_{\mathfrak{s}}]$-modules, $\mathcal{F}(\mathcal{V})=\Ind_{Z_{W_{\mathfrak{s}}}(\zeta)}^{W_{\mathfrak{s}}} \mathcal{V}$.
\item \cite[9.3]{Lusztig:1989} Il y a une équivalence de catégories $$\module(\mathcal{H}_{\mathfrak{s}}')_{\omega} \simeq \module(\mathbb{H}_{\mathfrak{s},\mathcal{O}}')_{\overline{\omega}},$$ qui combinée avec la précédente donne lieu à l'équivalence de catégories $$\module(\mathcal{H}_{\mathfrak{s}}')_{\omega} \simeq \module(\mathbb{H}_{\mathfrak{s},\zeta,\mu_{\zeta}}')_{\overline{\omega}}.$$
\item \cite[4.3,4.4]{Lusztig:2002a} Cette dernière équivalence de catégories, préservent les modules tempérées et ceux de la série discrète.
\end{enumerate}
\end{theo}

À présent, on note $\zeta$ un caractère non-ramifié \emph{unitaire} de $L$. La fonction paramètre $\mu_{\zeta}$ associée à $(\lambda_{\mathfrak{s}},\lambda_{\mathfrak{s}}^{*})$ est définie de la façon suivante : pour tout $\alpha \in \Delta_{\zeta}$, $$\mu_{\zeta}(\alpha)= \left\{ \begin{array}{*2{>{\displaystyle}l}}
2\lambda_{\mathfrak{s}}(\alpha) & \mbox{si  } \alpha^{\vee} \not \in 2\Lambda_{\mathcal{O}}^{\vee}\\
\lambda_{\mathfrak{s}}(\alpha)+\lambda_{\mathfrak{s}}^{*}(\alpha) \theta_{-\alpha}(\zeta) & \mbox{si  }  \alpha^{\vee}  \in 2\Lambda_{\mathcal{O}}^{\vee}
\end{array}\right. $$
Puisque $s_{\alpha}(\zeta) = \zeta$, $ \theta_{-\alpha}(\zeta)= \pm 1$. Plus précisément, $$ \theta_{-\alpha}(\zeta)=\left\{ \begin{array}{ll}
1 & \mbox{si $(\sigma \otimes \zeta)^{s_{\alpha}} \simeq \sigma $} \\
-1 & \mbox{si $(\sigma \otimes \zeta)^{s_{\alpha}} \not \simeq \sigma$} 
\end{array}
\right.$$

On peut décomposer $\Lambda_{\mathfrak{s}} = \bigoplus_{i=1}^{r} \Lambda_{\mathfrak{s},i}= \bigoplus_{i=1}^{r}\bigoplus_{j=1}^{\ell_{i}} \Lambda_{i,j}$, où $\Lambda_{i,j}$ est un sous-réseau de $\Lambda(\GL_{n_i})$ d'indice $t_{i}$. Pour tout $\nu_{i,j} \in \Lambda_{i,j} \otimes_{\Z} \R $, notons $\varsigma_{i,j}=(\zeta_{i,j}\chi_{\nu_{i,j}},q^{t_i/2})$, $\overline{\varsigma}_{i,j}=(\nu_{i,j},t_{i}(\log q)/2)$, $\varsigma=(\varsigma_{i,j})$, $\overline{\varsigma}=(\overline{\varsigma}_{i,j})$. Soient $\omega=W_{\mathfrak{s}} \cdot \varsigma$ le caractère central de $\mathcal{H}_{\mathfrak{s}}'$ de partie elliptique dans l'orbite de $\zeta$ et $\overline{\omega}$ le caractère central hyperbolique correspondant.\\

D'après les résultats de Lusztig (Théorème \ref{theoaffingrad}), on a une équivalence de catégories : $$\module(\mathcal{H}_{\mathfrak{s}}')_{\omega} \simeq \module(\mathbb{H}_{\mathfrak{s},\zeta,\mu_{\zeta}}')_{\overline{\omega}},$$

induisant des bijections : $$\Irr(G)_{\omega} \leftrightarrow \Irr(\mathcal{H}_{\mathfrak{s}}')_{\omega} \leftrightarrow \Irr(\mathbb{H}_{\mathfrak{s},\zeta,\mu_{\zeta}}')_{\overline{\omega}}.$$ La première bijection résulte de l'équivalence de catégories prouvée par Heiermann, la deuxième par les réductions dues à Lusztig.\\

Les valeurs de la fonction paramètre $\mu_{\zeta}$ se lisent sur les diagrammes suivants (il s'agit du diagramme de Dynkin associé à $\Sigma_{i,\zeta}=\Sigma_{\zeta} \cap \Sigma_{\mathfrak{s},i}$) : 

\begin{center}
\renewcommand{\arraystretch}{1.5}
\begin{tabular}{ | >{\centering\arraybackslash}m{.11\textwidth} | >{\centering\arraybackslash}m{.45\textwidth}| l }
\hhline{--~}
type de $\Sigma_{i,\zeta}$ & diagramme de Dynkin pondéré par $\mu_{\zeta}(\alpha)$, $\alpha \in \Sigma_{i,\zeta}$ \\
\hhline{==~}
 $A_{\ell_i-1}$ & \begin{tikzpicture}

	\draw (0.1,0) -- (0.9,0);
	\draw (1.1,0) -- (1.9,0);
	\draw[dotted] (2.1,0) -- (2.9,0);
	\draw (3.1,0) -- (3.9,0);
		
	\draw (0,0) circle(.1);
	\draw (1,0) circle(.1);
	\draw (2,0) circle(.1);
	\draw (3,0) circle(.1);
	\draw (4,0) circle(.1);
	
	 \node at (0,0.3) {$2$};
	 \node at (1,0.3) {$2$};
	 \node at (2,0.3) {$2$};
	 \node at (3,0.3) {$2$};
	 \node at (4,0.3) {$2$};	
\end{tikzpicture} &  \\  \cline{1-2}
$B_{\ell_{i}}$ & \begin{tikzpicture}

	\draw (0.1,0) -- (0.9,0);
	\draw (1.1,0) -- (1.9,0);
	\draw[dotted] (2.1,0) -- (2.9,0);
	\draw (3,0.1) -- (4,0.1);
	\draw (3,-0.1) -- (4,-0.1);
		
	\draw (0,0) circle(.1);
	\draw (1,0) circle(.1);
	\draw (2,0) circle(.1);
	\draw (3,0) circle(.1);
	\draw (4,0) circle(.1);
	
	 \node at (0,0.3) {$2$};
	 \node at (1,0.3) {$2$};
	 \node at (2,0.3) {$2$};
	 \node at (3,0.3) {$2$};
	 \node at (4,0.3) {$2x_{i}^{\pm}$};	
\end{tikzpicture} & si $\sigma_{i} \in\Jord(\sigma')$ \\ \cline{1-2}
$B_{\ell_i}/C_{\ell_i}$ & \begin{tikzpicture}

	\draw (0.1,0) -- (0.9,0);
	\draw (1.1,0) -- (1.9,0);
	\draw[dotted] (2.1,0) -- (2.9,0);
	\draw (3,0.1) -- (4,0.1);
	\draw (3,-0.1) -- (4,-0.1);
		
	\draw (0,0) circle(.1);
	\draw (1,0) circle(.1);
	\draw (2,0) circle(.1);
	\draw (3,0) circle(.1);
	\draw (4,0) circle(.1);
	
	 \node at (0,0.3) {$2$};
	 \node at (1,0.3) {$2$};
	 \node at (2,0.3) {$2$};
	 \node at (3,0.3) {$2$};
	 \node at (4,0.3) {$2$};	
\end{tikzpicture} & si $\sigma_{i} \not\in\Jord(\sigma')$ \\ \cline{1-2}
$ D_{\ell_i}$ & \begin{tikzpicture}

	\draw (0.1,0) -- (0.9,0);
	\draw (1.1,0) -- (1.9,0);
	\draw[dotted] (2.1,0) -- (2.9,0);
	\draw (3.0866,-0.05) -- (3.9134,-0.45);
	\draw (3.0866,0.05) -- (3.9134,0.45);
		
	\draw (0,0) circle(.1);
	\draw (1,0) circle(.1);
	\draw (2,0) circle(.1);
	\draw (3,0) circle(.1);
	\draw (4,0.5) circle(.1);
	\draw (4,-0.5) circle(.1);
	
	 \node at (0,0.3) {$2$};
	 \node at (1,0.3) {$2$};
	 \node at (2,0.3) {$2$};
	 \node at (3,0.3) {$2$};
	 \node at (4.2,0.8) {$2$};	
 \node at (4.2,-0.8) {$2$};	
\end{tikzpicture}  & si $\sigma_{i} \in\Jord(\sigma')$ \\ \cline{1-2}
\end{tabular}\captionof{table}{Système de racines et paramètres associés à $\Sigma_{i,\zeta}$}\phantomsection\label{fonctionparametrerep}
\end{center}

\subsection{Classification en terme de module simples d'une algèbre de Hecke graduée étendue}

Rappelons brièvement la définition des algèbres de Hecke graduées associées à des triplets cuspidaux et le théorème de classification qui s'ensuit.\\

Soient $H$ un groupe réductif connexe complexe, $P=LU$ un sous-groupe parabolique de $H$. Notons $\mathfrak{h}$ l'algèbre de Lie de $H$ et $\mathfrak{p}=\mathfrak{l}+\mathfrak{u}$ celle de $P$. Soit $\mathcal{C} \subset \mathfrak{l}$ une $L$-orbite nilpotente supportant un système local irréductible cuspidal $L$-équivariant noté $\mathcal{L}$. Notons $\mathfrak{t}=[L,\mathcal{C},\mathcal{L}]$ le triplet cuspidal correspondant.\\ Soient $T$ le plus grand tore central dans $L$, c'est-à-dire $T=Z_{L}^{\circ}$ et $\mathfrak{t}$ son algèbre de Lie. On sait que $W_{L}^{H}=N_{H}(T)/L$ est un groupe de Coxeter. Plus précisément, pour toute forme linéaire $\alpha : \mathfrak{l} \longrightarrow \C$, notons $$\mathfrak{h}_{\alpha}=\{ X \in \mathfrak{h} \mid \forall t \in \mathfrak{t}, \, [t,X]=\alpha(t) X\},$$ son espace de poids. Notons $$\Sigma=\{\alpha \in \mathfrak{t}^{*} \mid \alpha \neq 0, \, \mathfrak{h}_{\alpha} \neq 0 \},$$ et $$\Sigma^{+}= \{ \alpha \in \Sigma \mid \mathfrak{h}_{\alpha} \subset \mathfrak{u} \}.$$ Soient $P_1,\ldots,P_m$ les sous-groupes paraboliques de $H$ qui contiennent strictement $P$ et qui sont minimal pour cette propriété. Pour tout $i \in \llbracket 1,m \rrbracket$, notons $L_i$ le sous-groupe de Levi de $P_i$ qui contient $L$ et $\Sigma_i^{+}=\{\alpha \in \Sigma \mid \restriction{\alpha}{\mathfrak{z}(\mathfrak{l}_i)}=0\}$. On a alors $\mathfrak{l}_i \cap \mathfrak{u} = \bigoplus_{\alpha \in \Sigma_{i}^{+}} \mathfrak{h}_{\alpha}$ et $\Sigma$ est un système de racines (non nécessairement réduit) dans $\mathfrak{t}^{*}$ qui est engendré par les formes linéaires de $\mathfrak{t}$ qui sont nulles sur le centre de $\mathfrak{h}$. De plus, $\Sigma_{i}^{+}$ contient un unique élément $\alpha_i$ tel que $\alpha_i/2 \not \in \Sigma$. L'ensemble $\Delta=\{\alpha_1, \ldots, \alpha_m\}$ est un système de racines simples de $\Sigma$. De plus,$W_{L}^{H}$ est le groupe de Coxeter engendré par les réflexions simples $s_i$, où $s_i$ est l'unique élément non trivial de $N_{L_i}(T)/L$.\\ Soit $N \in \mathcal{C}$. Pour tout $\alpha \in \Delta$, soit $\mu_{\mathfrak{t}}(\alpha) \geqslant 2$ le plus petit entier tel que $$\ad(N)^{\mu_{\mathfrak{t}}(\alpha)-1} : \mathfrak{l}_{\alpha} \cap \mathfrak{u} \longrightarrow  \mathfrak{l}_{\alpha} \cap \mathfrak{u},$$ est nul. Si $\alpha \in \Delta$ est conjuguée à $\alpha' \in \Delta$ dans $W_{L}^{H}$, alors $\mu_{\mathfrak{t}}(\alpha)=\mu_{\mathfrak{t}}(\alpha')$. La fonction $\mu_{\mathfrak{t}} : \Delta \longrightarrow \N$ est une fonction paramètre dans le sens vu dans la section précédente.\\

Considérons l'algèbre de Hecke graduée associé à ce triplet cuspidal. C'est le $\C[r]$-module $\C[W] \otimes_{\C} S(\mathfrak{t}^{*}) \otimes_{\C} \C[r]$ que l'on note $\mathbb{H}_{\mu_{\mathfrak{t}}}=\mathbb{H}_{\mu_{\mathfrak{t}}}(\mathfrak{t}^{*},\Sigma)$. Elle est engendrée par $(t_w)_{w \in W}$ et $(\gamma)_{\gamma \in \mathfrak{t}^{*}}$ vérifiant les relations 
\begin{itemize}
\item pour tout $w,w' \in W$, $t_w t_{w'}=t_{w w'}$ ;
\item pour tout $\alpha \in \Delta$, $\gamma \in \mathfrak{t}^{*}$, $\gamma t_{s_{\alpha}}-t_{s_{\alpha}} s_{\alpha}(\gamma)=r \mu_{\mathfrak{t}}(\alpha) \langle \gamma , \alpha^{\vee} \rangle$.\\
\end{itemize}

Soit $x \in \mathfrak{h}$ un élément nilpotent et soit $\Irr(A_{H}(x))_{\mathfrak{t}}$ l'ensemble des représentations irréductibles $\widetilde{\eta}$ de $A_{H}(x)$ telles que $\Phi_{H}(\mathcal{C}_X^{H},\widetilde{\eta})=\mathfrak{t}$, i.e. telles que $(\mathcal{C}_X^{H},\widetilde{\eta})$ soit associé par la correspondance de Springer généralisée à $\mathfrak{t}$. Soit $(s,r_{0}) \in \mathfrak{h} \oplus \C$ un élément semi-simple tel que $[s,x]=2r_{0}x$. On a une injection de $A_{H}(x,s)$ dans $A_{H}(x)$ et on note $\Irr(A_{H}(x,s))_{\mathfrak{t}}$ l'ensemble des représentations irréductibles de $A_{H}(x,s)$ apparaissant dans la restriction à $A_{H}(x,s)$ d'une représentation $\widetilde{\eta} \in \Irr(A_{H}(x))_{\mathfrak{t}}$.\\

En utilisant des techniques de cohomologie équivariante, Lusztig définit des modules standards $M_{\mathfrak{t}}(x,r_0,s,\eta)$ qui sont des $\mathbb{H}_{\mu_{\mathfrak{t}}}$-modules et qui vont classifier les objets que l'on veut. Pour être plus précis, on a le théorème suivant :

\begin{theo}[Lusztig]\phantomsection\label{thmlusztigalghecke}
\begin{enumerate}
\item \cite[8.10]{Lusztig:1988} $M_{\mathfrak{t}}(x,s,r_0,\eta) \neq 0$, si et seulement si, $\eta \in \Irr(A_{H}(x,s))_{\mathfrak{t}}$\\
\item \cite[8.15]{Lusztig:1988} Tout $\mathbb{H}_{\mu_{\mathfrak{t}}}$-module simple sur lequel $r$ agit par $r_0$ est un quotient $M_{\mathfrak{t}}(x,s,r_0,\eta)$ d'un $M_{\mathfrak{t}}(x,s,r_0,\eta)$, où $\eta \in \Irr(A_{H}(x,s))_{\mathfrak{t}}$\\
\item \cite[8.17]{Lusztig:1988}, \cite[8.18]{Lusztig:1995a} L'ensemble des classes de $\mathbb{H}_{\mu_{\mathfrak{t}}}$-module simple de caractère central $(s,r_0)$ est en bijection avec $$\mathcal{M}_{(s,r_{0})}=\{(x,\eta) \mid x \in \mathfrak{h}, \,\, [s,x]=2r_0 x, \eta \in \Irr(A_{H}(x,s))_{\mathfrak{t}} \}$$
\item \cite[1.21]{Lusztig:2002a} Le $\mathbb{H}_{\mu_{\mathfrak{t}}}$-module simple $\overline{M_{\mathfrak{t}}}(x,s,r_0,\eta)$ est tempéré, si et seulement si, il existe un $\mathfrak{sl}_2$-triplet $(x,y,z)$ vérifiant $[s,x]=2r_{0}x, \,\, [s,z]=0, \,\, [s,y]=-2r_{0}y$ et $s-r_{0}z$ est elliptique. Dans ce cas, $\overline{M_{\mathfrak{t}}}(x,s,r_0,\eta)=M_{\mathfrak{t}}(x,s,r_0,\eta)$\\
\item  \cite[1.22]{Lusztig:2002a} Le $\mathbb{H}_{\mu_{\mathfrak{t}}}$-module simple $\overline{M_{\mathfrak{t}}}(x,s,r_0,\eta)$ est de la série discrète, si et seulement si, $s$ et $x$ ne sont contenus dans aucune sous-algèbre de Levi propre.
\end{enumerate}
\end{theo}

Décrivons le système de racines, le groupe de Weyl et les paramètres associés à un triplet cuspidal dans le cas des groupes classiques. Ceci est complètement traité dans \cite[2.13]{Lusztig:1988}.

\begin{center}
\renewcommand{\arraystretch}{1.5}
\begin{tabular}{ | >{\centering\arraybackslash}m{.07\textwidth} | >{\centering\arraybackslash}m{.14\textwidth} | >{\centering\arraybackslash}m{.17\textwidth} | >{\centering\arraybackslash}m{.05\textwidth} | >{\centering\arraybackslash}m{.05\textwidth} | >{\centering\arraybackslash}m{.3\textwidth} | l }
\cline{1-6}
$H$ & $L$ & partition & $R$ & $R_{\red}$ & paramètres & \\
\hhline{======~}
\multirow{2}*{$\Sp_{N}(\C)$} & $(\C^{\times})^{\ell} \times \Sp_{N'}(\C)$ & $(1^\ell) \times (2,4,\ldots,2d)$ & $BC_{\ell}$ & $B_{\ell}$ & \begin{tikzpicture}

	\draw (0.1,0) -- (0.9,0);
	\draw (1.1,0) -- (1.9,0);
	\draw[dotted] (2.1,0) -- (2.9,0);
	\draw (3,0.1) -- (4,0.1);
	\draw (3,-0.1) -- (4,-0.1);
		
	\draw (0,0) circle(.1);
	\draw (1,0) circle(.1);
	\draw (2,0) circle(.1);
	\draw (3,0) circle(.1);
	\draw (4,0) circle(.1);
	
	 \node at (0,0.3) {$2$};
	 \node at (1,0.3) {$2$};
	 \node at (2,0.3) {$2$};
	 \node at (3,0.3) {$2$};
	 \node at (4,0.3) {$a+1$};	
\end{tikzpicture}  & $N' \geqslant 2$ \\  \cline{2-6}
  & $(\C^{\times})^{\frac{N}{2}}$ & $(1^{\frac{N}{2}})$ & $C_{\frac{N}{2}}$ & $C_{\frac{N}{2}}$ & \begin{tikzpicture}

	\draw (0.1,0) -- (0.9,0);
	\draw (1.1,0) -- (1.9,0);
	\draw[dotted] (2.1,0) -- (2.9,0);
	\draw (3,0.1) -- (4,0.1);
	\draw (3,-0.1) -- (4,-0.1);
		
	\draw (0,0) circle(.1);
	\draw (1,0) circle(.1);
	\draw (2,0) circle(.1);
	\draw (3,0) circle(.1);
	\draw (4,0) circle(.1);
	
	 \node at (0,0.3) {$2$};
	 \node at (1,0.3) {$2$};
	 \node at (2,0.3) {$2$};
	 \node at (3,0.3) {$2$};
	 \node at (4,0.3) {$2$};	
\end{tikzpicture}  & \\  \cline{1-6}
  & $(\C^{\times})^\ell \times \SO_{N'}(\C)$ & $(1^{\ell}) \times (1,3,\ldots,2d+1)$ & $B_\ell$ & $B_\ell$ & \begin{tikzpicture}

	\draw (0.1,0) -- (0.9,0);
	\draw (1.1,0) -- (1.9,0);
	\draw[dotted] (2.1,0) -- (2.9,0);
	\draw (3,0.1) -- (4,0.1);
	\draw (3,-0.1) -- (4,-0.1);
		
	\draw (0,0) circle(.1);
	\draw (1,0) circle(.1);
	\draw (2,0) circle(.1);
	\draw (3,0) circle(.1);
	\draw (4,0) circle(.1);
	
	 \node at (0,0.3) {$2$};
	 \node at (1,0.3) {$2$};
	 \node at (2,0.3) {$2$};
	 \node at (3,0.3) {$2$};
	 \node at (4,0.3) {$a+1$};	
\end{tikzpicture}  & $N' \geqslant 3$ \\ \cline{2-6}

\multirow{2}*{$\SO_{N}(\C)$} & $(\C^{\times})^{\frac{N-1}{2}}$  & $(1^{\frac{N-1}{2}})$ & $B_{\frac{N-1}{2}}$ & $B_{\frac{N-1}{2}}$ & \begin{tikzpicture}

	\draw (0.1,0) -- (0.9,0);
	\draw (1.1,0) -- (1.9,0);
	\draw[dotted] (2.1,0) -- (2.9,0);
	\draw (3,0.1) -- (4,0.1);
	\draw (3,-0.1) -- (4,-0.1);
		
	\draw (0,0) circle(.1);
	\draw (1,0) circle(.1);
	\draw (2,0) circle(.1);
	\draw (3,0) circle(.1);
	\draw (4,0) circle(.1);
	
	 \node at (0,0.3) {$2$};
	 \node at (1,0.3) {$2$};
	 \node at (2,0.3) {$2$};
	 \node at (3,0.3) {$2$};
	 \node at (4,0.3) {$2$};	
\end{tikzpicture} & $N$ impair \\ \cline{2-6}

 & $(\C^{\times})^{\frac{N}{2}}$  & $(1^{\frac{N}{2}})$ & $D_{\frac{N}{2}}$ & $D_{\frac{N}{2}}$ & \begin{tikzpicture}

	\draw (0.1,0) -- (0.9,0);
	\draw (1.1,0) -- (1.9,0);
	\draw[dotted] (2.1,0) -- (2.9,0);
	\draw (3.0866,-0.05) -- (3.9134,-0.45);
	\draw (3.0866,0.05) -- (3.9134,0.45);
		
	\draw (0,0) circle(.1);
	\draw (1,0) circle(.1);
	\draw (2,0) circle(.1);
	\draw (3,0) circle(.1);
	\draw (4,0.5) circle(.1);
	\draw (4,-0.5) circle(.1);
	
	 \node at (0,0.3) {$2$};
	 \node at (1,0.3) {$2$};
	 \node at (2,0.3) {$2$};
	 \node at (3,0.3) {$2$};
	 \node at (4.2,0.8) {$2$};	
 \node at (4.2,-0.8) {$2$};	
\end{tikzpicture} & $N$ pair \\ \cline{1-6}
\end{tabular}\captionof{table}{Système de racines et paramètres associé à un triplet cuspidal}\phantomsection\label{tableparametregal} 
\end{center}

\begin{rema}
Dans la table précédente, on a noté $a$ la plus grande part de la partition non triviale, c'est-à-dire $a=2d$ dans le cas $\Sp_{N}(\C)$ et $a=2d+1$ dans le cas $\SO_N(\C)$. 
\end{rema}

\subsection{Paramétrage des paramètres de Langlands enrichis}

Soient $G$ un groupe classique et $\cj=[\widehat{L},\varphi,\varepsilon] \in \mathcal{B}_{e}^{\st}(G)$ un triplet inertiel de $G$. On suppose que $\varphi$ est normalisé comme précédemment. Soient $\zeta \in \mathcal{X}(\widehat{L})$ (resp. $\chi$) un cocaractère non-ramifié unitaire (resp. hyperbolique) de $\widehat{L}$. \\

On s'intéresse à tout les paramètres de Langlands enrichis $(\phi,\eta)$ qui admettent $(\widehat{L},\varphi \zeta \chi,\varepsilon)$ pour support cuspidal. On note $\lambda=\lambda_{\varphi\zeta \chi}$ le caractère infinitésimal associé.
Soit $(\phi,\eta) \in \Phi_{e}(G)$ tel que $\cSc(\phi,\eta)=(\widehat{L},\varphi \zeta \chi,\varepsilon)$. On a vu qu'on avait alors $\lambda=\lambda_{\phi}=\lambda_{\varphi \zeta \chi}=\restriction{\varphi}{W_F} \zeta \chi \chi_{\varphi}^{-1}$. Puisque $\varphi$ est en particulier tempéré, la partie elliptique de $\lambda$ est $\restriction{\varphi}{W_F} \zeta$ et la partie hyperbolique est $\chi \chi_{\varphi}^{-1}$. Notons $\lambda(\Fr)=s_{\lambda}t_{\lambda}$ la décomposition de l'élément semi-simple $\lambda(\Fr)$ en produit de l'élément semi-simple hyperbolique (resp. elliptique) $s_{\lambda}$ (resp. $t_{\lambda}$). On rappelle qu'on a défini, en section \ref{groupesweildeligne}, le groupe $J_{\lambda}^{G}=Z_{\widehat{G}}(\restriction{\lambda}{I_F},t_{\lambda})$. Il se trouve qu'ici on a : $J_{\lambda}^{G}=H_{\varphi\zeta}^{G}$. Soit $\overline{s_{\lambda}} \in \mathfrak{j}_{\lambda}^{L}$ l'unique élément semi-simple hyperbolique de $\mathfrak{j}_{\lambda}^{L}$ tel que $\exp(\overline{s_{\lambda}})=s_{\lambda}$, on a : $$A_{\widehat{G}}(\phi)\simeq  A_{\widehat{G}}(\lambda,N_{\phi})=A_{J_{\lambda}^{G}}(s_{\lambda},N_{\phi}) \quad \text{et} \quad A_{\widehat{L}}(\varphi\zeta \chi) \simeq A_{\widehat{G}}(\lambda,N_{\varphi})=A_{J_{\lambda}^{L}}(s_{\lambda},N_{\varphi}).$$ Considérons le triplet cuspidal $\mathfrak{t}_{\zeta}=(J_{\lambda}^{L},\mathcal{C}_{N_{\varphi}}^{J_{\lambda}^{L}},\varepsilon)$ de $J_{\lambda}^{G}$.\\

On a alors $$\mathcal{W}_{\cjp,\zeta}=\mathcal{W}_{\cjp,\zeta}^{\circ} \rtimes \mathcal{R}_{\cjp,\zeta}, \quad \mathcal{W}_{\cjp,\zeta}=W_{H_{\varphi \zeta}^{L}}^{H_{\varphi\zeta}^{G}}=W_{J_{\lambda}^{L}}^{J_{\lambda}^{G}}  \quad \text{et} \quad \mathcal{W}_{\cjp,\zeta}^{\circ} = W_{{J_{\lambda}^{L}}^{\circ}}^{{J_{\lambda}^{G}}^{\circ}}.  $$

\begin{prop}\label{propparamparamlang}
Soit $G$ un groupe linéaire ou un groupe classique. Soient $\cj \in \mathcal{B}_{e}^{\st}(G)$ un triplet inertiel, $\omega=\mathcal{W}_{\cjp} \cdot (\varphi\zeta \chi) \in \mathcal{T}_{\cjp}/\mathcal{W}_{\cjp}$ un caractère infinitésimal avec $\zeta$ unitaire et $\chi$ hyperbolique. On considère d'une part, l'ensemble des paramètres de Langlands enrichis  $\Phi_e(G)_{\cjp,\omega}$ qui admettent $\omega$ pour caractère infinitésimal et d'autre part, l'algèbre de Hecke graduée étendue $\Irr(\mathbb{H}_{\cjp,\zeta,\mu_{\zeta}} \rtimes \C[\mathcal{R}_{\cjp,\zeta}])$ définie par le triplet cuspidal $(J_{\lambda}^{L},\mathcal{C}_{N_{\varphi}}^{J_{\lambda}^{L}},\varepsilon) \in \mathcal{S}_{J_{\lambda}^{G}}$ de $J_{\lambda}^{G}$. On a une bijection :
$$\Irr(\mathbb{H}_{\cjp,\zeta,\mu_{\zeta}} \rtimes \C[\mathcal{R}_{\cjp,\zeta}]) \leftrightarrow \Phi_e(G)_{\cjp,\omega} $$ qui induit des bijections entre $$\Irr(\mathbb{H}_{\cjp,\zeta,\mu_{\zeta}} \rtimes \C[\mathcal{R}_{\cjp,\zeta}])_{2} \leftrightarrow \Phi_e(G)_{\cjp,\omega,2} \quad \text{et} \quad \Irr(\mathbb{H}_{\cjp,\zeta,\mu_{\zeta}} \rtimes \C[\mathcal{R}_{\cjp,\zeta}])_{\bdd} \leftrightarrow \Phi_e(G)_{\cjp,\omega,\bdd}. $$
\end{prop}

\begin{proof}
Comme précédemment, soit $\lambda=\lambda_{\varphi\zeta \chi}$ le caractère infinitésimal de $\varphi\zeta \chi$. Nous avons déjà expliqué que la donnée d'un paramètre de Langlands $\phi$ de caractère infinitésimal $\lambda$ est équivalent à la donnée de $\lambda$ et $N_{\phi}=dS_{\phi}\left( \begin{smallmatrix}0 & 1 \\  0 & 0 \end{smallmatrix} \right)$. De plus, on a : $A_{\widehat{G}}(\phi)=A_{J_{\lambda}^{G}}(s_{\lambda},N_{\phi})$. Par conséquent, l'ensemble des paramètres de Langlands enrichis $(\phi,\eta)$ de $G$ de support cuspidal $(\widehat{L},\varphi \zeta \chi,\varepsilon)$ est en bijection avec $\{(x,\eta) \mid x \in \mathfrak{j}_{\lambda}^{G}, [\overline{s_{\lambda}},x]=2r_{0}x, \eta \in \Irr(A_{J_{\lambda}^{G}}(s_{\lambda},x))_{\mathfrak{t}_{\zeta}}\}$ où $\mathfrak{t}_{\zeta}=[J_{\lambda}^{L},\mathcal{C}_{N_{\varphi}}^{J_{\lambda}^{L}},\varepsilon] \in \mathcal{S}_{J_{\lambda}^{G}}$. Puisque $J_{\lambda}^{G}$ est un produit de groupes linéaires, symplectiques ou de type orthogonal, en procédant comme dans la démonstration du théorème \ref{theoremesupportcuspidal} et en utilisant les théorèmes \ref{thm:paramfibremodsimple1} et \ref{thm:paramfibremodsimple2}, on obtient une bijection : $$\Irr(\mathbb{H}_{\cjp,\zeta,\mu_{\zeta}} \rtimes \C[\mathcal{R}_{\cjp,\zeta}]) \leftrightarrow \Phi_e(G)_{\cjp,\omega}.$$

Reste à voir que cette bijection envoie un paramètre de Langlands discret sur un module simple de la série discrète et un paramètre tempéré sur un module simple tempéré. Le $\mathbb{H}_{\cjp,\zeta,\mu_{\zeta}} \rtimes \C[\mathcal{R}_{\cjp,\zeta}]$-module simple de caractère central $(\overline{s_{\lambda}},r_{0})$ associé $(N_{\phi},\eta)$ est de la série discrète, si et seulement si, $\overline{s_{\lambda}}$ et $N_{\phi}$ ne sont contenus dans aucune sous-algèbre de Levi propre de $\mathfrak{j}_{\lambda}^{G}$. Or, si l'image du paramètre $\phi$, correspondant à $(\lambda,N_{\phi})$, se factorise à un sous-groupe de Levi $\widehat{M} \supseteq \widehat{L} $ de $\widehat{G}$, alors $\mathfrak{j}_{\lambda}^{M}$ est une sous-algèbre de Levi de propre de $\mathfrak{j}_{\lambda}^{G}$ qui contient $\overline{s_{\lambda}}$ et $N_{\phi}$. Par conséquent, le paramètre de Langlands $\phi$ est discret.\\

D'après le théorème de Lusztig \ref{thmlusztigalghecke}, le $\mathbb{H}_{\cjp,\zeta,\mu_{\zeta}} \rtimes \C[\mathcal{R}_{\cjp,\zeta}]$-module simple associé à $(\overline{s_{\lambda}},N_{\phi})$ est tempéré, si et seulement si, il existe un $\mathfrak{sl}_{2}$-triplet $(x,y,z)$ avec $x=N_{\phi}$ tel que $[\overline{s_{\lambda}},x]=2r_{0}x, \,\, [\overline{s_{\lambda}},z]=0, \,\, [\overline{s_{\lambda}},y]=-2r_{0}y$ et $\overline{s_{\lambda}}-r_0 z$ elliptique. Puisque $\overline{s_{\lambda}}-r_0 z$ est un élément semi-simple hyperbolique, ceci est équivalent à $\overline{s_{\lambda}}=r_0 z$. En d'autres termes, ceci signifie que $\chi \chi_{\varphi}^{-1}=\chi_{\phi}^{-1}$, c'est-à-dire $\chi=\chi_{\varphi}/\chi_{\phi}^{-1}=\chi_{c}^{-1}$. Puisque $\restriction{\phi}{W_F}=\restriction{\varphi}{W_F} \zeta \chi \chi_{c}=\restriction{\varphi}{W_F} \zeta$, ceci est bien équivalent à ce que le paramètre $\phi$ est tempéré.
\end{proof}

\subsection{Paramétrage du dual admissible}

Dans cette section, ayant fixé une paire inertielle $\mathfrak{s} \in \mathcal{B}(G)$ et un $L$-triplet inertiel correspondant $\cj \in \mathcal{B}_{e}^{\st}(G)$, on établit une bijection entre les paramètres de Langlands enrichis dans $\Phi_{e}(G)_{\cjp}$ et les représentations irréductibles dans le bloc de Bernstein $\Irr(G)_{\mathfrak{s}}$. Pour cela, nous allons utiliser les résultats de Lusztig concernant la classification de certains « paramètres » avec des modules simples d'une algèbre de Hecke graduée.

\begin{theo}\phantomsection\label{theorembijection}
Soit $G$ un groupe classique déployé. Soient $\mathfrak{s}=[L,\sigma] \in \mathcal{B}(G)$ et $\cj=[\widehat{L},\varphi,\varepsilon] \in \mathcal{B}_{e}^{\st}(G)$ correspondant. On a une bijection $$\Irr(G)_{\mathfrak{s}} \leftrightarrow \Phi_{e}(G)_{\cjp},$$ induisant des bijections $$\Irr(G)_{\mathfrak{s},2} \leftrightarrow\Phi_{e}(G)_{\cjp,2},$$ et $$\Irr(G)_{\mathfrak{s},\temp} \leftrightarrow \Phi_{e}(G)_{\cjp,\bdd}.$$ 
\end{theo}

\begin{proof}
On suppose avoir normalisé $L$, $\sigma$ et $\varphi$ comme dans les sections précédentes. Lorsque tous les nombres de torsions $t_i$ valent $1$, on peut identifier $\mathcal{T}_{\cjp}$ à un tore maximal de $Z_{\widehat{G}}(\varphi)^{\circ}$. Sinon, on peut l'identifier à l'image par $z \mapsto z^{t_i}$ sur chaque composante indexée par $i$. Ainsi, on peut définir la donnée radicielle basée $\Psi_{\cjp}$ associée à ce groupe et ce tore maximal. D'après la démonstration du théorème \ref{verifcompatibilite}, par dualité on peut identifier $\Psi_{\mathfrak{s}}$ et $\Psi_{\cjp}$.\\

Soient $\zeta \in \mathcal{X}(L)$ un caractère non-ramifié unitaire et $\zeta \in \mathcal{X}(\widehat{L})$ le cocaractère non-ramifié correspondant. Soient $\chi \in \mathcal{X}(L)$ un caractère non-ramifié hyperbolique, $\omega=W_{\mathfrak{s}} \cdot (\sigma \zeta \chi)$ un caractère infinitésimal dont la partie elliptique est dans l'orbite de $W_{\mathfrak{s}} \cdot (\sigma \zeta)$. Soit $\widehat{\omega}=\mathcal{W}_{\cjp} \cdot (\varphi \zeta \chi)$ le caractère infinitésimal correspondant (il faudrait ajouter $\varepsilon$ pour être plus précis). Notons $\overline{\omega}=Z_{\mathcal{W}_{\cjp}} \cdot (s)$ le caractère infinitésimal hyperbolique correspondant avec $s \in \mathfrak{t}_{\cjp,\R}$. On a vu que l'équivalence de catégorie du théorème \ref{thmHeiermann} et la réduction de Lusztig induisent des bijections :

$$\Irr(G)_{\omega} \leftrightarrow \Irr(\mathcal{H}_{\mathfrak{s}}')_{\omega} \leftrightarrow \Irr(\mathbb{H}_{\mathfrak{s},\zeta,\mu_{\zeta}}')_{\overline{\omega}}.$$

Puisque nous avons montré que $\mathcal{T}_{\cjp}$ et $T_{\mathfrak{s}}$ sont isomorphes et que les actions respectives de $\mathcal{W}_{\cjp}$ et de $W_{\mathfrak{s}}$ sont compatibles. Ainsi, $\Psi_{\mathfrak{s},\zeta} \simeq \Psi_{\cjp,\zeta}$. En particulier, $W_{\mathfrak{s},\zeta} \simeq \mathcal{W}_{\cjp,\zeta}$ et $R_{\mathfrak{s},\zeta} \simeq \mathcal{R}_{\cjp,\zeta}$. Puisque $\mathfrak{t}_{\mathfrak{s}}^{*} \simeq \mathfrak{t}_{\cjp}^{*}$ et $\Sigma_{\mathfrak{s},\zeta} \simeq \Sigma_{\cjp,\zeta}$, pour montrer que l'algèbre de Hecke graduée $\mathbb{H}_{\cjp,\zeta,\mu_{\zeta}}$ définie par Lusztig coïncide avec $\mathbb{H}_{\mathfrak{s},\zeta,\mu_{\zeta}}$, il suffit de montrer que la fonction paramètre est la même. D'une part, la fonction paramètre de $\mathbb{H}_{\mathfrak{s},\zeta,\mu_{\zeta}}$ est décrite dans la table \ref{fonctionparametrerep} et d'autre part, la fonction paramètre de $\mathbb{H}_{\cjp,\zeta,\mu_{\zeta}}$ est décrite dans la table \ref{tableparametregal}. D'après le théorème \ref{thmmoeglinreduc} de M\oe glin sur les points de réductibilités d'induites de cuspidales, on a : $2x=a_{\sigma_i}+1$, où $a_{\sigma_i}=\max \{a \in \N \mid (\sigma ,a) \in \Jord(\sigma')\}$. Ainsi, lorsque $\sigma_{i} \in \Jord(\sigma')$, alors  $a_{\sigma_i}+1$ est exactement la valeur apparaissant dans la table \ref{tableparametregal}. Ceci montre donc que les deux algèbres de Hecke graduée $\mathbb{H}_{\mathfrak{s},\zeta,\mu_{\zeta}}$ et $\mathbb{H}_{\cjp,\zeta,\mu_{\zeta}}$ sont égales. \\

En combinant les précédentes bijections avec celle obtenue à la proposition \ref{propparamparamlang}, on obtient donc une bijection $$\Irr(G)_{\omega} \leftrightarrow \Phi_{e}(G)_{\omega}.$$
\end{proof}

Faisons quelques remarques à la suite de ce théorème avant d'annoncer quelques conséquences.

\begin{rema}
La bijection entre $\Irr(G)_{\omega}$ et $\Phi_e(G)_{\omega}$ ne suppose pas la correspondance de Langlands. Plus précisément, ce qu'on a utilisé est : \begin{enumerate}[label=(\arabic*)]
\item la bijection entre $T_{\mathfrak{s}}$ et $\mathcal{T}_{\cjp}$ qui est une conséquence de la correspondance de Langlands pour les caractères ;
\item la bijection entre $W_{\mathfrak{s}}$ et $\mathcal{W}_{\cjp}$ obtenue par un calcul explicite ;
\item la description des fonctions paramètres d'une part pour $\mathbb{H}_{\mathfrak{s}}$ en terme de points de réductibilités d'induites parabolique, d'autre part pour $\mathbb{H}_{\cjp}$ par les calculs explicites de Lusztig et le lien se faisant par un résultat de M\oe glin qui est très difficile à obtenir.
\end{enumerate}
Ceci montre en particulier, \textit{a posteriori}, que si les $L$-paquets des groupes classiques sont paramétrés par les représentations irréductibles de $S$-groupe, alors nécessairement les représentations supercuspidales sont paramétrées par les paramètres de Langlands enrichis cuspidaux.\\ 

En effet, pour commencer, dans le cas d'un tore $T$, on sait d'après la correspondance de Langlands pour les tores que les paramètres de Langlands enrichis cuspidaux sont les paramètres des représentations irréductibles supercuspidales.\\ Ensuite, fixons un groupe classique $G$ et supposons que pour tout groupe classique $G'$ de même type que $G$ mais de rang semi-simple strictement inférieur à celui de $G$, les paramètres de Langlands des représentations supercuspidales de $G'$ correspondent aux paramètres de Langlands cuspidaux enrichis de $G'$. Soit $\mathfrak{s}=[L,\sigma] \in \mathcal{B}(G)$ une paire inertielle de $G$, où $L$ est un sous-groupe de Levi propre de $G$. Puisque $L$ est un produit de groupes linéaires et d'un groupe classique de même type que $G$ mais de rang semi-simple strictement inférieur, par hypothèse de récurrence, le paramètre de Langlands de $\sigma$ est un paramètre de Langlands enrichi cuspidal de $L$. Notons donc $\cj=[\widehat{L},\varphi,\varepsilon] \in \mathcal{B}_{e}^{\st}(G)$ le triplet inertiel correspondant. En admettant le point (3) ci-dessus, le théorème \ref{theorembijection} montre donc qu'on a une bijection entre $\Irr(G)_{\mathfrak{s}}$ et $\Phi_{e}(G)_{\cjp}$. Ainsi, pour tout paramètre de Langlands $\phi \in \Phi(G)$, la partie non cuspidale de $\Irr(\mathcal{S}_{\phi}^{G})$ paramètre une représentation irréductible qui n'est pas cuspidale. Par conséquent, les représentations supercuspidales de $G$ ne peuvent être paramétrés que par des paramètres de Langlands enrichis cuspidaux.
\end{rema}

En conséquence du théorème précédent on obtient :

\begin{theo}\phantomsection \label{compatibiliteinductionpreuve}
La conjecture \ref{conjindpar} est vraie pour les groupes classiques. Autrement dit, la correspondance de Langlands est compatible avec l'induction parabolique.
\end{theo}

\begin{proof}
En effet, nous savons grâce aux travaux d'Harris-Taylor, Henniart, Scholze pour le groupe linéaire et Arthur et M\oe glin pour les groupes classiques que les paramètres de Langlands des supercuspidales sont ceux qu'on avait prédit par la conjecture. Ce qui précède décrit les paramètres de Langlands des sous-quotients d'une induite de supercuspidale. Or, ces paramètres sont compatibles avec la conjecture.
\end{proof}

\subsection{Changement de paramétrage dans un $L$-paquet, support cuspidal et généricité}

Soient $G$ (les $F$-points d') un groupe classique, $\phi \in \Phi(G)_{\bdd}$ un paramètre de Langlands tempéré. Il existe un modèle de Whittaker $\mathfrak{m}$ de $G$ et une représentation générique $\pi_{\mathfrak{m}}$ dans le $L$-paquet $\Pi_{\phi}(G)$. Par ailleurs, la théorie de l'endoscopie montre qu'il existe une unique bijection $$\iota : \Pi_{\phi}(G) \rightarrow \Irr(\mathcal{S}_{\phi}^{G}),$$ telle que $\pi_{\mathfrak{m}}$ s'envoie sur la représentation triviale de $\mathcal{S}_{\phi}^{G}$. Si $\pi_{\mathfrak{m}'} \in \Pi_{\phi}(G)$ est une représentation générique pour un modèle de Whittaker $\mathfrak{m}'$, alors on dispose d'une autre bijection $\iota' : \Pi_{\phi}(G) \rightarrow \Irr(\mathcal{S}_{\phi}^{G})$ telle que $\iota'(\pi_{\mathfrak{m}'})$ est la représentation triviale. Les deux paramétrages sont reliés par un certain caractère $\omega_{\mathfrak{m},\mathfrak{m}'} : \mathcal{S}_{\phi}^{G} \rightarrow \C^{\times}$, si bien que pour tout $\pi \in \Pi_{\phi}(G)$ : $$\iota'(\pi)=\iota(\pi) \otimes \omega_{\mathfrak{m},\mathfrak{m}'}.$$ Ce caractère est décrit dans le cas des groupes classiques dans \cite[\textsection 10]{Gan:2012} et nous rappelons sa description ci-dessous.\\

Notons $V_{\phi}$ l'espace sur lequel la représentation $\phi$ se réalise et décomposons $\Std_G \circ \phi : W_F' \rightarrow \widehat{G} \hookrightarrow \GL(V_{\phi}) $ : $$\Std_G \circ \phi = \bigoplus_{\pi \in I^{\O}} \bigoplus_{a \in \N}  \left( \pi \boxtimes S_a \right) \otimes M_{\pi,a}  \bigoplus_{\pi \in I^{\S}} \bigoplus_{a \in \N}\left( \pi \boxtimes S_a \right) \otimes M_{\pi,a} \oplus \bigoplus_{\pi \in I^{\GL}} \bigoplus_{a \in \N} \left(\left(\pi\oplus \pi^{\vee} \right) \boxtimes S_a \right) \otimes M_{\pi,a},$$ où $M_{\pi,a}$ désigne l'espace de multiplicité de la représentation irréductible $\pi \boxtimes S_a$. Soit $z \in Z_{\widehat{G}}(\phi)$ un élément semi-simple. Cet élément définit un isométrie de $V_{\phi}$ qui se restreint en une isométrie sur chacun des sous-espaces de multiplicité. Notons $V_{\phi}^{z,-1}=\{v \in V_{\phi} \mid z \cdot v =-v\}$ et $\phi^{z,-1}$ la sous-représentation de $\phi$  sur $V_{\phi}^{z,-1}$. On a : $$\phi^{z,-1}=\bigoplus_{\pi \in I^{\O}} \bigoplus_{a \in \N}  \left( \pi \boxtimes S_a \right) \otimes M_{\pi,a}^{z_{\pi,a},-1} \oplus \bigoplus_{\pi \in I^{\S}} \bigoplus_{a \in \N}\left( \pi \boxtimes S_a \right) \otimes M_{\pi,a}^{z_{\pi,a},-1} \oplus \bigoplus_{\pi \in I^{\GL}} \bigoplus_{a \in \N} \left(\left(\pi\oplus \pi^{\vee} \right) \boxtimes S_a \right) \otimes M_{\pi,a}^{z_{\pi,a},-1}.$$ 

Il s'ensuit que $\det(\phi^{z,-1})$ définit un caractère quadratique de $W_F'$  qui se factorise par $W_F^{\ab} \simeq F^{\times}$ et donc par $F^{\times}/F^{\times 2}$. Pour tout $f \in F^{\times}/F^{\times 2}$, notons $$\omega_{f}(z)=\det(\phi^{z,-1}(f)).$$ Comme il est expliqué dans \cite[\textsection 4]{Gan:2012}, $\omega_{f}$ définit un caractère quadratique qui se factorise par $A_{\widehat{G}}(\phi)$. Lorsque $G$ est un groupe symplectique ou spécial orthogonal pair (déployé), l'ensemble des orbites sous $T$ des caractères génériques est un espace homogène principal sous $F^{\times}/F^{\times 2}$. Par conséquent, si $\mathfrak{m}'=f \cdot \mathfrak{m}$ où $f \in F^{\times}/F^{\times 2}$, alors $\omega_{\mathfrak{m},\mathfrak{m}'}=\omega_{f}$.\\

\begin{prop}\phantomsection \label{restrictionomega}
Pour tout $f \in F^{\times}/F^{\times 2}$, la restriction du caractère $\omega_{f}$ à $A_{(H_{\phi}^{G})^{\circ}}(\restriction{\phi}{\SL_{2}{(\C)}})$ est triviale.
\end{prop}

\begin{proof}
Décomposons la restriction de $\phi$ à $W_F$ et pour tout $z \in H_{\phi}^{G}$, on a : $$\restriction{\phi}{W_F}=\bigoplus_{\pi \in I^{\O}} \pi \otimes M_{\pi} \bigoplus_{\pi \in I^{\S}} \pi \otimes M_{\pi} \bigoplus_{\pi \in I^{\GL}} \left(\pi \oplus \pi^{\vee} \right) \otimes M_{\pi}, $$ $$\restriction{\phi}{W_F}^{z,-1}=\bigoplus_{\pi \in I^{\O}} \pi \otimes M_{\pi}^{z_{\pi},-1} \bigoplus_{\pi \in I^{\S}} \pi \otimes M_{\pi}^{z_{\pi},-1} \bigoplus_{\pi \in I^{\GL}} \left(\pi \oplus \pi^{\vee} \right) \otimes M_{\pi}^{z_{\pi},-1}.$$

Pour tout $z \in H_{\phi}^{G}$ et $\pi \in I^{\O} \sqcup I^{\S}$, on a : $\det(z_{\pi})=(-1)^{\dim(M_{\pi}^{z_{\pi},-1})}$. En particulier, pour $z \in \left(H_{\phi}^{G}\right)^{\circ}$ et $\pi \in I^{\O} \sqcup I^{\S}$, $\dim(M_{\pi}^{z_{\pi},-1})$ est pair. Revenons à présent au calcul du déterminant de $\phi^{z,-1}$.\\ Pour tout $z \in Z_{\left(H_{\phi}^{G}\right)^{\circ}}(\restriction{\phi}{\SL_2(\C)})$ et $\pi \in I^{\O} \sqcup I^{\S}$, on a $M_{\pi}^{z_{\pi}}=\bigoplus_{a \in \N} S_a \boxtimes M_{\pi,a}^{z_{\pi,a},-1}$ et : \begin{align*}
 \det \left( \bigoplus_{a \in \N}  \left( \pi \boxtimes S_a \right) \otimes M_{\pi,a}^{z_{\pi,a},-1} \right)&=\prod_{a \in \N} \det(\pi)^{\dim(S_a) \dim(M_{\pi,a})} \det(S_a)^{\dim(\pi) \dim(M_{\pi,a}^{z_{\pi,a},-1})} \\
&= \det(\pi)^{\sum_{a \in \N} \dim(S_a) \dim(M_{\pi,a}^{z_{\pi,a},-1}) } \\
&= \det(\pi)^{\dim({M_{\pi}^{z_{\pi},-1}})}
\end{align*}
Pour tout $z \in Z_{\left(H_{\phi}^{G}\right)^{\circ}}(\restriction{\phi}{\SL_2(\C)})$ et pour tout $\pi \in I^{\O} \sqcup I^{\S}$,  $\dim({M_{\pi}^{z_{\pi},-1}})$ est pair, donc $\det(\pi)^{\dim({M_{\pi}^{z_{\pi},-1}})}=1$.\\ Ceci montre donc que pour tout $f \in F^{\times}/F^{\times 2}$, la restriction de $\omega_{f}$ à $A_{\left(H_{\phi}^{G}\right)^{\circ}}(\restriction{\phi}{\SL_2(\C)})$ est triviale.
\end{proof}

\begin{prop}
Soient $\pi \in \Irr(G)$ une représentation irréductible de $G$ et $(L,\sigma) \in \Omega(G)$ son support cuspidal. Soient $\iota$ et $\iota'$ deux bijections $\Pi_{\phi}(G) \rightarrow \Irr(\mathcal{S}_{\phi}^{G})$ correspondant au choix de modèles de Whittaker. Soient $\eta=\iota(\pi)$ et $\eta'=\iota'(\pi)$. Soient $(\widehat{L},\varphi,\varepsilon) \in \Omega_{e}^{\st}(G)$ (resp. $(\widehat{L}',\varphi',\varepsilon') \in \Omega_{e}^{\st}(G)$) le support cuspidal de $(\phi,\eta)$ (resp. $(\phi,\eta')$) tel qu'il a été défini dans le théorème \ref{theoremesupportcuspidal}. Alors (à $\widehat{G}$-conjugaison près), on a : $\widehat{L}=\widehat{L}'$ et $\varphi=\varphi'$.
\end{prop}

\begin{proof}
En reprenant les notations de la proposition, on a vu que $\eta' = \eta \otimes \omega_{\mathfrak{m},\mathfrak{m}'}$. Il résulte de la proposition \ref{restrictionomega} que $\restriction{\eta}{A_{{\left(H_{\phi}^{G} \right)}^{\circ}}(\restriction{\phi}{\SL_{2}(\C)})}=\restriction{\eta'}{A_{{\left(H_{\phi}^{G} \right)}^{\circ}}(\restriction{\phi}{\SL_{2}(\C)})}$. Or, les constructions du support cuspidal détaillées dans la preuve du théorème \ref{theoremesupportcuspidal} montrent que (la $\widehat{G}$-classe de conjugaison) de $\widehat{L}$ et $\varphi$ ne dépendent que $\restriction{\eta}{A_{{\left(H_{\phi}^{G} \right)}^{\circ}}(\restriction{\phi}{\SL_{2}(\C)})}$.
\end{proof}

\begin{prop}
Soit $\sigma$ une représentation irréductible supercuspidale orpheline. Alors $\sigma$ n'est pas générique. De façon équivalente, une représentation supercuspidale générique est nécessairement ordinaire.
\end{prop}

\begin{proof}
Supposons que $\sigma$ est générique et soit $(\varphi,\varepsilon)$ « le » paramètre de Langlands enrichi de $\sigma$. Puisque $\sigma$ est générique, alors $\varepsilon=\omega_{f}$ pour un certain $f \in F^{\times}/F^{\times 2}$. D'après la proposition \ref{restrictionomega}, la restriction de $\omega_{f}$ à $A_{(H_{\phi}^{G})^{\circ}}(u_{\varphi})$ est triviale. Par ailleurs, par hypothèse, $\varepsilon$ est cuspidale. D'après la classification des systèmes locaux cuspidaux de Lusztig (ou par la correspondance de Springer ordinaire) il s'ensuit que $u_{\varphi}=1$. Ainsi, $\sigma$ est ordinaire.
\end{proof}

\section{Comparaison des supports cuspidaux}

En supposant qu'on a la correspondance de Langlands pour les représentations supercuspidales, nous avons défini un paramétrage des représentations qui ne sont pas supercuspidales par les représentations irréductibles de $S$-groupe. \textit{A priori}, on ne sait pas que le paramétrage qu'on obtient et celui obtenu par les travaux d'Arthur coïncident. Dans cette section nous montrons que le support cuspidal qu'on a défini au théorème \ref{theoremesupportcuspidal} est le même que celui qui a été défini par M\oe glin et M\oe glin-Tadi{{\'c}} dans \cite{Moeglin:2002} et \cite{Moeglin:2002a} pour les séries discrètes. Pour cela, on calcule d'abord explicitement la correspondance de Springer généralisée pour les unipotents distingués des groupes symplectiques et spéciaux orthogonaux. Ensuite, on relie nos constructions à celles de \cite{Moeglin:2002} et \cite{Moeglin:2002a}.  Puis, en utilisant les résultats de M\oe glin ou de \cite{Xu:2015} qui montre que le paramétrage de M\oe glin coïncide avec celui obtenu par les travaux d'Arthur, on en déduit que le paramétrage qu'on propose est cohérent.

\subsection{Groupe symplectique}

\subsubsection{Description combinatoire de la correspondance de Springer généralisée}

On note $\N=\{0,1,\ldots\}$ l'ensemble des entiers naturels.\\ Soit $N \geqslant 2$ un entier pair. Notons $\widetilde{\Psi}_{N}$ l'ensemble des couples $(A,B)$ où $A$ (resp. $B$) est un sous-ensemble fini de $\N$ (resp. $\N^*$) vérifiants les conditions suivantes :
\begin{itemize}
\item pour tout $i \in \N$, $\{i,i+1\} \not \subset A$ et $\{i,i+1\} \not \subset B$ ;
\item $|A|+|B|$ est impair ;
\item $\sum_{a \in A} a + \sum_{b \in B} b = \frac{1}{2}N+\frac{1}{2}(|A|+|B|)(|A|+|B|-1)$.
\end{itemize}
Considérons la relation d'équivalence sur $\widetilde{\Psi}_N$ engendrée par $$(A,B) \sim (\{0\}\cup (A+2),\{1\}\cup (B+2)),$$ et notons $\Psi_N$ l'ensemble des classes d'équivalences. \\ Soient $(A,B), (A',B') \in \Psi_N$. On dit que $(A,B)$ et $(A',B')$ sont \emph{similaires} lorsque $$A\cup B=A' \cup B' \quad \text{et} \quad A\cap B=A' \cap B'.$$ Dans chaque classes de similitudes, il existe un unique élément $(A_{\varnothing},B_{\varnothing})$ dit \emph{distingué}, où $A_{\varnothing}=\{a_1,\ldots,a_{m+1}\}$ et $B_{\varnothing}=\{b_1,\ldots,b_m\}$ vérifiant $$a_1 \leqslant b_1 \leqslant a_2 \leqslant b_2 \leqslant \ldots \leqslant a_m \leqslant b_m \leqslant a_{m+1}.$$ Afin de décrire les éléments d'une classe de similitudes, nous fixons un élément distingué $(A_{\varnothing},B_{\varnothing}) \in \Psi_N$ et nous posons $C=A_{\varnothing} \Delta B_{\varnothing}$. On appelle intervalle de $C$ tout sous-ensemble de $C$ non vide $I$ de la forme $I=\llbracket i,j \rrbracket$ tel que $i-1 \not \in C, \,\, j+1 \not \in C$ et $i \neq 0$. Notons $\mathcal{I}$ l'ensemble des intervalles de $C$ et $H$ l'ensemble des éléments de $C$ qui ne sont dans aucun intervalle. Ainsi, $H=\varnothing$ ou $H=\llbracket 0, h \rrbracket$. Pour tout $\alpha \in \mathcal{P}(\mathcal{I})$, notons $$A_{\alpha}=\bigcup_{I \in \mathcal{I}-\alpha} (I \cap A_{\varnothing}) \cup \bigcup_{I \in \alpha} (I \cap B_{\varnothing}) \cup (H \cap A_{\varnothing}) \cup (A_{\varnothing} \cap B_{\varnothing}) \quad \text{et} \quad B_{\alpha}=\bigcup_{I \in \mathcal{I}-\alpha} (I \cap B_{\varnothing}) \cup \bigcup_{I \in \alpha} (I \cap A_{\varnothing}) \cup (H \cap B_{\varnothing}) \cup (A_{\varnothing} \cap B_{\varnothing}) .$$ Pour tout $\alpha \in \mathcal{P}(\mathcal{I})$, $(A_{\alpha},B_{\alpha})$ est similaire à $(A_{\varnothing},B_{\varnothing})$. Plus précisément, l'application $$\begin{array}[t]{ccc}
\mathcal{P}(\mathcal{I}) & \longrightarrow & \{ \text{classe de similitude de } (A_{\varnothing},B_{\varnothing})\} \\
\alpha  & \longmapsto & (A_{\alpha},B_{\alpha})
\end{array}$$ est une bijection. Par ailleurs, $\mathcal{P}(\mathcal{I})$ est un $\kF_2$-espace vectoriel pour la différence symétrique $\Delta$ avec pour base naturelle les intervalles formés d'un seul élément.\\

Soit $\mathbf{p}=(p_1,\ldots,p_{k})=(q_1^{r_{q_1}},\ldots,q_m^{r_{q_m}})$ une partition de $N$ avec $p_1 \leqslant \ldots \leqslant p_{k}$ et $u_{\mathbf{p}} \in \Sp_{N}(\C)$ un élément unipotent paramétré par $\mathbf{p}$. Rappelons qu'en notant $\Delta_{\mathbf{p}}=\{q \in \mathbf{p} \mid q \text{ pair et  }  r_{q} \geqslant 1 \}$, on a : $A_{\Sp_{N}(\C)}(u_{\mathbf{p}}) = \prod_{q \in \Delta_{\mathbf{p}}} \langle z_{q} \rangle \simeq (\Z/2\Z)^{|\Delta_{\mathbf{p}}|}$. Quitte à supposer $p_1$ nul, on peut supposer que $k=2m$ est pair. Pour tout $i \in \llbracket 1,k \rrbracket$, définissons $$p_i'=p_i+(i-1).$$ Ainsi $p_1'<p_2'<\ldots<p_{2m}'$. Dans cette suite, il y a exactement $m$ nombre pairs (resp. impairs) $2y_1 < 2 y_2 < \ldots < 2y_m$ (resp. $2y_1'+1<\ldots<2y_m'+1$). vérifiant $$0 \leqslant y_1+1 \leqslant y_1'+2\leqslant y_2+2 \leqslant y_2'+3\leqslant \ldots \leqslant y_m+m\leqslant y_m'+(m+1).$$ Définissons $$A_{\varnothing}=\{0,y_1'+2,y_2'+3,\ldots,y_m'+(m+1)\} \quad \text{et} \quad B_{\varnothing}=\{y_1+1,y_2+2,\ldots,y_m+m\}.$$ Alors $(A_{\varnothing},B_{\varnothing})$ est distingué. Par ailleurs, on a une bijection entre l'ensemble des intervalles $\mathcal{I}_{\mathbf{p}}$ de $C=A_{\varnothing} \Delta B_{\varnothing}$ et $\Delta_{\mathbf{p}}$ définie de la façon suivante. Ordonnons les éléments de $\mathcal{I}_{\mathbf{p}}$ de façon croissante $I_1,\ldots,I_f$ si bien que pour tout $i,j \in \llbracket 1,f \rrbracket, x \in I_i, y \in I_j$, $i\leqslant j  \Rightarrow x \leqslant y$. Soit $\Delta_{\mathbf{p}}=\{q_1,\ldots,q_{f'}\}$ une énumération croissante de $\Delta_{\mathbf{p}}$. On a alors $f=f'$ et on fait correspondre $I_r$ à $a_r$.  De plus, $|I_r|=r_{q_r}$.

\subsubsection{Calcul de la correspondance de Springer pour les unipotents distingués}\label{section}

Soit $u \in \Sp_{N}(\C)$ un élément unipotent distingué admettant pour partition $\mathbf{p}=(p_1,\ldots,p_k)$. Les entiers $p_i$ sont distincts deux à deux, ont pour multiplicités un et sont pairs. On a $\Delta_{\mathbf{p}}=\{p_{i}, i \in \llbracket 1, k \rrbracket\}$ et $A_{\Sp_{N}(\C)}(u) \simeq \prod_{i=1}^{k} \langle z_{p_{i}} \rangle$. Décrivons le $u$-symbole distingué associé à $\mathbf{p}$ et sa classe de similitude.\\
\begin{itemize}
\item Si $k=2m$ est pair, on a : $$p_1 \leqslant p_2 +1 \leqslant p_3 +2 \leqslant \ldots	\leqslant p_{2m-1}+2m-2\leqslant p_{2m}+2m-1.$$

$$2 \left( \frac{p_1}{2}\right) < 2 \left(\frac{p_3+2}{2}\right) < \ldots < 2 \left(\frac{p_{2m-1}+2m-2}{2}\right) \quad \text{et} \quad 2 \left(\frac{p_2}{2}\right)+1 < 2 \left(\frac{p_4+2}{2}\right)+1 < \ldots < 2 \left(\frac{p_{2m}+2m-2}{2}\right)+1$$

D'où $$A_{\varnothing}=\left\{ 0, \frac{p_2}{2}+2, \ldots, \frac{p_{2i}}{2}+2i	,\ldots,\frac{p_{2m}}{2} +2m \right\} \quad \text{et} \quad B_{\varnothing}=\left\{\frac{p_1}{2}+1,\ldots, \frac{p_{2i-1}}{2}+2i-1,\ldots,\frac{p_{2m-1}}{2}+2m-1\right\}$$

$$C=\left\{0,\frac{p_1}{2}+1,\frac{p_2}{2}+2, \ldots, \frac{p_i}{2}+i ,\ldots, \frac{p_{2m}}{2}+2m \right\}$$

Les intervalles de $C$ sont les singletons $I_{i}=\{\frac{p_i}{2}+i\}$ pour $i \in \llbracket 1,k \rrbracket$. De plus, $H=\{0\}$, $H \cap A_{\varnothing}=\{0\}$, $H \cap B_{\varnothing} = \varnothing$ et $A_{\varnothing} \cap B_{\varnothing} = \varnothing$.\\
\item Si $k=2m-1$ est impair, on a : $$0 \leqslant p_1+1 \leqslant p_2 +2 \leqslant p_3 +3 \leqslant \ldots	\leqslant p_{2m-1}+2m-2\leqslant p_{2m-1}+2m-1.$$

$$2 \frac{0}{2} < 2 \left( \frac{p_2+2}{2}\right) +1< 2 \left(\frac{p_4+4}{2}\right) < \ldots < 2 \left(\frac{p_{2m-2}+2m-2}{2}\right) \quad \text{et} \quad 2 \left(\frac{p_1}{2}\right)+1 < 2 \left(\frac{p_3+2}{2}\right)+1 < \ldots < 2 \left(\frac{p_{2m-1}+2m-2}{2}\right)+1$$

D'où $$A_{\varnothing}=\left\{ 0, \frac{p_1}{2}+2, \ldots, \frac{p_{2i-1}}{2}+2i	,\ldots,\frac{p_{2m-1}}{2} +2m \right\} \quad \text{et} \quad B_{\varnothing}=\left\{1,\frac{p_2}{2}+3,\ldots, \frac{p_{2i-2}}{2}+2i-1,\ldots,\frac{p_{2m-2}}{2}+2m-1\right\}$$
$$C=\left\{0,1,\frac{p_1}{2}+2,\frac{p_2}{2}+3, \ldots, \frac{p_i}{2}+i+1 ,\ldots, \frac{p_{2m-1}}{2}+2m \right\}$$

Les intervalles de $C$ sont les singletons $I_{i}=\{\frac{p_i}{2}+i+1\}$ pour $i \in \llbracket 1,k \rrbracket$. De plus, $H=\{0,1\}$, $H \cap A_{\varnothing}=\{0\}$, $H \cap B_{\varnothing} = \{1\}$ et $A_{\varnothing} \cap B_{\varnothing} = \varnothing$.
\end{itemize}

Soit $\eta \in \Irr(A_{\Sp_{N}(\C)}(u))$. Le caractère $\eta$ détermine une partie de $\mathcal{I}_{\mathbf{p}}$ définie par $\alpha_{\eta}=\{I_{p_{i}} \mid i \in \llbracket 1,r \rrbracket, \eta(z_{p_i}) \neq 1 \}$. 

\begin{prop}
Le $u$-symbole $(A_{\eta},B_{\eta})$ défini par $\eta$ est : \begin{itemize}
\item si $k$ est pair 
$$A_{\eta}=\{0\} \sqcup \bigsqcup_{\substack{i \in \llbracket 1,k \rrbracket \\ \eta(z_{p_i})=1 \\ i \,\, \text{pair} }} I_i  \sqcup \bigsqcup_{\substack{i \in \llbracket 1,k \rrbracket \\ \eta(z_{p_i})=-1 \\ i \,\, \text{impair} }} I_i    \quad \text{et} \quad B_{\eta}=\bigsqcup_{\substack{i \in \llbracket 1,k \rrbracket \\ \eta(z_{p_i})=1 \\ i \,\, \text{impair} }} I_i  \sqcup \bigsqcup_{\substack{i \in \llbracket 1,k \rrbracket \\ \eta(z_{p_i})=-1 \\ i \,\, \text{pair} }} I_i $$
\item si $k$ est impair 
$$A_{\eta}=\{0\} \sqcup \bigsqcup_{\substack{i \in \llbracket 1,k \rrbracket \\ \eta(z_{p_i})=1 \\ i \,\, \text{impair} }} I_i  \sqcup \bigsqcup_{\substack{i \in \llbracket 1,k \rrbracket \\ \eta(z_{p_i})=-1 \\ i \,\, \text{pair} }} I_i   \quad \text{et} \quad B_{\eta}=\{1\} \sqcup \bigsqcup_{\substack{i \in \llbracket 1,k \rrbracket \\ \eta(z_{p_i})=1 \\ i \,\, \text{pair} }} I_i  \sqcup \bigsqcup_{\substack{i \in \llbracket 1,k \rrbracket \\ \eta(z_{p_i})=-1 \\ i \,\, \text{impair} }} I_i  .$$
\end{itemize} 
\end{prop}

Soit $j \in \llbracket 1 , k-1 \rrbracket$ et supposons que $\eta(z_{p_{j}})=\eta(z_{p_{j+1}})$. Posons $N'=N-p_{j}-p_{j+1}$, $(p_1',\ldots,p_{k-2}')=(p_1,\ldots,p_{j-1},\widehat{p_{j}},\widehat{p_{j+1}},p_{j+1},\ldots,p_{k})$ (partition où l'on a retiré $p_j$ et $p_{j+1}$). Remarquons que l'indice correspondant à une part $p$ dans $(p_1',\ldots,p_{k-2}')$ a la même parité que que l'indice correspondant à la part $p$ dans $(p_1,\ldots,p_{k})$.

\begin{prop}\phantomsection\label{propdefautusymb}
Soit $u' \in \Sp_{N'}(\C)$ un élément unipotent distingué admettant pour partition $\mathbf{p}'=(p_1',\ldots,p_{k-2}')$. Le groupe $A_{\Sp_{N'}(\C)}(u')$ est un sous-groupe de $A_{\Sp_{N}(\C)}(u)$ et on considère $\eta'$ la restriction de $\eta$ à $A_{\Sp_{N'}(\C)}(u')$. Le $u$-symbole $(A_{\eta'},B_{\eta'})$ défini par $\eta'$ est :
\begin{itemize}
\item si ($j$ est pair et $\eta(z_{p_j})=1$) ou ($j$ est impair et $\eta(z_{p_j})=-1$) alors : $$A_{\eta'}=
\left\{ \begin{array}{ll}
A_{\eta}-I_{j}  & \text{si $k$ pair} \\
A_{\eta}-I_{j+1}  & \text{si $k$ impair} 
\end{array} \right.  \quad \text{et} \quad
B_{\eta'}=
\left\{ \begin{array}{ll}
B_{\eta}-I_{j+1}  & \text{si $k$ pair} \\
B_{\eta}-I_{j}  & \text{si $k$ impair} 
\end{array} \right.$$
\item si ($j$ est impair et $\eta(z_{p_j})=1$) ou ($j$ est pair et $\eta(z_{p_j})=-1$) alors : 
$$A_{\eta'}=
\left\{ \begin{array}{ll}
A_{\eta}-I_{j+1}  & \text{si $k$ pair} \\
 A_{\eta}-I_{j}  &\text{si $k$ impair} 
\end{array} \right.  \quad \text{et} \quad
B_{\eta'}=
\left\{ \begin{array}{ll}
B_{\eta}-I_{j}  & \text{si $k$ pair} \\
B_{\eta}-I_{j+1} & \text{si $k$ impair} 
\end{array} \right.$$
\end{itemize}
En particulier, les $u$-symboles $(A_{\eta},B_{\eta})$ et $(A_{\eta'},B_{\eta'})$ ont même \emph{défaut} $$d'_{\eta}=|A_{\eta}|-|B_{\eta}|=|A_{\eta'}|-|B_{\eta'}|=d'_{\eta'}.$$
\end{prop}

\begin{prop}\phantomsection\label{defautsum}
Le défaut du $u$-symbole $(A_{\eta},B_{\eta})$ est $$d_{\eta}'=\left\{ \begin{array}{ll}
\displaystyle 1+\sum_{i=1}^{k} (-1)^{i} \eta(z_{p_i})  & \text{si $k$ pair} \\
\displaystyle \sum_{i=1}^{k} (-1)^{i+1} \eta(z_{p_i}) & \text{si $k$ impair} 
\end{array} \right.
.$$
\end{prop}

\begin{proof}
Si $k$ est pair remarquons pour tout $i \in \llbracket 1, k \rrbracket$, $I_{i} \subset A_{\eta} \Leftrightarrow (-1)^{i} \eta(z_{p_i})=1$ et $I_{i} \subset B_{\eta} \Leftrightarrow (-1)^{i} \eta(z_{p_i})=-1$. Il s'ensuit 

$$|A_{\eta}|=1+\sum_{\substack{i \in \llbracket 1,k \rrbracket \\ I_{i} \subset A_{\eta}}} (-1)^{i}\eta(z_{p_i}), \quad  | B_{\eta} | =-\sum_{\substack{i \in \llbracket 1,k \rrbracket \\ I_{i} \subset B_{\eta}}} (-1)^{i}\eta(z_{p_i}) \quad \text{et} \quad d_{\eta}'=1+\sum_{i=1}^{k} (-1)^{i} \eta(z_{p_i}).$$
Si $k$ est pair remarquons pour tout $i \in \llbracket 1, k \rrbracket$, $I_{i} \subset A_{\eta} \Leftrightarrow (-1)^{i+1} \eta(z_{p_i})=1$ et $I_{i} \subset B_{\eta} \Leftrightarrow (-1)^{i+1} \eta(z_{p_i})=-1$. Il s'ensuit 

$$|A_{\eta}|=1+\sum_{\substack{i \in \llbracket 1,k \rrbracket \\ I_{i} \subset A_{\eta}}} (-1)^{i+1}\eta(z_{p_i}), \quad | B_{\eta} | =1-\sum_{\substack{i \in \llbracket 1,k \rrbracket \\  I_{i} \subset B_{\eta}}} (-1)^{i+1}\eta(z_{p_i}) \quad \text{et} \quad d_{\eta}'=\sum_{i=1}^{k} (-1)^{i+1} \eta(z_{p_i}).$$
\end{proof}

Pour tout entier impair $d' \in \Z$, notons $\mathbf{p}_{d'}^{\S}$ la partition de $d'(d'-1)$ définie par $$\mathbf{p}_{d'}^{\S}=\left\{ \begin{array}{ll}
(2d'-2,2d-4,\ldots,4,2) & \text{si $d' \geqslant 1 $} \\
(2|d'|,2|d'|-2,\ldots,4,2) & \text{si $d\leqslant -1 $ } 
\end{array} \right. .$$

Appelons procédé d'élimination le procédé décrit pour obtenir $(u',\eta')$ à partir de $(u,\eta)$. Puisque la longueur de la partition associée à $u'$ est strictement plus petite que celle associée à $u$, il est clair qu'en appliquant autant que faire se peut le procédé d'élimination à partir de $(u,\eta)$ on obtiendra une partition $\widetilde{\mathbf{p}}=(\widetilde{p}_1,\ldots,\widetilde{p}_{\widetilde{k}})$ de $\widetilde{N}=\widetilde{p}_1+\ldots+\widetilde{p}_{\widetilde{k}}$ et un couple $(\widetilde{u},\widetilde{\eta})$ tel que pour tout $i \in \llbracket 1, \widetilde{k}-1 \rrbracket, \,\, \widetilde{\eta}(z_{\widetilde{p}_i}) \neq \widetilde{\eta}(z_{\widetilde{p}_{i+1}})$. Le couple $(\widetilde{u},\widetilde{\eta})$ ne dépend pas du choix des parts à chaque étape et on a vu dans la proposition \ref{propdefautusymb} que les $u$-symboles associés à $(u,\eta)$ et $(u',\eta')$ ont le même défaut. Par conséquent $(u,\eta)$ et $(\widetilde{u},\widetilde{\eta})$ ont même défaut $d_{\eta}'$. D'après \cite[12.4]{Lusztig:1984}, la correspondance de Springer généralisée associe à $(u,\eta)$ (resp. $(\widetilde{u},\widetilde{\eta})$) : \begin{itemize}
\item le sous-groupe de Levi $(\C^{\times})^{\frac{N-d_{\eta}'(d_{\eta}'-1)}{2}} \times \Sp_{d_{\eta}'(d_{\eta}'-1)}(\C)$ (resp. $(\C^{\times})^{\frac{\widetilde{N}-d_{\eta}'(d_{\eta}'-1)}{2}} \times \Sp_{d_{\eta}'(d_{\eta}'-1)}(\C)$) ; 
\item l'orbite unipotente correspondante à la partition $(1^{\frac{N-d_{\eta}'(d_{\eta}'-1)}{2}}) \times \mathbf{p}_{d_{\eta}'}^{\S}$ (resp. $(1^{\frac{\widetilde{N}-d_{\eta}'(d_{\eta}'-1)}{2}}) \times \mathbf{p}_{d_{\eta}'}^{\S}$) ;
\item la représentation irréductible cuspidale $\varepsilon^{\S}_{d_{\eta}'(d_{\eta}'-1)}$ (resp. $\varepsilon^{\S}_{d_{\eta}'(d_{\eta}'-1)}$).
\end{itemize} 

\begin{prop}\phantomsection\label{denfonctiondedprime}
Soit $d_{\eta} \in \N$ l'entier tel que $2d_{\eta}$ soit la plus grande part de $\mathbb{p}_{d_{\eta}'}^{\S}$. On a : 
$$d_{\eta}=\left\{ \begin{array}{ll}
\widetilde{k}-1 & \text{si $\widetilde{\eta}(z_{\widetilde{p}_1})=1$} \\
\widetilde{k} & \text{si $\widetilde{\eta}(z_{\widetilde{p}_1})=-1$} 
\end{array} \right.$$
\end{prop}

\begin{proof}
D'après la proposition\ref{defautsum}, on voit que : 
$$\text{si $k$ est pair :} \,\, d_{\eta}'=\left\{ \begin{array}{ll}
1-\widetilde{k} & \text{si $\widetilde{\eta}(z_{\widetilde{p}_1})=1$} \\
1+\widetilde{k} & \text{si $\widetilde{\eta}(z_{\widetilde{p}_1})=-1$} 
\end{array} \right. ,  \quad 
\text{si $k$ est impair :}\,\; d_{\eta}'=\left\{ \begin{array}{ll}
\widetilde{k} & \text{si $\widetilde{\eta}(z_{\widetilde{p}_1})=1$} \\
-\widetilde{k} & \text{si $\widetilde{\eta}(z_{\widetilde{p}_1})=-1$} 
\end{array} \right. .$$
Enfin, si $d_{\eta}' \geqslant 1$ alors $d_{\eta}=d_{\eta}'-1$ et si $d_{\eta}' \leqslant 1$ alors $d_{\eta}=-d_{\eta}'$. 
\end{proof}

\subsection{Groupe orthogonal}

\subsubsection{Description combinatoire de la correspondance de Springer généralisée}

Soit $N \geqslant 3$ un entier. Notons $\widetilde{\Psi}_{N}'$ l'ensemble des ensembles $\{A,B\}$ où $A$ et $B$ sont des sous-ensembles finis de $\N$  vérifiants les conditions suivantes :
\begin{itemize}
\item pour tout $i \in \N$, $\{i,i+1\} \not \subset A$ et $\{i,i+1\} \not \subset B$ ;
\item $\sum_{a \in A} a + \sum_{b \in B} b = \frac{1}{2}N+\frac{1}{2}((|A|+|B|-1)^{2}-1)$.
\end{itemize}
Considérons la relation d'équivalence sur $\widetilde{\Psi}_N$ engendrée par $$\{A,B\} \sim \{\{0\}\cup (A+2),\{0\}\cup (B+2)\},$$ et notons $\Psi_N'$ l'ensemble des classes d'équivalences. \\ Un élément de $\Psi_N'$ de la forme $\{A,A\}$ sera dit dégénéré; les autres éléments seront dit non-dégénérés. Soient $\{A,B\}, \{A',B'\} \in \Psi_N'$. On dit que $\{A,B\}$ et $\{A',B'\}$ sont \emph{similaires} lorsque $$A\cup B=A' \cup B' \quad \text{et} \quad A\cap B=A' \cap B'.$$ Dans chaque classes de similitudes, il existe un unique élément $\{A_{\varnothing},B_{\varnothing}\}$ dit \emph{distingué}, où  $A_{\varnothing}=\{a_1,\ldots,a_{m'}\}$ et $B_{\varnothing}=\{b_1,\ldots,b_m\}$ avec si $N$ est impair $m'=m+1$ , si $N$ est pair $m'=m$, vérifiant $$a_1 \leqslant b_1 \leqslant a_2 \leqslant b_2 \leqslant \ldots \leqslant a_m \leqslant b_m \; \text{et} \; b_m \leqslant a_{m+1} \; \text{si $N$ est impair}.$$ Afin de décrire les éléments d'une classe de similitudes, nous fixons un élément distingué $\{A_{\varnothing},B_{\varnothing}\} \in \Psi_N'$ et nous posons $C=A_{\varnothing} \Delta B_{\varnothing}$. On appelle intervalle de $C$ tout sous-ensemble de $C$ non vide $I$ de la forme $I=\llbracket i,j \rrbracket$ tel que $i-1 \not \in C, \,\, j+1 \not \in C$. Notons $\mathcal{I}$ l'ensemble des intervalles de $C$ et $H$ l'ensemble des éléments de $C$ qui ne sont dans aucun intervalle. Ainsi, $H=\varnothing$. Pour tout $\alpha \in \mathcal{P}(\mathcal{I})$, notons $$A_{\alpha}=\bigcup_{I \in \mathcal{I}-\alpha} (I \cap A_{\varnothing}) \cup \bigcup_{I \in \alpha} (I \cap B_{\varnothing}) \cup (H \cap A_{\varnothing}) \cup (A_{\varnothing} \cap B_{\varnothing}) \quad \text{et} \quad B_{\alpha}=\bigcup_{I \in \mathcal{I}-\alpha} (I \cap B_{\varnothing}) \cup \bigcup_{I \in \alpha} (I \cap A_{\varnothing}) \cup (H \cap A_{\varnothing}) \cup (A_{\varnothing} \cap B_{\varnothing}) .$$ Pour tout $\alpha \in \mathcal{P}(\mathcal{I})$, $\{A_{\alpha},B_{\alpha}\}$ est similaire à $\{A_{\varnothing},B_{\varnothing}\}$. Plus précisément, l'application $$\begin{array}[t]{ccc}
\mathcal{P}(\mathcal{I}) & \longrightarrow & \{ \text{classe de similitude de } \{A_{\varnothing},B_{\varnothing} \} \} \\
\alpha  & \longmapsto & \{A_{\alpha},B_{\alpha} \}
\end{array}$$ est une bijection. Par ailleurs, $\mathcal{P}(\mathcal{I})$ est un $\kF_2$-espace vectoriel pour la différence symétrique $\Delta$ avec pour base naturelle les intervalles formés d'un seul élément.\\

Soit $\mathbf{p}=(p_1,\ldots,p_{k})=(q_1^{r_{q_1}},\ldots,q_m^{r_{q_m}})$ une partition de $N$ avec $p_1 \leqslant \ldots \leqslant p_{k}$. On dira que $\mathbf{p}$ est dégénérée si toutes ses parts sont paires et de multiplicités paires. Dans ce cas ... Si $\mathbf{p}$ n'est pas dégénérée, soit $u_{\mathbf{p}} \in \Sp_{N}(\C)$ un élément unipotent paramétré par $\mathbf{p}$. Rappelons qu'en notant $\Delta_{\mathbf{p}}=\{q \in \mathbf{p} \mid q \text{ impair et  }  r_{q} \geqslant 1 \}$, on a : $A_{\SO_{N}(\C)}(u_{\mathbf{p}}) = \prod_{q \in \Delta_{\mathbf{p}}} \langle z_{q} \rangle / \langle \prod_{q \in \Delta_{\mathbf{p}}} z_q \rangle \simeq (\Z/2\Z)^{|\Delta_{\mathbf{p}}|-1}$. Quitte à supposer $p_1$ nul, on peut supposer que $k$ a même parité que $N$. Pour tout $i \in \llbracket 1,k \rrbracket$, définissons $$p_i'=p_i+(i-1).$$ Ainsi $p_1'<p_2'<\ldots<p_{k}'$. Dans cette suite, il y a exactement $\left[\frac{k}{2}\right]$ nombres pairs et $\left[\frac{k+1}{2}\right]$ nombres impairs $2y_1 < 2 y_2 < \ldots < 2y_{\left[\frac{k}{2}\right]}$ (resp. $2y_1'+1<\ldots<2y_{\left[\frac{k+1}{2}\right]}'+1$) vérifiant $$0 \leqslant y_1+1 \leqslant y_1'+2\leqslant y_2+2 \leqslant y_2'+3\leqslant \ldots \leqslant y_{\left[\frac{k}{2}\right]}'+\left[\frac{k}{2}\right]-1 \leqslant y_{\left[\frac{k}{2}\right]}+\left[\frac{k}{2}\right]-1 \; \text{et si $N$ est impair} \;  y_{\frac{1}{2}(k-1)}+\frac{1}{2}(k-1)-1 \leqslant y_{\frac{1}{2}(k+1)}'+\frac{1}{2}(k+1)-1 .$$ Définissons $$A_{\varnothing}=\{y_1',y_2'+1,\ldots,y_{\left[\frac{k+1}{2}\right]}'+\left[\frac{k+1}{2}\right]-1\} \quad \text{et} \quad B_{\varnothing}=\{y_1,y_2+1,\ldots,y_{\left[\frac{k}{2}\right]}+\left[\frac{k}{2}\right]-1\}.$$ Alors $(A_{\varnothing},B_{\varnothing})$ est distingué. Par ailleurs, on a une bijection entre $\Delta_{\mathbf{p}}$ et $\mathcal{I}_{\mathbf{p}}$ définie de la façon suivante. Ordonnons les éléments de $\mathcal{I}_{\mathbf{p}}$ de façon croissante $I_1,\ldots,I_f$ si bien que pour tout $i,j \in \llbracket 1,f \rrbracket, x \in I_i, y \in I_j$, $i\leqslant j  \Rightarrow x \leqslant y$. On a $\Delta_{\mathbf{p}}=\{q \in \mathbf{p} \mid q \text{ pair et  }  r_{q} \geqslant 1 \}$. Soit $\Delta_{\mathbf{p}}=\{q_1,\ldots,q_{f'}\}$ une énumération croissante de $\Delta_{\mathbf{p}}$. On a alors $f=f'$ et on correspondre $I_r$ à $a_r$.  De plus, $|I_r|=r_{q_r}$.

\subsubsection{Calcul de la correspondance de Springer pour les unipotents distingués}

Soit $u \in \SO_{N}(\C)$ un élément unipotent distingué admettant pour partition $\mathbf{p}=(p_1,\ldots,p_k)$. Les entiers $p_i$ sont distincts deux à deux, ont pour multiplicités un et sont impairs. On a $\Delta_{\mathbf{p}}=\{p_{i}, i \in \llbracket 1, k \rrbracket\}$,  $A_{\O_{N}(\C)}(u) \simeq \prod_{i=1}^{k} \langle z_{p_{i}} \rangle$ et $A_{\SO_{N}(\C)}(u) \simeq \prod_{i=1}^{k} \langle z_{p_{i}} \rangle / \langle z_{p_{1}} \ldots z_{p_{k}} \rangle$. La parité des $p_i$ implique que $k \equiv N \mod 2$. Décrivons le $u$-symbole distingué associé à $\mathbf{p}$ et sa classe de similitude.\\
On a : $$p_1 \leqslant p_2 +1 \leqslant p_3 +2 \leqslant \ldots	\leqslant p_{k-1}+k-2\leqslant p_{k}+k-1.$$

$$2 \left( \frac{p_2+1}{2}\right) < 2 \left(\frac{p_4+3}{2}\right) < \ldots < 2 \left(\frac{p_{[k/2]}+[k/2]-1}{2}\right) \quad \text{et} \quad 2 \left(\frac{p_{1}-1}{2}\right)+1 < 2 \left(\frac{p_{3}+1}{2}\right)+1 < \ldots < 2 \left(\frac{p_{\left[\frac{k+1}{2}\right]}+\left[\frac{k+1}{2}\right]-2}{2}\right)+1$$ D'où $$A_{\varnothing}=\left\{\frac{p_{1}-1}{2}, \ldots, \frac{p_{2i-1}-3}{2}+2i-1, \ldots, \frac{p_{2\left[\frac{k+1}{2}\right]-1}-3}{2} +2\left[\frac{k+1}{2}\right]-1 \right\} \quad \text{et} \quad B_{\varnothing}=\left\{ \frac{p_2+1}{2}, \ldots, \frac{p_{2i}-3}{2}+2i,\ldots, \frac{p_{2 \left[\frac{k}{2}\right]}-3}{2}+2\left[\frac{k}{2}\right] \right\}$$ $$C=\left\{\frac{p_1-1}{2}+1,\frac{p_2+1}{2}+2, \ldots, \frac{p_i-3}{2}+i-1 ,\ldots, \frac{p_{k}-3}{2}+k-1\right\}.$$ Les intervalles de $C$ sont les singletons $I_{i}=\{\frac{p_i-3}{2}+i-1 \}$  pour tout $i \in \llbracket 1,k \rrbracket$. De plus, $H=\varnothing$, $H \cap A_{\varnothing}=\varnothing$, $H \cap B_{\varnothing} = \varnothing$ et $A_{\varnothing} \cap B_{\varnothing} = \varnothing$.

Soit $\eta \in \Irr(A_{\SO_{N}(\C)}(u))$. On voit le caractère $\eta$ comme un caractère de $A_{\O_{N}(\C)}(u)$ trivial sur $\langle z_{p_{1}} \ldots z_{p_{k}} \rangle$. Le caractère $\eta$ détermine une partie de $\mathcal{I}_{\mathbf{p}}$ définie par $\alpha_{\eta}=\{I_{p_{i}} \mid i \in \llbracket 1,r \rrbracket, \eta(z_{p_i}) \neq 1 \}$. 

\begin{prop}
Le $u$-symbole $(A_{\eta},B_{\eta})$ défini par $\eta$ est :
$$A_{\eta}=\bigsqcup_{\substack{i \in \llbracket 1,k \rrbracket \\ \eta(z_{p_i})=1 \\ i \,\, \text{impair} }} I_i \sqcup \bigsqcup_{\substack{i \in \llbracket 1,k \rrbracket \\ \eta(z_{p_i})=-1 \\ i \,\, \text{pair} }} I_i \quad \text{et} \quad B_{\eta}=\bigsqcup_{\substack{i \in \llbracket 1,k \rrbracket \\ \eta(z_{p_i})=1 \\ i \,\, \text{pair} }} I_i \sqcup \bigsqcup_{\substack{i \in \llbracket 1,k \rrbracket \\ \eta(z_{p_i})=-1 \\ i \,\, \text{impair} }} I_i $$
\end{prop}

Soit $j \in \llbracket 1 , k-1 \rrbracket$ et supposons que $\eta(z_{p_{j}})=\eta(z_{p_{j+1}})$. Posons $N'=N-p_{j}-p_{j+1}$, $(p_1',\ldots,p_{k-2}')=(p_1,\ldots,p_{j-1},\widehat{p_{j}},\widehat{p_{j+1}},p_{j+1},\ldots,p_{k})$ (partition où l'on a retiré $p_j$ et $p_{j+1}$). Remarquons que l'indice correspondant à une part $p$ dans $(p_1',\ldots,p_{k-2}')$ a la même parité que que l'indice correspondant à la part $p$ dans $(p_1,\ldots,p_{k})$.

\begin{prop}\phantomsection\label{propdefautusymb2}
Soit $u' \in \SO_{N'}(\C)$ un élément unipotent distingué admettant pour partition $\mathbf{p}'=(p_1',\ldots,p_{k-2}')$. Le groupe $A_{\O_{N'}(\C)}(u')$ est un sous-groupe de $A_{\O_{N}(\C)}(u)$ et on considère $\eta'$ la restriction de $\eta$ à $A_{\O_{N'}(\C)}(u')$. Le $u$-symbole $(A_{\eta'},B_{\eta'})$ défini par $\eta'$ est :
\begin{itemize}
\item si ($j$ est pair et $\eta(z_{p_j})=1$) ou ($j$ est impair et $\eta(z_{p_j})=-1$) alors : $$A_{\eta'}=A_{\eta}-I_{j+1}
  \quad \text{et} \quad
B_{\eta'}=B_{\eta}-I_{j}
$$

\item si ($j$ est impair et $\eta(z_{p_j})=1$) ou ($j$ est pair et $\eta(z_{p_j})=-1$) alors : 
$$A_{\eta'}=A_{\eta}-I_{j} \quad \text{et} \quad
B_{\eta'}=B_{\eta}-I_{j+1}
$$
\end{itemize}
On a : $|A_{\eta}|-|B_{\eta}|=|A_{\eta'}|-|B_{\eta'}|$. En particulier les $u$-symboles $\{ A_{\eta},B_{\eta} \}$ et $ \{ A_{\eta'},B_{\eta'} \}$ ont même défaut $$d_{\eta}'= | |A_{\eta}|-|B_{\eta}| |=| |A_{\eta'}|-|B_{\eta'}| |=d_{\eta'}'.$$
\end{prop}

\begin{prop}\phantomsection\label{defautsum2}
Le défaut du $u$-symbole $ \{ A_{\eta},B_{\eta} \}$ est $$d_{\eta}'=\Big| \sum_{i=1}^{k} (-1)^{i+1} \eta(z_{p_i}) \Big|
.$$
\end{prop}

\begin{proof}
Remarquons pour tout $i \in \llbracket 1, k \rrbracket$, $I_{i} \subset A_{\eta} \Leftrightarrow (-1)^{i+1} \eta(z_{p_i})=1$ et $I_{i} \subset B_{\eta} \Leftrightarrow (-1)^{i+1} \eta(z_{p_i})=-1$. Il s'ensuit 
$$|A_{\eta}|=\sum_{\substack{i \in \llbracket 1,k \rrbracket \\ I_{i} \subset A_{\eta}}} (-1)^{i+1}\eta(z_{p_i}), \quad  | B_{\eta} | =-\sum_{\substack{i \in \llbracket 1,k \rrbracket \\ I_{i} \subset B_{\eta}}} (-1)^{i+1}\eta(z_{p_i}) \quad \text{et} \quad d_{\eta}'=\Big| \sum_{i=1}^{k} (-1)^{i+1} \eta(z_{p_i}) \Big|.$$
\end{proof}

Pour tout entier naturel $d \in \N$, notons $\mathbf{p}_{d}^{\O}$ la partition de $d^2$ définie par $$\mathbf{p}_{d}^{\O}=(2d-1,2d-3,\ldots,3,1).$$

Appelons procédé d'élimination procédé décrit pour obtenir $(u',\eta')$ à partir de $(u,\eta)$. Puisque la longueur de la partition associée à $u'$ est strictement plus petite que celle associée à $u$, il est clair qu'en appliquant autant que faire se peut le procédé d'élimination à partir de $(u,\eta)$ on obtiendra une partition $\widetilde{\mathbf{p}}=(\widetilde{p}_1,\ldots,\widetilde{p}_{\widetilde{k}})$ de $\widetilde{N}=\widetilde{p}_1+\ldots+\widetilde{p}_{\widetilde{k}}$ et un couple $(\widetilde{u},\widetilde{\eta})$ tel que pour tout $i \in \llbracket 1, \widetilde{k}-1 \rrbracket, \,\, \widetilde{\eta}(z_{\widetilde{p}_i}) \neq \widetilde{\eta}(z_{\widetilde{p}_{i+1}})$. Le couple $(\widetilde{u},\widetilde{\eta})$ ne dépend pas du choix des parts à chaque étape et on a vu dans la proposition \ref{propdefautusymb2} que les $u$-symboles associés à $(u,\eta)$ et $(u',\eta')$ ont le même défaut. Par conséquent $(u,\eta)$ et $(\widetilde{u},\widetilde{\eta})$ ont même défaut $d$. D'après \cite[13.4]{Lusztig:1984}, la correspondance de Springer généralisée associe à $(u,\eta)$ (resp. $(\widetilde{u},\widetilde{\eta})$) : \begin{itemize}
\item le sous-groupe de Levi $(\C^{\times})^{\frac{N-d_{\eta}'^2}{2}} \times \SO_{d_{\eta}'^2}(\C)$ (resp. $(\C^{\times})^{\frac{\widetilde{N}-d_{\eta}'^2}{2}} \times \SO_{d_{\eta}'^2}(\C)$) ; 
\item l'orbite unipotente correspondante à la partition $(1^{\frac{N-d_{\eta}'^2}{2}}) \times \mathbf{p}_{d_{\eta}'}^{\O}$ (resp. $(1^{\frac{\widetilde{N}-d_{\eta}'^2}{2}}) \times \mathbf{p}_{d_{\eta}'}^{\O}$) ;
\item la représentation irréductible cuspidale $\varepsilon^{\O}_{d_{\eta}'^2}$ (resp. $\varepsilon^{\O}_{d_{\eta}'^2}$).
\end{itemize} Pour harmoniser les notations avec le cas symplectique, on notera aussi $d_{\eta}\in \N$ l'entier $d_{\eta}'$. D'après la proposition \ref{defautsum2}, on voit que : $d_{\eta}=d_{\eta}'=\widetilde{k}$.

\subsection{Supports cuspidaux}

Soient $\phi \in \Phi(G)_2$ un paramètre de Langlands discret d'un groupe classique $G$ et $\eta \in \Irr(A_{\widehat{G}}(\phi))$. Décomposons $$\Std_G \circ \phi = \bigoplus_{\pi \in I} \pi \boxtimes S_{\pi},$$ où $I$ est l'ensemble des représentations irréductibles de $W_F$ apparaissant dans $\restriction{\phi}{W_F}$ et $S_{\pi}$ la représentation de $\SL_2(\C)$ correspondante au facteur $\pi \in I$. En prenant le centralisateur de $\phi$ dans le groupe orthogonal correspondant si $G$ est un groupe spécial orthogonal, on a vu qu'on a une décomposition : $$A_{\widehat{G}_{*}}(\phi) = \prod_{\pi \in I} A_{\widehat{G}_{\pi}}(S_{\pi}) \quad \text{et} \quad \eta \simeq \boxtimes_{\pi \in I} \eta_{\pi}.$$ Notons $\Jord(\phi)$ le bloc de Jordan de $\phi$. Pour tout $\pi \in I$, notons $n_{\pi}$ la dimension de $\pi$, $m_{\pi}$ la dimension de $S_{\pi}$ et $d_{\pi} \in \N$ l'entier tel que la correspondance de Springer généralisée associe à $(u_{\pi},\eta_{\pi})$ l'unipotent admettant pour partition $(2,\ldots,2d_{\pi}-2,2d_{\pi})$ dans le cas symplectique et $(1,\ldots,2d_{\pi}-3,2d_{\pi}-1)$ dans le cas orthogonal.

Dans ce qui suit, nous noterons un multiensemble, c'est-à-dire un ensemble où on autorise des multiplicités, par $\{\!\!\{  \ldots \}\!\!\}$.\\ Soit $\pi \in I$, décomposons $S_{\pi}=\bigoplus_{i=1}^{k_{\pi}} S_{p_{\pi,i}}$ et notons $E_{\pi}$ la réunion des multiensembles formés des exposants de $|\cdot|$ dans $S_{\pi}\left( \begin{smallmatrix} |w|^{1/2} & 0 \\0 &  |w|^{-1/2} \end{smallmatrix} \right)$ : $$E_{\pi}=\bigcup_{i=1}^{k_{\pi}} \lBrace \frac{p_{\pi,i}-1}{2}-j,j \in \llbracket 0, p_{\pi,i}-1 \rrbracket \rBrace.$$

Supposons que $\widehat{G}_{\pi}=\Sp_{m_{\pi}}(\C)$. Rappelons que tous les $p_{\pi,i}$ sont pairs, ainsi pour tout $i \in \llbracket 1,k_{\pi} \rrbracket, \,\, 2i\leqslant p_{\pi,i}$. Il s'ensuit que l'ensemble des exposants de $|\cdot|$ dans $\left(\bigoplus_{i=1}^{d_{\pi}} S_{2i} \right)\left( \begin{smallmatrix} |w|^{1/2} & 0 \\0 &  |w|^{-1/2} \end{smallmatrix} \right)$ est un sous-ensemble de $E_{\pi}$. Notons ainsi $$E_{c,\pi}=\bigcup_{i=1}^{k_{\pi}} \lBrace \frac{p_{\pi,i}-1}{2}-j,j \in \llbracket 0, p_{\pi,i}-1 \rrbracket \rBrace - \bigcup_{i=1}^{d_{\pi}} \lBrace \frac{2i-1}{2}-j,j \in \llbracket 0, 2i-1 \rrbracket \rBrace.$$  Le multiensemble $E_{c,\pi}$ est le multiensemble des exposants qui du cocaractère correcteur $\chi_{c,\pi}$. Comme on peut le voir à sa définition, si $e \in E_{c,\pi}$ alors $e \neq 0$ et $-e \in E_{c,\pi}$ avec la même multiplicité que $e$. Notons $E_{c,\pi}'$ le sous-multiensemble de $E_{c,\pi}$ constitué des éléments positifs, si bien que $E_{c,\pi}=E_{c,\pi}' \sqcup -E_{c,\pi}'$.\\

Supposons que $\widehat{G}_{\pi}=\O_{m_{\pi}}(\C)$. Rappelons que tous les $p_{\pi,i}$ sont impairs, ainsi pour tout $i \in \llbracket 1,k_{\pi} \rrbracket, \,\, 2i-1\leqslant p_{\pi,i}$. Il s'ensuit que l'ensemble des exposants de $|\cdot|$ dans $\left(\bigoplus_{i=1}^{d_{\pi}} S_{2i-1} \right)\left( \begin{smallmatrix} |w|^{1/2} & 0 \\0 &  |w|^{-1/2} \end{smallmatrix} \right)$ est un sous-ensemble de $E_{\pi}$. Notons ainsi $$E_{c,\pi}=\bigcup_{i=1}^{k_{\pi}} \lBrace \frac{p_{\pi,i}-1}{2}-j,j \in \llbracket 0, p_{\pi,i}-1 \rrbracket \rBrace - \bigcup_{i=1}^{d_{\pi}} \lBrace \frac{2i-2}{2}-j,j \in \llbracket 0, 2i-2 \rrbracket \rBrace.$$  Le multiensemble $E_{c,\pi}$ est le multiensemble des exposants qui du cocaractère correcteur $\chi_{c,\pi}$. Comme on peut le voir à sa définition, si $e \in E_{c,\pi}$ alors $-e \in E_{c,\pi}$ avec la même multiplicité que $e$. De plus, $0 \in E_{c,\pi}$ avec multiplicité $m_{\pi}-d_{\pi}$. Puisque $m_{\pi} \equiv d_{\pi}^2 \equiv d_{\pi} \mod 2$, $m_{\pi}-d_{\pi}$ est divisible par $2$. Notons $E_{c,\pi}'$ le sous-multiensemble de $E_{c,\pi}$ constitué des éléments strictement positifs et $0$ compté avec multiplicité $\frac{m_{\pi}-d_{\pi}}{2}$, si bien que $E_{c,\pi}=E_{c,\pi}' \sqcup -E_{c,\pi}'$.

\begin{prop}\label{suppcuspidalexplicite}
Soit $\phi \in \Phi(G)_2$ un paramètre de Langlands discret d'un groupe classique $G$, soient $N$ le rang semi-simple de $\widehat{G}$ et $\eta \in \Irr(A_{\widehat{G}}(\phi))$.  Comme précédemment on décompose $\Std_G \circ \phi = \bigoplus_{\pi \in I} \pi \boxtimes S_{\pi}$, où $I$ est l'ensemble des représentations irréductibles de $W_F$ apparaissant dans $\restriction{\phi}{W_F}$ et écrivons $\eta \simeq \boxtimes_{\pi \in I} \eta_{\pi}$.\\
Pour tout $\pi \in I$, on pose : $$n_{\pi}=\dim \pi, \,\, m_{\pi}=\dim S_{\pi}, \,\, d_{\pi}'=\left\{ \begin{array}{ll}
 \sum_{i=1}^{k_{\pi}} (-1)^{i+k_{\pi}} \eta_{\pi}(z_{p_{\pi,i}}) +2k_{\pi}+2-2 \left[ \frac{k_\pi+1}{2}\right] & \text{si $\widehat{G}_{\pi}$ est symplectique} \\
 \sum_{i=1}^{k_{\pi}} (-1)^{i+1} \eta_{\pi}(z_{p_{\pi,i}}) & \text{si $\widehat{G}_{\pi}$ est orthogonal} \\
 \end{array} \right. ,$$ $$\ell_{\pi}=\left\{ \begin{array}{ll}
 \frac{m_{\pi}-d_{\pi}'(d_{\pi}'-1)}{2} & \text{si $\widehat{G}_{\pi}$ est symplectique} \\
 \frac{m_{\pi}-d_{\pi}'^2}{2} & \text{si $\widehat{G}_{\pi}$ est orthogonal} \\
 \end{array} \right. \quad \text{et} \quad N^{\sharp}=\sum_{\substack{\pi \in I \\ \widehat{G}_{\pi} \;\; \text{symp.}}} n_{\pi} d_{\pi}'(d_{\pi}'-1) + \sum_{\substack{\pi \in I \\ \widehat{G}_{\pi} \;\; \text{orth.}}} n_{\pi} d_{\pi}'^2.$$ Si $\widehat{G}_{\pi}$ est symplectique, soit $d_{\pi} \in \N$ l'unique entier naturel tel que $(d_{\pi}+1)d_{\pi}=d_{\pi}'(d_{\pi}'-1)$ et si $\widehat{G}_{\pi}$ est orthogonal, soit $d_{\pi}= |d_{\pi}'|$. Le support cuspidal de $(\phi,\eta)$ que nous avons défini au théorème \ref{theoremesupportcuspidal} est $(\widehat{L},\varphi,\varepsilon)$ avec : \begin{itemize}
\item $\displaystyle \widehat{L}=\prod_{\pi \in I} \GL_{n_{\pi}}(\C)^{\ell_{\pi}} \times \widehat{G}'$, avec $\widehat{G^{\sharp}}$ un groupe classique de rang $N^{\sharp}$ de même type que $\widehat{G}$ ;
\item $\displaystyle \varphi=\left(\bigoplus_{\pi \in I} \bigoplus_{e \in E_{c,\pi}'} |\cdot|^{e} \pi \oplus \left( |\cdot|^{e} \pi \right)^{\vee} \right)\oplus \left( \bigoplus_{\substack{\pi \in I \\ \widehat{G}_{\pi} \;\; \text{symp.}}} \bigoplus_{a=1}^{d_{\pi}} \pi \boxtimes S_{2a} \oplus \bigoplus_{\substack{\pi \in I \\ \widehat{G}_{\pi} \;\; \text{orth.}}} \bigoplus_{a=1}^{d_{\pi}} \pi \boxtimes S_{2a-1} \right)$;
\item $\varepsilon \simeq \boxtimes_{\substack{\pi \in I \\ \widehat{G}_{\pi} \;\; \text{symp.}}} \varepsilon_{d_{\pi}(d_{\pi}+1)}^{\S} \boxtimes \boxtimes_{\substack{\pi \in I \\ \widehat{G}_{\pi} \;\; \text{orth.}}} \varepsilon_{d_{\pi}^2}^{\O \pm}$.
\end{itemize}
\end{prop}

\begin{proof}
Ceci résulte directement de nos constructions, des propositions \ref{defautsum} et \ref{defautsum2} et de la discussion précédente.
\end{proof}

Dans ce qui suit, nous allons montrer que le support cuspidal que nous avons défini est celui construit par M\oe glin et M\oe glin-Tadi{{\'c}} dans \cite{Moeglin:2002a}. Commençons par rappeler les résultats de \cite{Moeglin:2002a} dont nous aurons besoin.\\

Tout d'abord, dans \cite{Moeglin:2002a} est défini le \emph{support cuspidal partiel} d'une représentation irréductible d'un groupe classique. Si $\pi $ est une représentation irréductible de $G$, le support cuspidal partiel de $\pi$ est une représentation irréductible supercuspidale $\sigma^{'}$ d'un groupe classique $G^{\sharp}$ de même type que $G$ telle que $\pi$ est sous-quotient de $i_P^G(\pi' \boxtimes \sigma')$ avec $\pi'$ une représentation irréductible d'un groupe linéaire $\GL_{m}(F)$. Du côté galoisien, si $(\phi,\eta) \in \Phi(G)_2$ est un paramètre de Langlands discret d'un groupe classique, son support cuspidal partiel $(\varphi^{\sharp},\varepsilon)$ est un paramètre de Langlands enrichi cuspidal d'un groupe classique $G^{\sharp}$ de même type que $G$. Rappelons que si $\phi \in \Phi(G)_{2}$, alors on note (\cite[p. 716]{Moeglin:2002a}) : $\Jord^{+}(\phi)=\Jord(\phi) \sqcup \{ (\pi,0) \mid \exists a \in 2\N, (\pi,a) \in \Jord(\phi) \}$ et on étend $\eta$ à $\Jord^{+}(\phi)$ en posant pour tout $\pi \in I$ telle que $(\pi,0) \in \Jord^{+}(\phi), \,\, \eta_{\pi}^{+}(z_{0})=1$. De plus, pour tout $\pi \in I$, on notera $k_{\pi}$ le cardinal de $\Jord(\phi)_{\pi}$. Le support cuspidal partiel de $(\phi,\eta)$ est alors défini récursivement de la façon suivante :

\begin{itemize}
\item si $(\phi,\eta)$ est alterné, c'est-à-dire si pour tout $\pi \in I$, pour tout $i \in \llbracket 1,k_{\pi}-1 \rrbracket $, $\eta_{\pi}(z_{p_{\pi,i}}) \neq \eta_{\pi}(z_{p_{\pi,i+1}})$, alors le support cuspidal partiel $(\varphi',\varepsilon)$ doit satisfaire l'existence d'une application injective $$\psi : \Jord(\phi) \rightarrow \Jord^{+}(\varphi'),$$ telle que l'image de $\Jord(\phi)$ contient $\Jord(\varphi')$. De plus, pour tout $\pi \in I$, il existe une application croisante $\psi_{\pi} : \Jord(\phi)_{\pi} \rightarrow \Jord(\varphi')_{\pi}$  telle que pour tout $a \in \Jord(\phi)_{\pi}$,  $$\psi(\pi,a)=(\pi,\psi_{\pi}(a)) \quad \text{et} \quad \eta_{\pi}(z_{a})=\varepsilon^{+}(z_{\psi_{\pi}(a)}).$$
\item si $(\phi,\eta)$ n'est pas alterné, alors il existe $\pi \in I$, $j \in \in \llbracket 1,k_{\pi}-1 \rrbracket$ tels que $\eta_{\pi}(z_{p_{\pi,j}}) = \eta_{\pi}(z_{p_{\pi,j+1}})$. On pose alors $$\phi'=\bigoplus_{(\pi,a) \in \Jord(\phi) \setminus \{(\pi,p_{\pi,j}),(\pi,p_{\pi,j+1})\}} \pi \boxtimes S_a,$$ et la restriction $\eta'$ de $\eta$ définie comme précédemment. Ainsi, $(\phi',\eta')$ est un paramètre de Langlands discret enrichi pour un groupe de même type que $G$ mais de rang plus petit. Le support cuspidal partiel de $(\phi,\eta)$ est alors celui défini par $(\phi',\eta')$.
\end{itemize}

Revenons au support cuspidal que nous avons défini, notons $$\varphi^{\sharp}=\bigoplus_{\substack{\pi \in I \\ \widehat{G}_{\pi} \;\; \text{symp.}}} \bigoplus_{a=1}^{d_{\pi}} \pi \boxtimes S_{2a} \oplus \bigoplus_{\substack{\pi \in I \\ \widehat{G}_{\pi} \;\; \text{orth.}}} \bigoplus_{a=1}^{d_{\pi}} \pi \boxtimes S_{2a-1},$$ et montrons que $(\varphi^{\sharp},\varepsilon)$ est le support cuspidal partiel que M\oe glin et Tadi{{\'c}} définissent.\\

Soit $\pi \in I$. Le procédé d'élimination décrit plus haut et dans \cite{Moeglin:2002a} associe au sous-bloc de Jordan $\Jord(\phi)_{\pi}=\{ p_{\pi,1}, \ldots,p_{\pi,k_{\pi}} \}$ et $\eta_{\pi}$, c'est-à-dire à la partition $(p_{\pi,1},\ldots,p_{\pi,k_{\pi}})$ et $\eta_{\pi}$, une sous-partition $(\widetilde{p}_{\pi,1},\ldots,\widetilde{p}_{\pi,\widetilde{k}_{\pi}})$ et un caractère $\widetilde{\eta}_{\pi}$. Pour tout $\pi \in I$, on note $\widetilde{S}_{\pi}=\bigoplus_{i=1}^{\widetilde{k}_{\pi}} S_{\widetilde{p}_{\pi,i}}$ et $\widetilde{\phi}=\bigoplus_{\pi \in I} \pi \boxtimes \widetilde{S}_{\pi}$. Le paramètre $\widetilde{\phi}$ est obtenu à partir de $\phi$ et $\eta$ par le procédé d'élimination et on note $\widetilde{\eta}$ le caractère obtenu à partir de $\eta$.\\

Soit $\pi \in I$. 

Supposons que $\widehat{G}_{\pi}$ est un groupe symplectique et $\widetilde{k}_{\pi}>0$. \begin{itemize}
\item 
D'après la proposition \ref{denfonctiondedprime}, si $\eta_{\pi}(z_{\widetilde{p}_{\pi,1}})=1$, alors $d_{\pi}=\widetilde{k}_{\pi}-1$ et on définit $\psi_{\pi} : \{ \widetilde{p}_{\pi,1}, \ldots, \widetilde{p}_{\pi,\widetilde{k}_{\pi}} \} \rightarrow \{ 0 , 2 , \ldots, 2(\widetilde{k}_{\pi}-1)-2,2(\widetilde{k}_{\pi}-1) \} $ par $\psi_{\pi}(\widetilde{p}_{\pi,i})=2(i-1)$. 
\item 
D'après la proposition \ref{denfonctiondedprime}, si $\eta_{\pi}(z_{\widetilde{p}_{\pi,1}})=-1$, alors $d_{\pi}=\widetilde{k}_{\pi}$ et on définit $\psi_{\pi} : \{ \widetilde{p}_{\pi,1}, \ldots, \widetilde{p}_{\pi,\widetilde{k}_{\pi}} \} \rightarrow \{ 0 , 2 , \ldots, 2(\widetilde{k}_{\pi}-1),2\widetilde{k}_{\pi} \}$ par $\psi_{\pi}(\widetilde{p}_{\pi,i})=2i$. 
\end{itemize}

Supposons que $\widehat{G}_{\pi}$ est un groupe orthogonal et $\widetilde{k}_{\pi}>0$. On a alors $d_{\pi}=\widetilde{k}_{\pi}$ et on définit $\psi_{\pi} : \{ \widetilde{p}_{\pi,1}, \ldots, \widetilde{p}_{\pi,\widetilde{k}_{\pi}} \} \rightarrow \{1 , \ldots, 2\widetilde{k}_{\pi}-3,2\widetilde{k}_{\pi}-1\} $ par $\psi_{\pi}(\widetilde{p}_{\pi,i})=2i-1$.

\begin{prop}\label{comparaisonsupportcuspidalmoeglin}
L'application $\psi_{\pi}$ définit une injection $\Jord(\widetilde{\phi})_{\pi}\rightarrow \Jord^{+}(\varphi^{\sharp})_{\pi}$ telle que pour tout $(\pi,a) \in \Jord(\widetilde{\phi})_{\pi}$, $\widetilde{\eta}_{\pi}(z_{a})=\varepsilon_{\pi}^{+}(z_{\psi_{\pi}(a)})$. Par conséquent, le paramètre de Langlands enrichi $(\varphi^{\sharp},\varepsilon)$ que nous avons construit est bien le support cuspidal partiel défini dans \cite{Moeglin:2002a}. De plus, le support cuspidal de $(\widetilde{\phi},\widetilde{\eta})$ est : \begin{itemize}
\item $\prod_{\pi \in I} \GL_{n_{\pi}}(\C)^{\sum_{i=1}^{\widetilde{k}_{\pi}}\psi_{\pi}(\widetilde{p}_{\pi,i})} \times \widehat{G^{\sharp}}$ ;
\item $\displaystyle \bigoplus_{\pi \in I} \bigoplus_{i=1}^{\widetilde{k}_{\pi}} \bigoplus_{f \in \llbracket 0, \frac{\widetilde{p}_{\pi,i}-\psi_{\pi}(\widetilde{p}_{\pi,i})}{2}-1 \rrbracket} |\cdot|^{\frac{\widetilde{p}_{\pi,i}-1}{2}-f} \pi \oplus (|\cdot|^{\frac{\widetilde{p}_{\pi,i}-1}{2}-f} \pi )^{\vee} \oplus \varphi^{\sharp}$ ;
\item $\varepsilon$
\end{itemize}
\end{prop}

\begin{proof}
La première partie de la proposition résulte de la discussion précédent la proposition. Concernant le support cuspidal de $(\widetilde{\phi},\widetilde{\eta})$, il s'agit d'écrire explicitement les ensembles $E_{c,\pi}$ à l'aide de l'application $\psi$.
\end{proof}

À présent, on peut montrer dans le cas d'un paramètre discret quelconque que le support cuspidal que nous avons défini correspond à celui défini dans \cite{Moeglin:2002a} et \cite{Moeglin:2002}. Pour être plus précis, on ne trouvera pas mention explicite de support cuspidal d'un paramètre de Langlands enrichi dans \cite{Moeglin:2002a} ou \cite{Moeglin:2002}, en revanche, il apparait dans les constructions contenues dans ces articles.

\begin{theo}\label{suppcuspicompatible}
Soit $\pi \in \Irr(G)_{2}$, $(M,\sigma) \in \Omega(G)$ son support cuspidal. Soient $(\phi,\eta) \in \Phi_{e}(G)_{2}$ le paramètre de Langlands enrichi de $\pi$ et $(\widehat{L},\varphi,\varepsilon) \in \Omega_{e}^{\st}(G)$ son support cuspidal. Alors, après conjugaison on peut supposer, $\widehat{M}=\widehat{L}$ et le paramètre de Langlands enrichi de $\sigma$ est $(\varphi,\varepsilon)$.
\end{theo}

\begin{proof}
En effet, le cas $(\widetilde{\phi},\widetilde{\eta})$ que nous avons traité précédemment coïncide avec \cite[7-10]{Moeglin:2002a}. Pour le cas général, on peut se ramener à ce cas par procédé d'élimination. Remarquons qu'on a :\begin{footnotesize}
\begin{align*}
\lBrace \frac{p_{\pi,j+1}-1}{2}-f, f \in \llbracket 0,p_{\pi,j+1}-1 \rrbracket \rBrace &= \lBrace \frac{p_{\pi,j+1}-1}{2}-f, f \in \llbracket 0,\frac{p_{\pi,j}+p_{\pi,j+1}}{2}-1 \rrbracket \rBrace \cup \lBrace \frac{p_{\pi,j+1}-1}{2}-f, f \in \llbracket \frac{p_{\pi,j}+p_{\pi,j+1}}{2},p_{\pi,j+1}-1 \rrbracket \rBrace,\\
\lBrace f-\frac{p_{\pi,j+1}-1}{2}, f \in \llbracket 0,\frac{p_{\pi,j}+p_{\pi,j+1}}{2}-1 \rrbracket \rBrace &= \lBrace \frac{p_{\pi,j}-1}{2}-f, f \in \llbracket 0,\frac{p_{\pi,j}+p_{\pi,j+1}}{2}-1 \rrbracket \rBrace,  \quad f \leftrightarrow \frac{p_{\pi,j}+p_{\pi,j+1}}{2}-1-f\\
&=\lBrace \frac{p_{\pi,j}-1}{2}-f, f \in \llbracket 0,p_{\pi,j}-1 \rrbracket \rBrace \cup \lBrace \frac{p_{\pi,j}-1}{2}-f, f \in \llbracket p_{\pi,j},\frac{p_{\pi,j}+p_{\pi,j+1}}{2}-1 \rrbracket \rBrace \\
&= \lBrace \frac{p_{\pi,j}-1}{2}-f, f \in \llbracket 0,p_{\pi,j}-1 \rrbracket \rBrace \cup \lBrace \frac{p_{\pi,j+1}-1}{2}-f, f \in \llbracket \frac{p_{\pi,j}+p_{\pi,j+1}}{2},p_{\pi,j+1}-1 \rrbracket \rBrace, \,f \leftrightarrow f-\frac{p_{\pi,j+1}-p_{\pi,j}}{2} 
\end{align*}
\end{footnotesize}
Par conséquent, on obtient : \begin{footnotesize}
$$
\lBrace \frac{p_{\pi,j+1}-1}{2}-f, f \in \llbracket 0,p_{\pi,j+1}-1 \rrbracket \rBrace \cup \lBrace \frac{p_{\pi,j}-1}{2}-f, f \in \llbracket 0,p_{\pi,j}-1 \rrbracket \rBrace = \lBrace \frac{p_{\pi,j+1}-1}{2}-f, f \in \llbracket 0,\frac{p_{\pi,j}+p_{\pi,j+1}}{2}-1 \rrbracket \rBrace \cup \lBrace f-\frac{p_{\pi,j+1}-1}{2}, f \in \llbracket 0,\frac{p_{\pi,j}+p_{\pi,j+1}}{2}-1 \rrbracket \rBrace.
$$
\end{footnotesize}

On fixe $(\widetilde{\phi},\widetilde{\eta})$, $\pi$ une représentation irréductible de $W_F$ qui apparait dans $\restriction{\widetilde{\phi}}{W_F}$. Soit $(\phi,\eta) \in \Phi_e(G)_{2}$ un paramètre de Langlands discret enrichi d'un groupe classique $G$ tel que $(\widetilde{\phi},\widetilde{\eta})$ soit obtenu par procédé d'élimination à partir de $(\phi,\eta)$. On raisonne par récurrence sur la longueur de la partition $(p_{\pi,1},\ldots,p_{\pi,k_{\pi}})$. Le cas initial résulte de la proposition \ref{comparaisonsupportcuspidalmoeglin}. Supposons la propriété acquise pour une certaine longueur de partition. Supposons qu'il existe $j \in \Jord(\phi)_{\pi}$ tel que $\eta_{\pi}(z_{p_{\pi,j}})=\eta_{\pi}(z_{p_{\pi,j+1}})$.\\ Soit $(\phi',\eta') \in \Phi_e(G')_2$ le paramètre de Langlands discret défini par $\Jord(\phi')=\Jord(\phi) \setminus \{(\pi,p_{\pi,j}),(\pi,p_{\pi,j+1})\}$.\\ On a vu que les $u$-symboles associés à $(u_{\phi,\pi},\eta_{\pi})$ et $(u_{\phi',\pi},\eta_{\pi}')$ (voir les deux sections précédentes) ont même défaut $d_{\pi}'$. Par conséquent, les supports cuspidaux de $(\widehat{L},\varphi,\varepsilon)$ et $(\widehat{L}',\varphi',\varepsilon')$ les supports cuspidaux respectifs de $(\phi,\eta)$ et $(\phi',\eta')$ sont reliés par :
$\widehat{L}=\GL_{n_{\pi}}(\C)^{\frac{p_{\pi,j}+p_{\pi,j+1}}{2}} \times \widehat{L}'$ et $$\Std_{L}\circ \varphi=\bigoplus_{f \in \llbracket 0,p_{\pi,j+1}-1 \rrbracket}  |\cdot|^{\frac{p_{\pi,j+1}-1}{2}-f} \pi \oplus  \bigoplus_{f \in \llbracket 0,p_{\pi,j}-1 \rrbracket}  |\cdot|^{\frac{p_{\pi,j}-1}{2}-f} \pi \oplus \Std_{L'}\circ\varphi'.$$

En utilisant l'égalité $
\lBrace \frac{p_{\pi,j+1}-1}{2}-f, f \in \llbracket 0,p_{\pi,j+1}-1 \rrbracket \rBrace \cup \lBrace \frac{p_{\pi,j}-1}{2}-f, f \in \llbracket 0,p_{\pi,j}-1 \rrbracket \rBrace = \lBrace \frac{p_{\pi,j+1}-1}{2}-f, f \in \llbracket 0,\frac{p_{\pi,j}+p_{\pi,j+1}}{2}-1 \rrbracket \rBrace \cup \lBrace f-\frac{p_{\pi,j+1}-1}{2}, f \in \llbracket 0,\frac{p_{\pi,j}+p_{\pi,j+1}}{2}-1 \rrbracket \rBrace
$, on peut réorganiser les termes de $\varphi$ et voir que cela coïncide également avec \cite[5.1.3]{Moeglin:2002}, c'est-à-dire : $$\Std_{L}\circ \varphi=\bigoplus_{f \in \llbracket 0,\frac{p_{\pi,j}+p_{\pi,j+1}}{2}-1 \rrbracket}  |\cdot|^{\frac{p_{\pi,j+1}-1}{2}-f} \pi \oplus  (|\cdot|^{\frac{p_{\pi,j+1}-1}{2}-f} \pi)^{\vee} \oplus \Std_{L'}\circ\varphi' .$$

Par conséquent, d'une part avec les travaux de M\oe glin et M\oe glin-Tadi{{\'c}}, on sait que les propriétés combinatoires donnent lieu à un support cuspidal pour $\pi$ et pour $(\phi,\eta)$, qui coïncide avec nos constructions. Et de plus, Xu a montré dans \cite{Xu:2015} que ces propriétés coïncident avec la correspondance de Langlands pour les groupes classiques.

\end{proof}

\newpage

\appendix

\section{Correspondance de Springer généralisée pour le groupe orthogonal}\phantomsection\label{sectionspringorth}

Nous avons vu précédemment que la correspondance de Springer généralisée pour un groupe réductif connexe $H$ établit une bijection (à $H$-conjugaison près) $$ \Sigma_{H} : (\mathcal{C}_{u}^{H},\eta) \longmapsto (L,\mathcal{C}_{v}^{L},\varepsilon,\rho),$$ avec \begin{itemize}
\item une $H$-orbite $\mathcal{C}_{u}^{H}$ d'un élément unipotent $u \in H$ ;
\item une représentation irréductible $\eta$ de $A_{H}(u)$ ;
\item un sous-groupe de Levi $L$ de $H$ ;
\item une $L$-orbite $\mathcal{C}_{v}^{L}$ d'un élément unipotent $v \in L$ ;
\item une représentation irréductible \emph{cuspidale} $\varepsilon$ de $A_{L}(v)$ ;
\item une représentation irréductible $\rho$ de $N_{H}(L)/L$. 
\end{itemize} Nous souhaitons étendre cette bijection au groupe orthogonal. Pour cela, précisons quels objets seront en bijection et décrivons notre démarche. 

\begin{defi}
Soient $H$ un groupe réductif non nécessairement connexe, $A \subset H$ un tore et $L=Z_H(A)$. On appelle sous-groupe de quasi-Levi de $H$, le centralisateur dans $H$ d'un tore contenu dans $H$. Le groupe de Weyl de $L$ dans $H$ est $W_{L}^{H}=N_{H}(A)/Z_H(A)$.
\end{defi}

\begin{rema}
Soient $A \subset H$ un tore contenu dans $H$ et $L=Z_{H}(A)$ un sous-groupe de quasi-Levi de $H$. Alors, $L^{\circ}=Z_{H}(A)^{\circ}=Z_{H^{\circ}}(A)^{\circ}=Z_{H^{\circ}}(A)$ est un sous-groupe de Levi de $H^{\circ}$. Réciproquement, tout sous-groupe de Levi de $H^{\circ}$ est la composante neutre d'un sous-groupe de quasi-Levi de $H$. De plus, le groupe de Weyl de $L^{\circ}$ dans $H^{\circ}$, $W_{L^{\circ}}^{H^{\circ}}=N_{H^{\circ}}(A)/Z_{H^{\circ}}(A)$ est un sous-groupe distingué de $W_{L}^{H}$.
\end{rema}

Décrivons les sous-groupes de quasi-Levi du groupe orthogonal et leurs groupes de Weyl relatifs. 

\renewcommand{\arraystretch}{1.3}
$$\begin{array}{|c|c|c|c|c| l}
\hhline{-----~}
H & L^{\circ} & L & L/L^{\circ} & W_{L}^{H}/W_{L^{\circ}}^{H^{\circ}} &\\
\hhline{=====~}
\O_{2n+1}(\C) & \prod_{i=1}^{k} \GL_{n_i}(\C) \times \SO_{2n'+1}(\C) & \prod_{i=1}^{k} \GL_{n_i}(\C) \times \O_{2n'+1}(\C) & \Z/2\Z & \{1\} &  n_i \geqslant 0, n' \geqslant 0\\ 
\hhline{-----~}
\multirow{2}*{$\O_{2n}(\C)$} & \prod_{i=1}^{k} \GL_{n_i}(\C) \times \SO_{2n'}(\C) & \prod_{i=1}^{k} \GL_{n_i}(\C) \times \O_{2n'}(\C) & \Z/2\Z & \{1\} & n_i \geqslant 0, n' \geqslant 2 \\ \hhline{~----~}
 & \prod_{i=1}^{k} \GL_{n_i}(\C) & \prod_{i=1}^{k} \GL_{n_i}(\C) & \{ 1 \} & \Z/2\Z & n_i \geqslant 0\\ \hhline{-----~}
\end{array} $$\captionof{table}{Sous-groupes de quasi-Levi du groupe orthogonal} 

Nous allons associer à tout couple formé d'une $H$-classe de conjugaison d'un élément unipotent $u \in H^{\circ}$ et d'une représentation irréductible de $A_{H}(u)$, une $H$-classe de conjugaison d'un quadruplet formé d'un sous-groupe de quasi-Levi $L$ de $H$, d'une $L$-classe de conjugaison d'un élément unipotent $v \in L^{\circ}$, d'une représentation irréductible \emph{cuspidale} de $A_{L}(v)$ (voir défition ci-dessous) et d'une représentation irréductible de $W_{L}^{H}$. Pour cela nous allons procéder en trois étapes. Tout d'abord nous étendons la correspondance de Springer généralisée au sous-groupe $H^{L}$ de $H$ tel que $H^{L}/H^{\circ}=L/L^{\circ}$. Ensuite, nous l'étendons au sous-groupe $H^{\mathcal{C}_{u}^{H^{\circ}}}$ de $H$ qui stabilise l'orbite $\mathcal{C}_{u}^{H^{\circ}}$ et enfin à $H$.\\

Le groupe des composantes du groupe orthogonal étant d'ordre $2$, une seule de ces étapes apparaitra dans la section \ref{sectionspringerorth}. Néanmoins, dans la section \ref{sectionspringerorth2} ces trois étapes apparaitront. Avant de commencer, définissons la notion de cuspidalité et introduisons une notation.

\begin{defi}\phantomsection\label{defcusporth}
Soient $L$ un sous-groupe de quasi-Levi de $H$, $v \in L^{\circ}$ un élément unipotent et $\varepsilon \in \Irr(A_{L}(v))$. On dit que $\varepsilon$ est \emph{cuspidale} lorsque la restriction de $\varepsilon$ à $A_{L^{\circ}}(v)$ est somme de représentations cuspidales. On notera $\Irr(A_{L}(v))_{\cusp}$ l'ensemble des (classes de) représentations irréductibles cuspidales de $A_{L}(v)$.
\end{defi}

Notons $$\mathcal{N}_{H}^{+}=\{(\mathcal{C}_{u}^{H},\eta), \,\, u \in H^{\circ} \, \text{ unipotent}, \,\eta \in \Irr(A_{H}(u) \}_{ / H-\rm{conj}},$$ et $$\mathcal{S}_{H}=\{(L,\mathcal{C}_{v}^{L},\varepsilon), \,\, L \text{ quasi-Levi de } H, \,\, v\in L^{\circ} \text{ unipotent et } \varepsilon \in \Irr(A_{L}(v))_{\cusp}\}_{ / H-\rm{conj}}.$$ Pour tout $\mathfrak{t}=[L,\mathcal{C}_{v}^{L},\tau] \in \mathcal{S}_{H}$, notons $W_{\mathfrak{t}}=W_{L}^{H}$.

\subsection{Correspondance de Springer généralisée pour le groupe orthogonal}\phantomsection\label{sectionspringerorth}

Jusqu'à la fin de section nous notons $H=\O_N(\C)$ le groupe orthogonal.\\

Soit $u \in H^{\circ}$ un élément unipotent et supposons que $A_{H}(u) \neq A_{H^{\circ}}(u)$. Dans ce cas, pour tout $s \in H \setminus H^{\circ}, \,\, s\mathcal{C}^{H^{\circ}}_{u}s^{-1}=\mathcal{C}^{H^{\circ}}_{u}$ et puisque $A_{H}(u)$ est un produit de $\Z/2\Z$, on a : $$A_{H}(u)=A_{H^{\circ}}(u) \times (\Z/2\Z).$$ Tout caractère $\eta_{0}$ de $A_{H^{\circ}}(u)$ s'étend deux façons en un caractère de $A_{H}(u)$ : $\eta_{0} \boxtimes 1$ et $\eta_{0} \boxtimes \xi$, où $\xi$ est le caractère non trivial de $\Z/2\Z$.\\

\begin{theo}\phantomsection\label{springerorth}
Il existe une application surjective $$\Psi_{H} : \mathcal{N}_{H}^{+} \longrightarrow \mathcal{S}_{H},$$ qui étend $\Psi_{H^{\circ}}$ et qui induit une décomposition $$\mathcal{N}_{H}^{+}=\bigsqcup_{\mathfrak{t} \in \mathcal{S}_{H}} \mathcal{M}_{\mathfrak{t}},$$ où $\mathcal{M}_{\mathfrak{t}}=\Psi^{-1}(\mathfrak{t})$. De plus, pour tout $\mathfrak{t}\in \mathcal{S}_{H}$, on a une bijection $$\Sigma_{\mathfrak{t}} : \mathcal{M}_{\mathfrak{t}} \longrightarrow \Irr(W_{\mathfrak{t}}).$$ Ainsi, $$\mathcal{N}_{H}^{+} \simeq \bigsqcup_{\mathfrak{t} \in \mathcal{S}_H} \Irr(W_\mathfrak{t}).$$
\end{theo}

\begin{proof}
Soit $(u,\eta) \in \mathcal{N}_{H}^{+}$. Notons $\eta_{0}$ la restriction de $\eta$ à $A_{H^{\circ}}(u)$, $\Psi_{H^{\circ}}(\mathcal{C}_{u}^{H^{\circ}},\eta_{0})=(L^{\circ},\mathcal{C}_{v}^{L^{\circ}},\varepsilon_{0})$ le triplet cuspidal défini par $(\mathcal{C}_{u}^{H^{\circ}},\eta_{0})$ et $\rho_{0} \in \Irr(W_{L^{\circ}}^{H^{\circ}})$ la représentation irréductible du groupe de Weyl associée par la correspondance de Springer généralisée pour $H^{\circ}$.

\begin{enumerate}[label={(\Roman*)}]
\item Supposons ($N$ est impair) ou ($N$ est pair et $L=(\C^{\times})^{\ell} \times \O_{N'}(\C)$ avec $N' \geqslant 4$).\\ Dans ce cas, on a : $$A_{H}(u)=A_{H^{\circ}}(u) \times S, \quad A_{L}(v)=A_{L^{\circ}}(v) \times S' \quad \text{et} \quad W_{L}^{H}=W_{L^{\circ}}^{H^{\circ}},$$ où $S=\langle s \rangle \simeq \Z/2 \Z$ et $S'=\langle s' \rangle \simeq \Z/2 \Z$.
On peut donc écrire $\eta=\eta_0 \boxtimes \chi$, avec $\chi \in \Irr(\Z/2\Z)$ et on pose alors : $$\Psi_{H}(\mathcal{C}_{u}^{H},\eta_{0} \boxtimes \chi)=(L,\mathcal{C}_{v}^{L},\varepsilon_{0} \boxtimes \chi) \quad \text{et} \quad \Sigma_{\mathfrak{t}}(\mathcal{C}_{u}^{H},\eta_{0} \boxtimes \chi)=\rho_{0}.$$

\item Supposons $N$ pair, $L^{\circ}=T=(\C^{\times})^{\ell}$ et que la partition associée à $u$ contient une part paire de multiplicité impaire.\\ Dans ce cas, on a : $$A_{H}(u)=A_{H^{\circ}}(u) \times S, \quad A_{T}(1)=\{1\} \quad \text{et} \quad W_{T}^{H}=W_{T}^{H^{\circ}} \rtimes R,$$ où $S=\langle s \rangle \simeq \Z/2 \Z$ et $R=\langle r \rangle \simeq \Z/2\Z$. Pour tout $h \in N_{H}(T) \setminus N_{H^{\circ}}(T)$, par équivariance de la correspondance de Springer généralisée pour $H^{\circ}$, on a : $$(T,\{1\},1,\rho_{0})=\Sigma(\mathcal{C}_{u}^{H^{\circ}},\eta_{0})=\Sigma \left(h \cdot (\mathcal{C}_{u}^{H^{\circ}},\eta_{0}) \right) =h \cdot \Sigma(\mathcal{C}_{u}^{H^{\circ}},\eta)=(T,\{1\},1,\rho_{0}^{h}).$$ Ainsi,  $\rho_{0}^{r} \simeq \rho_{0}$. Puisque le groupe $\Z/2\Z$ est cyclique, la représentation $\rho_{0}$ s'étend de deux façons en une représentation irréductible de $W_{T}^{H}$ : $\rho_{0} \boxtimes 1$ ou $\rho_{0} \boxtimes \xi$. On pose alors :

$$\Psi_{H}(\mathcal{C}_{u}^{H},\eta_{0} \boxtimes \chi)=(T,\{1\},1) \quad \text{et} \quad \Sigma_{\mathfrak{t}}(\mathcal{C}_{u}^{H},\eta_{0} \boxtimes \chi)=\rho_{0} \boxtimes \chi.$$

\item Supposons $N$ pair et que la partition associée à $u$ ne contient que des parts paires de multiplicités paires.\\ Dans ce cas, $L^{\circ}=T$ et on a : $$A_{H}(u)=A_{H^{\circ}}(u)=\{1\}, \quad A_{T}(1)=\{1\} \quad \text{et} \quad W_{T}^{H}=W_{T}^{H^{\circ}} \rtimes R,$$ avec $R=\langle r \rangle \simeq \Z/2\Z$. La différence avec le cas précédent est qu'ici $\mathcal{C}_{u}^{H}=\mathcal{C}_{u}^{H^{\circ}} \sqcup \mathcal{C}_{u'}^{H^{\circ}}$ est réunion de deux orbites unipotentes de $H^{\circ}$. Pour tout $h \in N_{H}(T) \setminus N_{H^{\circ}}(T)$, par équivariance de la correspondance de Springer généralisée pour $H^{\circ}$, puisque $(\mathcal{C}_{u}^{H^{\circ}},1) \neq (\mathcal{C}_{huh^{-1}}^{H^{\circ}},1)$, on a : $\rho_{0}^{r} \not \simeq \rho_{0}$. Ceci implique que $\Ind_{W_{T}^{H^{\circ}}}^{W_{T}^{H}} (\rho_{0})$ est irréductible. On pose alors :
$$\Psi_{H}(\mathcal{C}_{u}^{H},1)=(T,\{1\},1) \quad \text{et} \quad \Sigma_{\mathfrak{t}}(\mathcal{C}_{u}^{H},1)=\Ind_{W_{T}^{H^{\circ}}}^{W_{T}^{H}} (\rho_{0}).$$
\end{enumerate}
\end{proof}

\begin{prop}
Soient $(\mathcal{C}_{u}^{H},\eta) \in \mathcal{N}_{H}^{+}$ et $(L,\mathcal{C}_{v}^{L},\varepsilon)=\Psi_{H}(\mathcal{C}_{u}^{H},\eta)$. On a des morphismes $Z_{H} \rightarrow A_{H}(u)$ et $Z_{L} \rightarrow A_{L}(v)$. Alors $\eta(-1)=\varepsilon(-1)$ (où l'on voit $-1$ à travers les morphismes précédents).
\end{prop}

\begin{proof}
Dans le (I), on peut écrire $\eta=\eta_{0} \boxtimes \chi$ et $\varepsilon=\varepsilon_0 \boxtimes \chi$, où $\eta_{0} \in \Irr(A_{H^{\circ}}(u))$ et $\varepsilon_0 \in \Irr(A_{L^{\circ}}(v))$.\\
Supposons $N$ pair. Dans ce cas, $-1 \in H^{\circ}$.  D'après, \cite[5.23]{Lusztig:1995}, $\eta_{0}(-1)=\varepsilon_0(-1)$, d'où $\eta(-1)=\varepsilon(-1)$.\\
Supposons $N$ impair. Dans ce cas, $-1 \in A_{H}(u) \setminus A_{H^{\circ}}(u)$. Donc, $\eta(-1)=\chi(-1)=\varepsilon(-1)$. Dans les cas (II) et (III), on conclue de la même façon.
\end{proof}

\begin{rema}
On remarque d'après la forme des représentations cuspidales de $A_{L}(v)$, que pour tout $w \in W_{L}^{H}$ et toute représentation irréductible cuspidale $\varepsilon$ de $A_{L}(v)$, $\varepsilon^{w} \simeq \varepsilon$.
\end{rema}

À présent, fixons un triplet cuspidal $\mathfrak{t}=(L,\mathcal{C}_{v}^{L},\varepsilon) \in \mathcal{S}_{H}$ et notons $\mathfrak{t}^{\circ}=(L^{\circ},\mathcal{C}_{v}^{L^{\circ}},\varepsilon_{0}) \in \mathcal{S}_{H^{\circ}}$ le triplet cuspidal obtenu par restriction. On considère l'algèbre de Hecke graduée $\mathbb{H}_{\mu_{\mathfrak{t}^{\circ}}}$ associé à $\mathfrak{t}^{\circ}$ définie et étudiée par Lusztig dans \cite{Lusztig:1988}, \cite{Lusztig:1995a} et \cite{Lusztig:2002a}.\\

Soient $(s,r_{0}) \in \mathfrak{h} \oplus \C$ un élément semi-simple et $x \in \mathfrak{h}$ un élément nilpotent tel que $[s,x]=2r_0 x$. Rappelons que l'on note $\Irr(A_H(s,x))_{\mathfrak{t}}$ l'ensemble des représentations irréductibles de $A_{H}(s,x)$ qui apparaissent dans la restriction d'une représentation irréductible $\widetilde{\eta} \in A_{H}(x)$ telle que $\Psi_{H}(\mathcal{C}_{x}^{H},\widetilde{\eta})=\mathfrak{t}$. On s'intéresse aux couples $(x,\eta)$ où $\eta \in \Irr(A_H(s,x))_{\mathfrak{t}}$. On a vu que le théorème de Jacobson-Morozov montre l'existence d'un élément semi-simple $s_{0} \in \mathfrak{h}$ tel que $[s_0,s]=0$, $[s_0,x]=2r_0x$ et $A_{H}(s,x)=A_{H}(s s_{0}^{-1},x)$. Par conséquent, $A_{H}(s,x)=A_{Z_{H}(s s_{0}^{-1})}(x)$ et $Z_{H}(s s_{0}^{-1})$ est un sous-groupe de quasi-Levi de $H$. D'après la compatibilité de la correspondance de Springer généralisée avec l'induction parabolique (voir \cite[\textsection 8]{Lusztig:1984}), il s'ensuit que si $\eta \in \Irr(A_{H}(s,x))_{\mathfrak{t}}$, si et seulement si, $\Psi_{Z_H(s s_{0}^{-1})}(\mathcal{C}_{x}^{Z_H(s s_{0}^{-1})},\eta)=(L,\mathcal{C}_{v}^{L},\varepsilon)$.\\

\begin{theo}\phantomsection\label{thm:paramfibremodsimple1}
Soient $\mathfrak{t}=(L,\mathcal{C}_{v}^{L},\varepsilon) \in \mathcal{S}_{H}$ un triplet cuspidal de $H$, $W_{L}^{H}=W_{L^{\circ}}^{H^{\circ}} \rtimes R$ le groupe de Weyl relatif de $L$ dans $H$.
L'ensemble des classes de $\mathbb{H}_{\mu_{\mathfrak{t}}} \rtimes \C[R]$-module simple de caractère central $(s,r_0)$ est en bijection avec $$\mathcal{M}_{(s,r_{0})}=\{(x,\eta) \mid x \in \mathfrak{h}, \,\, [s,x]=2r_0 x, \eta \in \Irr(A_{H}(s,x))_{\mathfrak{t}} \}$$
\end{theo}

\begin{proof}
Soit $(x,\eta) \in A_{H}(s,x)_{\mathfrak{t}}$. Reprenons les cas (I), (II) et (III) décrit précédemment.
\begin{enumerate}[label={(\Roman*)}]
\item Dans ce cas, on a donc : $$A_{H}(s,x)=A_{H^{\circ}}(s,x) \times (\Z/2\Z), \quad A_{L}(v)=A_{L^{\circ}}(v) \times (\Z/2\Z) \quad \text{et} \quad W_{L}^{H}=W_{L^{\circ}}^{H^{\circ}}.$$
On peut donc écrire $\varepsilon=\varepsilon_{0}\boxtimes \chi$, $\eta=\eta_{0} \boxtimes \chi$ et $M$ le $\mathbb{H}_{\mu_{\mathfrak{t}^{\circ}}}$-module irréductible associé à $(x,\eta_0)$. 
\item Dans ce cas, on a donc : $$A_{H}(s,x)=A_{H^{\circ}}(s,x) \times (\Z/2\Z), \quad A_{T}(1)=\{1\} \quad \text{et} \quad W_{T}^{H}=W_{T}^{H^{\circ}} \rtimes R,$$ où $R \simeq \Z/2\Z$. Considérons l'algèbre de Hecke graduée étendue $\mathbb{H}_{\mu_{\mathfrak{t}^{\circ}}} \rtimes \C[R]$. Soient $\overline{M}(s,r_{0},x,\eta_{0})$ le $\mathbb{H}_{\mu_{\mathfrak{t}^{\circ}}}$-module irréductible associé à $(x,\eta_0)$ et $r \in R$. Par équivariance, on a alors : $$r \cdot M(s,r_{0},x,\eta_{0})=M(\Ad(r)s,r_{0},\Ad(r)x,\eta_{0}^{r}).$$ On a vu précédemment que dans le (II), $\eta_{0}^{r} = \eta_{0}$, $\mathcal{C}_{\Ad(r)x}^{H^{\circ}}=\mathcal{C}_{x}^{H^{\circ}}$. Il s'ensuit que $r \cdot \overline{M}(s,r_{0},x,\eta_{0}) \simeq \overline{M}(s,r_{0},x,\eta_{0})$. On associe $(x,\eta)$ le $\mathbb{H}_{\mu_{\mathfrak{t}^{\circ}}} \rtimes \C[R]$-module irréductible $\overline{M}(s,r_{0},x,\eta_{0}) \boxtimes \C_{\chi}$.
\item Dans ce cas, on a donc : $$A_{H}(s,x)=A_{H^{\circ}}(s,x)=\{1\}, \quad A_{T}(1)=\{1\} \quad \text{et} \quad W_{T}^{H}=W_{T}^{H^{\circ}} \rtimes (\Z/2\Z).$$ Soient $\overline{M}(s,r_{0},x,1)$ le $\mathbb{H}_{\mu_{\mathfrak{t}^{\circ}}}$-module irréductible associé à $(x,1)$ et $r \in R$. Comme précédemment, puisque $\mathcal{C}_{x}^{H}$ est réunion de deux orbites nilpotentes de $H^{\circ}$, par équivarience on a $r \cdot \overline{M}(s,r_{0},x,1) \not \simeq \overline{M}(s,r_{0},x,1)$. On associe $(x,1)$ le $\mathbb{H}_{\mu_{\mathfrak{t}^{\circ}}} \rtimes \C[R]$-module irréductible $\Ind_{\mathbb{H}_{\mu_{\mathfrak{t}^{\circ}}}}^{\mathbb{H}_{\mu_{\mathfrak{t}^{\circ}}} \rtimes \C[R]} \overline{M}(s,r_{0},x,1)$.
\end{enumerate}
\end{proof}

\subsection{Un sous-groupe d'indice deux d'un produit de groupes orthogonaux}\phantomsection\label{sectionspringerorth2}

Soient $r \geqslant 2$ un entier, $m_1,\ldots,m_r \geqslant 1$ des entiers. Considérons $H=\prod_{i=1}^{r} \O_{m_i}(\C)$ et $$\widetilde{H}=\left\{(x_i) \in H \mid \prod_{i=1}^{r} \det(x_i)=1\right\}.$$ Soient $u \in H^{\circ}$ un élément unipotent et $\eta_{0}$ une représentation irréductible de $A_{H^{\circ}}(u)$. On a des décompositions : $$u=(u_i) \in \prod_{i=1}^{r} \SO_{m_i}(\C), \quad A_{H^{\circ}}(u)=\prod_{i=1}^{r} A_{\SO_{m_i}(\C)}(u_i), \quad \text{et} \quad \eta_{0} = \eta_{1} \boxtimes \ldots \boxtimes \eta_{r},$$ avec pour tout $i \in \llbracket 1,r \rrbracket, \,\, \eta_{i}$ une représentation irréductible de $A_{\SO_{m_i}}(u_i)$. On peut supposer avoir arrangé les indices de la façon suivante : \begin{itemize}
\item pour tout $i \in \llbracket 1,p \rrbracket$, $m_i$ est impair ou $(u_i,\sigma_i)$ est associé par la correspondance de Springer généralisée à un sous-groupe de Levi de la forme $(\C^{\times})^{\ell_i} \times \SO_{m_i'}(\C)$ avec $m_i' \geqslant 4$ ;
\item pour tout $i \in \llbracket p+1, q \rrbracket$, la partition associé à $u_i$ contient une part paire de multiplicité impaire ;
\item pour tout $i \in \llbracket q+1,r \rrbracket$, la partition associé à $u_i$ ne contient que des parts paires de multiplicités paires.
\end{itemize}

Pour tout $i \in \llbracket 1, r \rrbracket$, notons $H_{i}=\O_{m_i}$, et ajoutons un indice $i$ aux objets que nous avons dans la section précédente ($L_i$ le sous-groupe de quasi-Levi de $H_i$ associé à $(\mathcal{C}_{u_i}^{H_{i}},\eta_i \boxtimes \chi_i)$, $s_i$, $s_i'$, etc).

Si $p=0$, notons \begin{align*}
C_{\widetilde{L}}=\{1\},\,\, C_{\mathcal{O}}=\left\langle s_{p+1}s_{p+2}, \ldots, s_{q-1} s_q \right\rangle ,\,\, C_{I}=\left\langle s_q s_{q+1}, \ldots, s_{r-1}s_{r} \right\rangle, \\
C_{\widetilde{L}}'=\{1\},\,\, C_{\mathcal{O}}'=\left\langle s_{p+1}' s_{p+2}', \ldots, s_{q-1} ' s_q '\right\rangle ,\,\, C_{I}'=\left\langle s_q ' s_{q+1}', \ldots, s_{r-1}'s_{r}' \right\rangle.
\end{align*}

Si $p \geqslant 1$, notons \begin{align*}
C_{\widetilde{L}}=\left\langle s_1 s_2, \ldots, s_{p-1} s_{p} \right\rangle,\,\, C_{\mathcal{O}}=\left\langle s_p s_{p+1}, s_p s_{p+2} \ldots, s_{p} s_q \right\rangle ,\,\, C_{I}=\left\langle s_p s_{q+1}, \ldots, s_{p}s_{r} \right\rangle, \\
C_{\widetilde{L}}'=\left\langle s_1' s_2', \ldots, s_{p-1}' s_{p}' \right\rangle,\,\, C_{\mathcal{O}}'=\left\langle s_p' s_{p+1}', s_{p}' s_{p+2}',\ldots, s_{p} ' s_q '\right\rangle ,\,\, C_{I}'=\left\langle s_p ' s_{q+1}', \ldots, s_{p}'s_{r}' \right\rangle.
\end{align*}

Et dans chacun de ces deux cas, définissons les sous-groupes $H^{\circ} \subseteq \widetilde{H}^{\widetilde{L}} \subseteq \widetilde{H}^{\mathcal{O}} \subseteq \widetilde{H}$ par :
\begin{align*}
\widetilde{H}^{\widetilde{L}}=H^{\circ} \rtimes C_{\widetilde{L}}, \,\, \widetilde{H}^{\mathcal{O}}=H^{\circ} \rtimes (C_{\widetilde{L}} \times C_{\mathcal{O}}), \,\, \widetilde{H} =H^{\circ} \rtimes (C_{\widetilde{L}} \times C_{\mathcal{O}} \times C_{I}), \,\, \widetilde{L}=\widetilde{H} \cap \prod_{i=1}^{r} L_i. \end{align*}

\begin{rema}
La façon dont on a défini ces groupe suppose que l'on a $p\geqslant 2, \,\, q-p\geqslant 2,\,\, r-q \geqslant 2$ pour $C_{\widetilde{L}},\,\, C_{\mathcal{O}}, \,\, C_{I}$ respectivement. Si ce n'est pas le cas, alors le groupe considéré est trivial.
\end{rema}

Afin d'étendre la correspondance de Springer généralisée à $\widetilde{H}$, nous avons expliqué que l'on va l'étendre par étapes. À chaque étape, des objets (classes unipotentes, groupe de Weyl, groupe des composantes du centralisateur, etc) peuvent changer ou rester identique. Listons ci-dessous à gauche d'une part, les objets qui sont préservés et à droite d'autre part, les objets qui changent.

\begin{enumerate}[label=(\arabic*)]
\item $\mathcal{C}_{u}^{\widetilde{H}^{\widetilde{L}}}=\mathcal{C}_{u}^{H^{\circ}}$, $\mathcal{C}_{u}^{\widetilde{L}}=\mathcal{C}_{u}^{L^{\circ}}$, $W_{\widetilde{L}}^{\widetilde{H}^{\widetilde{L}}}=W_{L^{\circ}}^{H^{\circ}}$ \hfill $A_{\widetilde{H}^{\widetilde{L}}}(u)=A_{H^{\circ}}(u) \times  C_{\widetilde{L}}$, $A_{\widetilde{L}}(v)=A_{L^{\circ}}(v) \times C_{\widetilde{L}}'$
\item $\mathcal{C}_{u}^{\widetilde{H}^{\mathcal{O}}}=\mathcal{C}_{u}^{\widetilde{H}^{\widetilde{L}}}$, $\mathcal{C}_{u}^{\widetilde{L}}=\mathcal{C}_{u}^{L^{\circ}}$ \hfill $A_{\widetilde{H}^{\mathcal{O}}}(u)=A_{\widetilde{H}^{\widetilde{L}}}(u) \times C_{\mathcal{O}}$, $W_{\widetilde{L}}^{\widetilde{H}^{\mathcal{O}}}= W_{\widetilde{L}}^{\widetilde{H}^{\widetilde{L}}} \rtimes C_{\mathcal{O}}'$.
\item $A_{\widetilde{H}}(u)=A_{\widetilde{H}^{\mathcal{O}}}(u)$ \hfill $\mathcal{C}_{u}^{\widetilde{H}}=\sqcup_{h \in C_I} \mathcal{C}_{huh^{-1}}^{\widetilde{H}^{\mathcal{O}}}$, $W_{\widetilde{L}}^{\widetilde{H}}=W_{\widetilde{L}}^{\widetilde{H}^{\mathcal{O}}} \rtimes C_I$.
\end{enumerate}

\begin{theo}\phantomsection\label{springerorth}
Il existe une application surjective $$\Psi_{\widetilde{H}} : \mathcal{N}_{\widetilde{H}}^{+} \longrightarrow \mathcal{S}_{\widetilde{H}},$$ qui étend $\Psi_{H^{\circ}}$ et qui induit une décomposition $$\mathcal{N}_{\widetilde{H}}^{+}=\bigsqcup_{\mathfrak{t} \in \mathcal{S}_{\widetilde{H}}} \mathcal{M}_{\mathfrak{t}},$$ où $\mathcal{M}_{\mathfrak{t}}=\Psi^{-1}(\mathfrak{t})$. De plus, pour tout $\mathfrak{t}\in \mathcal{S}_{\widetilde{H}}$, on a une bijection $$\Sigma_{\mathfrak{t}} : \mathcal{M}_{\mathfrak{t}} \longrightarrow \Irr(W_{\mathfrak{t}}).$$ Ainsi, $$\mathcal{N}_{\widetilde{H}}^{+} \simeq \bigsqcup_{\mathfrak{t} \in \mathcal{S}_{\widetilde{H}}} \Irr(W_\mathfrak{t}).$$
\end{theo}

\begin{proof}
La correspondance de Springer généralisée pour $H^{\circ}$ associe à $\left(\mathcal{C}^{H^{\circ}}_{u},\eta_{0} \right)$ le quadruplet $\left( L^{\circ},\mathcal{C}^{L^{\circ}}_{v},\varepsilon_{0},\rho_{0} \right) $.

On construit la correspondance de Springer généralisée pour $\widetilde{H}$ par étapes de la façon suivante :
\begin{enumerate}[label=(\arabic*)]
\item on l'étend à $\widetilde{H}^{\widetilde{L}}$ en associant à $\left(\mathcal{C}^{H^{\circ}}_{u},\eta_{0} \boxtimes \chi_{\widetilde{L}}\right)$ le quadruplet $\left( \widetilde{L},\mathcal{C}^{\widetilde{L}}_{v},\varepsilon_{0} \boxtimes \chi_{\widetilde{L}},\rho_{0} \right)$ ;
\item on l'étend à $\widetilde{H}^{\mathcal{O}} $ en associant à $\left(\mathcal{C}^{H^{\circ}}_{u},\eta_{0} \boxtimes \chi_{\widetilde{L}} \boxtimes \chi_{\mathcal{O}} \right)$ le quadruplet $\left( \widetilde{L},\mathcal{C}^{\widetilde{L}}_{v},\varepsilon_{0} \boxtimes \chi_{\widetilde{L}},\rho_{0} \boxtimes \chi_{\mathcal{O}} \right)$ ;
\item on l'étendu à $\widetilde{H}$ en associant à $\left( \mathcal{C}^{\widetilde{H}}_{u},\eta_{0} \boxtimes \chi_{\widetilde{L}} \boxtimes \chi_{\mathcal{O}} \right)$ le quadruplet $\left( \widetilde{L},\mathcal{C}^{\widetilde{L}}_{v},\varepsilon_{0} \boxtimes \chi_{\widetilde{L}}, \Ind_{W_{\widetilde{L}}^{\widetilde{H}^{\mathcal{O}}}}^{W_{\widetilde{L}}^{\widetilde{H}}} \left( \rho_{0} \boxtimes \chi_{\mathcal{O}} \right) \right)$.
\end{enumerate}
\end{proof}

À présent, fixons un triplet cuspidal $\mathfrak{t}=(L,\mathcal{C}_{v}^{L},\varepsilon) \in \mathcal{S}_{H}$ et notons $\mathfrak{t}^{\circ}=(L^{\circ},\mathcal{C}_{v}^{L^{\circ}},\varepsilon_{0}) \in \mathcal{S}_{H^{\circ}}$ le triplet cuspidal obtenu par restriction. On considère l'algèbre de Hecke graduée $\mathbb{H}_{\mu_{\mathfrak{t}^{\circ}}}$ associé à $\mathfrak{t}^{\circ}$ définie et étudiée par Lusztig dans \cite{Lusztig:1988}, \cite{Lusztig:1995a} et \cite{Lusztig:2002a}.\\

Soient $(s,r_{0}) \in \mathfrak{h} \oplus \C$ un élément semi-simple et $x \in \mathfrak{h}$ un élément nilpotent tel que $[s,x]=2r_0 x$. Rappelons que l'on note $\Irr(A_H(s,x))_{\mathfrak{t}}$ l'ensemble des représentations irréductibles de $A_{H}(s,x)$ qui apparaissent dans la restriction d'une représentation irréductible $\widetilde{\eta} \in A_{H}(x)$ telle que $\Psi_{H}(\mathcal{C}_{x}^{H},\widetilde{\eta})=\mathfrak{t}$. On s'intéresse aux couples $(x,\eta)$ où $\eta \in \Irr(A_H(s,x))_{\mathfrak{t}}$. On a vu que le théorème de Jacobson-Morozov montre l'existence d'un élément semi-simple $s_{0} \in \mathfrak{h}$ tel que $[s_0,s]=0$, $[s_0,x]=2r_0x$ et $A_{H}(s,x)=A_{H}(s s_{0}^{-1},x)$. Par conséquent, $A_{H}(s,x)=A_{Z_{H}(s s_{0}^{-1})}(x)$ et $Z_{H}(s s_{0}^{-1})$ est un sous-groupe de quasi-Levi de $H$. D'après la compatibilité de la correspondance de Springer généralisée avec l'induction parabolique (voir \cite[\textsection 8]{Lusztig:1984}), il s'ensuit que si $\eta \in \Irr(A_{H}(s,x))_{\mathfrak{t}}$, si et seulement si, $\Psi_{Z_H(s s_{0}^{-1})}(\mathcal{C}_{x}^{Z_H(s s_{0}^{-1})},\eta)=(L,\mathcal{C}_{v}^{L},\varepsilon)$.\\

\begin{theo}\phantomsection\label{thm:paramfibremodsimple2}
Soient $\mathfrak{t}=(\widetilde{L},\mathcal{C}_{v}^{\widetilde{L}},\varepsilon) \in \mathcal{S}_{\widetilde{H}}$ un triplet cuspidal de $\widetilde{H}$, $W_{\widetilde{L}}^{\widetilde{H}}=W_{L^{\circ}}^{H^{\circ}} \rtimes R$ le groupe de Weyl relatif de $\widetilde{L}$ dans $\widetilde{H}$.
L'ensemble des classes de $\mathbb{H}_{\mu_{\mathfrak{t}}} \rtimes \C[R]$-module simple de caractère central $(s,r_0)$ est en bijection avec $$\mathcal{M}_{(s,r_{0})}=\{(x,\eta) \mid x \in \mathfrak{h}, \,\, [s,x]=2r_0 x, \eta \in \Irr(A_{\widetilde{H}}(s,x))_{\mathfrak{t}} \}$$
\end{theo}

\begin{proof}
Par la théorie de Clifford, les $\mathbb{H}_{\mu_{\mathfrak{t}}} \rtimes \C[R]$-module simples sont obtenus de la façon suivante : soit $M_{0}$ un $\mathbb{H}_{\mu_{\mathfrak{t}}}$-module simple. On note $R^{M}=\{r \in R \mid r \cdot M_{0} \simeq M_{0}\}$. On obtient alors un $\mathbb{H}_{\mu_{\mathfrak{t}}} \rtimes \C[R^{M}]$-module  $M_{0} \boxtimes V$. Enfin, $\Ind_{\mathbb{H}_{\mu_{\mathfrak{t}}} \rtimes \C[R^{M}]}^{\mathbb{H}_{\mu_{\mathfrak{t}}} \rtimes \C[R]} (M_{0} \boxtimes V)$ est irréductible et tous les modules simples sont obtenus de cette façon.\\

Soit $M_{0}$ le $\mathbb{H}_{\mu_{\mathfrak{t}}}$-module simple correspondant à $(x,\eta_0)$. Alors $R^{M_{0}}=C_{\mathcal{O}}$. Par conséquent, on note : $M_{0} \boxtimes \C_{\chi_{\mathcal{O}}}$. Et enfin, $\Ind_{\mathbb{H}_{\mu_{\mathfrak{t}}} \rtimes \C[R^{M}]}^{\mathbb{H}_{\mu_{\mathfrak{t}}} \rtimes \C[R]} (M_{0} \boxtimes \C_{\chi_{\mathcal{O}}})$.
\end{proof}

\printbibliography

\end{document}